\documentclass[reqno,english]{amsart}
\usepackage{amsfonts,amsmath,latexsym,verbatim,amscd,mathrsfs,color,array}
\usepackage[colorlinks=true]{hyperref}
\usepackage{amsmath,amssymb,amsthm,amsfonts,graphicx,color}
\usepackage[hmargin=3cm, vmargin=3cm]{geometry}
\usepackage{bbm}
\usepackage{amssymb}
\usepackage{pdfsync}
\usepackage{epstopdf}
\usepackage{cite}
\usepackage{graphicx}
\usepackage{multirow}
\usepackage[font=bf,aboveskip=15pt]{caption}
\usepackage[toc,page]{appendix}
\usepackage[colorlinks=true]{hyperref}
\newcommand\al{\alpha}

\newcommand\de{\delta}

\newcommand\la{\lambda}
\newcommand{\onn}{{\quad\mbox{on } }}
\newcommand{\inn}{{\quad\mbox{in } }}

\newcommand\De{\Delta}

\newcommand{\R}{\mathbb{R}}

\renewcommand{\Re} {\mathop{\mathrm{Re}}}
\renewcommand{\Im} {\mathop{\mathrm{Im}}}
\renewcommand{\div}{\mathop{\rm div}}
\newcommand{\curl} {\mathop{\rm curl}}

\newcommand{\pd}{\partial}

\newcommand{\na}{\nabla}
\newcommand{\ttt}{\tilde}

\newcommand{\RR}{\mathcal{R}}
\newcommand{\EE}{\mathcal{E}}
\newcommand{\KK}{\mathcal{K}}
\newcommand{\EQ}[1]{\begin{equation}\begin{split} #1 \end{split}\end{equation}}
\newcommand{\EQN}[1]{\begin{equation*}\begin{split} #1 \end{split}\end{equation*}}

\newtheorem{theorem}{Theorem}[section]
\newtheorem{lemma}{Lemma}[section]

\newtheorem{prop}{Proposition}[section]
\newtheorem{remark}{Remark}[section]

\numberwithin{equation}{section}

\begin{document}
\title[Finite time blow-up for the nematic liquid crystal flow in dimension two]{Finite time blow-up for the nematic liquid crystal flow in dimension two}

\author[C. Lai]{Chen-Chih Lai}
\address{\noindent
Department of Mathematics,
University of British Columbia, Vancouver, B.C., V6T 1Z2, Canada}
\email{chenchih@math.ubc.ca}

\author[F. Lin]{Fanghua Lin}
\address{\noindent
Courant Institute of Mathematical Sciences,
New York University, NY 10012, USA}
\email{linf@cims.nyu.edu}

\author[C. Wang]{Changyou Wang}
\address{\noindent
Department of Mathematics,
Purdue University, West Lafayette, IN 47907, USA}
\email{wang2482@purdue.edu}

\author[J. Wei]{Juncheng Wei}
\address{\noindent
Department of Mathematics,
University of British Columbia, Vancouver, B.C., V6T 1Z2, Canada}
\email{jcwei@math.ubc.ca}

\author[Y. Zhou]{Yifu Zhou}
\address{\noindent
Department of Mathematics,
University of British Columbia, Vancouver, B.C., V6T 1Z2, Canada}
\email{yfzhou@math.ubc.ca}

\begin{abstract}
We consider the initial-boundary value problem of a simplified nematic liquid crystal flow in a bounded, smooth domain $\Omega \subset \mathbb R^2$.
Given any $k$ distinct points in the domain, we develop a new {\em inner--outer gluing method} to construct solutions which blow up exactly at those $k$ points as $t$ goes to a finite time $T$.
Moreover, we obtain a precise description of the blow-up.
\end{abstract}
\maketitle


\section{Introduction}

\medskip

In this paper, we consider the following initial-boundary value problem of nematic liquid crystal flow in a bounded, smooth domain $\Omega$ in $\R^2$, and $T>0$
\begin{equation}\label{LCF}
\begin{cases}
\partial_t v+v\cdot \nabla v+\nabla P=\Delta v -\epsilon_0 \nabla\cdot\left(\nabla u\odot \nabla u-\frac12|\nabla u|^2\mathbb I_2\right)~&\mbox{ in }~\Omega\times(0,T),\\
\nabla\cdot v=0~&\mbox{ in }~\Omega\times(0,T),\\
\partial_t u+ v\cdot\nabla u=\Delta u+|\nabla u|^2u~&\mbox{ in }~\Omega\times(0,T),
\end{cases}
\end{equation}
with initial condition
\EQ{\label{ICs}
(v,u)\big|_{t=0}=(v_0,u_0)~&\mbox{ in }~\Omega,
}
and boundary condition
\EQ{\label{BCs}
v=0~&\mbox{ on }~\partial\Omega\times(0,T),\\
u=u_0~&\mbox{ on }~\partial\Omega\times(0,T),
}
where $v:\Omega\times[0,T)\to\R^2$ is the fluid velocity, $P:\Omega\times[0,T)\to\R$ is the fluid pressure, $u:\Omega\times[0,T)\to\mathbb{S}^2$ stands for the orientation field of nematic liquid crystal material, $\na\cdot$ denotes the divergence operator, $\na u\odot\na u$ denotes the $2\times2$ matrix given by $(\na u\odot\na u)_{ij}=\na_i u\cdot\na_j u$, and $\mathbb I_2$ is the identity matrix on $\R^2$. The parameter $\epsilon_0>0$  represents the competition between kinetic energy and potential energy. $(v_0, u_0):\Omega\to\R^2\times\mathbb S^2$ is a given
initial data.

\medskip

The system \eqref{LCF} can be viewed as a coupling between the incompressible Navier--Stokes equation and the equation of heat flow of harmonic maps. Many important contributions have been made on the studies to the incompressible Navier--Stokes equation and the equation of heat flow of harmonic maps. For the incompressible Navier--Stokes equation, the existence of global weak solutions to the initial value problem has been established by Leray \cite{Leray1934Acta} and Hopf \cite{Hopf1951Acta}. For comprehensive results regarding the Navier--Stokes equation, we refer the interested reader to Temam \cite{Temam2001book}, Lions \cite{Lions1996book}, Lemari\'{e}-Rieusset \cite{RieussetNS}, Galdi \cite{Galdi}, Seregin \cite{Seregin}, Tsai \cite{Tsai} and the references therein. The fundamental solution of the Stokes system, which is a linearized Navier--Stokes equation, was first established by Solonnikov in \cite{Solo1964}, together with estimates of weak solutions to the Cauchy problem. Solonnikov also derived several estimates of the initial-boundary value problem of the Stokes system in \cite{Solo2001, Solo2002, Solo2003}.  For
the heat flow of harmonic maps, Struwe \cite{Struwe1985CMH} and Chang \cite{Chang1989AnnInstHPANL} established the existence of a unique global weak solution in dimension two, which has at most finitely many singular points. In higher dimensions, the existence of a global weak solution has been proved by Chen and Struwe in \cite{ChenStruwe1989MathZ}, and Chen and Lin in \cite{ChenLin1993CAG}. Examples of finite time blow-up solutions have been constructed by Coron and Ghidaglia in \cite{CoronGhidaglia1989}, and Chen and Ding in \cite{ChenDing1990InventMath} for $n\ge3$. See also \cite{Grotowski1991manumath,Grotowski1993CVPDE} for more finite-time singularity results in dimension three. In dimension two, Chang, Ding and Ye \cite{CDY92JDG} constructed the first example of finite time singularities, which is a $1$-corotational solution in a disk with profile
\[u(x,t)=W\left(\frac{x}{\la(t)}\right)+O(1),\]
where $W$ is the least energy harmonic map
\begin{equation*}
W(y)=\frac{1}{1+|y|^2}\begin{bmatrix}
2y\\
|y|^2-1\\
\end{bmatrix},~y\in\R^2,
\end{equation*}
$O(1)$ is bounded in $H^1$-norm, and $0<\la(t)\to0$ as $t\to T$. Angenent, Hulshof and Matano \cite{AHM2009SIMA} obtained an estimation of the blow-up rate as $\la(t) = o(T-t)$. Using matched asymptotics formal analysis, van den Berg, Hulshof and King \cite{VHK2004SIMA} showed that this rate should be given by
\[\la(t)\sim\kappa\frac{T-t}{|\log(T-t)|^2}\] for some $\kappa>0$. Rapha\"el and Schweyer succeeded in constructing an entire 1-corotational solution with this blow-up rate rigorously \cite{RS2013CPAM}. Recently, Davila, del Pino and Wei \cite{17HMF} constructed non-symmetric finite time blow-up at multiple points and studied its stability by using the {\em inner--outer gluing method}. More precisely, for any given finite set of points in $\Omega$, they constructed solution blowing up exactly at those points simultaneously under suitable initial and boundary conditions. In another aspect, for higher-degree corotational harmonic map heat flow, global existence and blow-up have been investigated in a series of works \cite{GKT08DUKE,GGT09JDE,GNT10CMP,gustafson2017global} and the references therein. For other bubbling phenomena and regularity results of the heat flow of harmonic maps, we refer the readers to the book \cite{LWbook} by Lin and Wang.

\medskip

The nematic liquid crystal flow \eqref{LCF} was first proposed by Lin in \cite{L1989CPAM}, and it is a simplified version of the Ericksen--Leslie model for the hydrodynamics flow of nematic liquid crystal molecular considered by Ericksen \cite{Ericksen1962ARMA} and Leslie \cite{Leslie1968ARMA}. The existence and uniqueness of solutions to \eqref{LCF} has been extensively studied. In a pioneering paper \cite{LL1995CPAM}, Lin and Liu considered the Leslie system of variable length, and established the global existence of weak and classical solutions in dimensions two and three. They also built up the partial regularity theorem for suitable weak solutions of \eqref{LCF} in \cite{LL1996DCDS}, similar to those for the Navier--Stokes equation established by Caffarelli--Kohn--Nirenberg in \cite{CKN1982CPAM}. Later on, Lin, Lin and Wang \cite{LLW2010ARMA} proved the global existence of Leray--Hopf type weak solutions of \eqref{LCF} (see also Hong \cite{Hong2011CVPDE}, Hong--Xin \cite{HongXin2012AdvMath}, Xu--Zhang \cite{XZ2012JDE}, Huang--Lin--Wang \cite{HuangLinWang2014CMP}, Lei--Li--Zhang \cite{LeiLiZhang2014ProcAMS}, Wang--Wang \cite{WangWang2014CVPDE} for relevant results in dimension two). The uniqueness of such weak solution was also proved by Lin and Wang in \cite{LW2010CAM}. In dimension three, Ding and Wen \cite{WenDing2011NonlAnalRWA} proved the existence of a unique local strong solution (see also \cite{MaGongLi2014NonlAnal} for relevant results in dimension three). Subsequently, in the case of smooth and bounded domain in dimension three, Lin and Wang \cite{LW2016CPAM} proved the global existence of weak solutions satisfying the global energy inequality under the assumption that the initial orientation field $u_0(\Omega)\subset \mathbb{S}^2_+$. Blow up criteria have also been established. For instance, Huang and Wang \cite{HW2012CPDE} proposed a criterion for finite time singularity of strong solutions in dimensions two and three. Moreover, Huang, Wang and Wen \cite{HWW2012ARMA} established a blow up criterion for compressible nematic liquid crystal flows in dimension three. Recently, Chen and Yu \cite{CY2017ARMA} constructed global $m$-equivariant solutions in $\R^2$ that the orientation field blows up logarithmically as $t\to+\infty$. See also Lin--Wang \cite{LW2014RSTA} for a survey of some important developments of mathematical studies of nematic liquid crystals.

\medskip

The main concern of this paper is the existence of solutions to the nematic liquid crystal flow \eqref{LCF}, that develop finite time singularities. In the three dimensional case, Huang, Lin, Liu and Wang \cite{HLLW2016ARMA} have constructed two examples of finite time singularity of \eqref{LCF}. The first example is an axisymmetric finite time blow-up solution constructed in a cylindrical domain. (As remarked in [Remark 1.2(c), \cite{HLLW2016ARMA}], this blow-up example does not satisfy the no-slip boundary condition.)  The second example is constructed in a ball for any generic initial data that has small enough energy, and $u_0$ has a non-trivial topology.

\medskip

In this paper, we consider the two-dimensional nematic liquid crystal flow \eqref{LCF}, where the velocity field satisfies no-slip boundary condition, i.e., $v=0$ on $\partial\Omega$. Using the {\em inner--outer gluing method}, we construct a solution $(v,u)$ to problem \eqref{LCF} exhibiting finite time singularity as $t\to T$. To carry out a fixed point argument, we assume that the parameter $\epsilon_0$ is sufficiently small. Under this assumption, \eqref{LCF} is approximately decoupled into an incompressible Navier--Stokes equation and a transported harmonic map heat flow. It is also worth noting that if $v\equiv0$ and $u$ is a solution of the heat flow of harmonic maps that blows up in finite time, then $(v,u)$ is also a solution to \eqref{LCF} that blows up in finite time. Therefore, we are only interested in the case of non-trivial velocity field. In our construction, we choose a non-trivial initial velocity $v_0$ to ensure that the velocity field is non-zero at least in a short time period. See Remark \ref{rmk-SS} for more details.

\medskip

Our main result is stated as follows.

\begin{theorem}\label{thm}
Given $k$ distinct points $q_1,\cdots,q_k\in\Omega$, if $T,\epsilon_0>0$ are sufficiently small, then there exists a smooth initial data $(v_0,u_0)$ such that the short time smooth solution $(v,u)$ to the system \eqref{LCF} blows up exactly at those $k$ points as $t\to T$. More precisely, there exist numbers $\kappa_j^*>0$, $\omega_j^*$ and $u_*\in H^1(\Omega)\cap C(\bar\Omega)$ such that
\begin{equation*}
u(x,t)-u_*(x)-\sum_{j=1}^k Q^1_{\omega_j}Q^2_{\alpha_j}Q^3_{\beta_j} \left[W\left(\frac{x-q_j}{\la_j(t)}\right)-W(\infty)\right]\to 0 ~~\mbox{ as }~t\to T,
\end{equation*}
in $H^1(\Omega)\cap L^\infty(\Omega)$, where the blow-up rate and angles satisfy
\begin{equation*}
\la_j(t)=\kappa_j^*\frac{T-t}{|\log(T-t)|^2}(1+o(1))~\mbox{ as }~t\to T,
\end{equation*}
\begin{equation*}
\omega_j \to \omega_j^{*}, \ \alpha_j\to 0, \ \beta_j \to 0, ~\ \mbox{ as } \ t \to T,
\end{equation*}
and
$Q^1_{\omega}, Q^2_{\alpha}$ and $Q^3_{\beta}$ are rotation matrices defined in \eqref{def-Q}.
In particular, it holds that
\begin{equation*}
|\nabla u(\cdot,t)|^2\,dx\rightharpoonup |\nabla u_*|^2\,dx +8\pi \sum_{j=1}^k \delta_{q_j}~\mbox{ as }~t\to T,
\end{equation*}
as convergence of Radon measures.
Furthermore, the velocity field $v\not\equiv 0$ and satisfies
\begin{equation*}
|v(x,t)|\leq c\sum_{j=1}^k  \frac{\la_j^{\nu_{j}-1}(t)}{1+\left|\frac{x-q_j}{\la_j(t)}\right|}, ~~0<t<T,
\end{equation*}
for some $c>0$ and $0<\nu_j<1$, $j=1,\cdots,k$.
\end{theorem}

\medskip

Concerning Theorem \ref{thm}, we would like to point out

\medskip

\begin{remark}
~
\begin{itemize}
\item At each blow-up point $q_j\in\Omega$, $1\le j\le k$, the behavior of the velocity field $v$ is precisely
$$|v(x,t)|\leq   c\la_j^{\nu_j-1}(t)+o(1)~\mbox{ for }~\nu_j\in(0,1).$$
Theorem \ref{thm} suggests that $v$ might also blow up in finite time. In fact we conjecture that $ \|v(\cdot, t)\|_{L^\infty} \sim |\log (T-t)| $ as $t\to T$.
The singularity formation of the velocity field is driven by the Ericksen stress tensor $\nabla\cdot(\nabla u\odot\nabla u-\frac12|\nabla u|^2\mathbb I_2)$, which is induced by
the liquid crystal orientation field $u(x,t)$. Namely, $u(x,t)$ plays a role on generating the singular forcing in the incompressible Navier--Stokes equation.
For results of the Navier--Stokes equation with singular forcing in dimension two, we refer to \cite{CS2010}.

\medskip

\item It is well-known that the pressure $P$ can be recovered from the velocity field $v$ and the forcing. See \cite{Galdi} and \cite{Tsai} for instance.
\end{itemize}
\end{remark}

The proof of Theorem \ref{thm} is based on the {\em inner--outer gluing method}, which has been a very powerful tool in constructing solutions in many elliptic problems, see for instance \cite{DKW07CPAM,DKW11Annals,bubbling10JEMS,manifold15JMPA} and the references therein. Also, this method has been successfully applied to various parabolic flows recently, such as the infinite time and finite time blow-ups in energy critical heat equations \cite{Green16JEMS,173D,17type2,del2018sign,tower7D}, singularity formation for the $2$-dimensional harmonic map heat flow \cite{17HMF},  vortex dynamics in Euler flows \cite{18Euler}, and others arising from geometry and fractional context \cite{18type2Yamabeflow,17halfHMF,18fractionalcritical,sire2019singularity}. We refer the interested readers to a survey by del Pino \cite{delPinosurvey} for more results in parabolic settings.

\medskip

The nematic liquid crystal flow \eqref{LCF} is a strongly coupled system of the incompressible Navier--Stokes equation and the transported harmonic map heat flow. In this paper, the construction of the finite time blow-up solution is close in spirit to the singularity formation of the standard two dimensional harmonic map heat flow
\begin{equation}\label{intro-HMF}
\begin{cases}
\partial_t u=\Delta u+|\nabla u|^2 u,~&\mbox{ in }~\Omega\times(0,T),\\
u=u_{0},~&\mbox{ on }~\partial\Omega\times(0,T),\\
u(\cdot,0)=u_0,~&\mbox{ in }~\Omega.\\
\end{cases}
\end{equation}
In \cite{17HMF}, by the {\em inner--outer gluing method}, Davila, del Pino and Wei successfully constructed type II finite time blow-up for the harmonic map heat flow \eqref{intro-HMF}. More precisely, the solution constructed in \cite{17HMF} takes the bubbling form
$$
 |\nabla u(\cdot, t)|^2 \rightharpoonup  |\nabla u_*|^2 + 8\pi \sum_{j=1}^k \delta_{q_j}, ~\mbox{ as }~ t\to  T,
$$
where $u_*\in H^1(\Omega)\cap C (\bar \Omega)$, $(q_1,\ldots, q_k)\in \Omega^k$ are  given $k$ points, and $\delta_{q_j}$ denotes the unit Dirac mass at $q_j$ for $j=1,\cdots,k.$ The construction in \cite{17HMF} consists of finding a good approximate solution based on the 1-corotational harmonic maps and then looking for the inner and outer profiles of the small perturbations. Basically, the inner problem is the linearization around the harmonic map which captures the heart of the singularity formation, while the outer problem is a heat equation coupled with the inner problem.

\medskip

Our construction of a finite time blow-up solution to the nematic liquid crystal flow \eqref{LCF}--\eqref{BCs} relies crucially on the delicate analysis carried out in \cite{17HMF}. The strategy is to regard the term $v\cdot \nabla u$ in the equation for the orientation field $u$ as a perturbation so that the equation for $u$ is basically the harmonic map heat flow, and regard the equation for the velocity field $v$ as the Stokes system by neglecting $v\cdot \nabla v$. After we establish the estimates for the harmonic map heat flow and the Stokes system, we shall show that the terms $v\cdot \nabla u$ and $v\cdot \nabla v$ are indeed small perturbations in the corresponding equations. The existence of a desired blow-up solution will be finally proved by the fixed point argument.

\medskip

In \cite{17HMF}, the parameter functions $\la(t)$, $\xi(t)$, $\omega(t)$, which correspond to the dilation, translation and rotation about $z$-axis, respectively, are required to adjust certain orthogonality conditions to guarantee the existence of desired solutions. As mentioned, the incompressible Navier--Stokes equation and the transported harmonic map heat flow in nematic liquid crystal flow are essentially coupled, which requires more refined estimates especially for the inner problem where the singularity for the orientation field $u$ takes place. To find a better inner solution, we need to add two new parameter functions $\alpha(t)$ and $\beta(t)$ associated to the rotations about $x$ and $y$ axes, respectively, to adjust the orthogonality conditions at mode $-1.$ After this, we then are able to develop a new linear theory at mode $-1$ which is sufficient to construct the desired solution to problem \eqref{LCF}. See Section \ref{sec-HMF} for details.

\medskip

Very surprisingly, our construction suggests that the incompressible Navier--Stokes equation and transported harmonic map heat flow are strongly coupled through the inner problem of $u$ where the singularity occurs. In other words, the two systems are {\em fully coupled}, as one can see from the following scaling invariance for the system (\ref{LCF})
$$ (v_\lambda (x, t), P_\lambda (x, t), u_\lambda (x, t))= (\lambda v(\lambda t, \lambda^2 t), \lambda^2 P(\lambda t, \lambda^2 x), u(\lambda x, \lambda^2 t)).$$ 
This is the main reason why we have to add the extra parameter functions $\alpha(t)$, $\beta(t)$ and improve the linear theory for the inner problem. Moreover, the assumption that $\epsilon_0\ll 1$ in \eqref{LCF} is required to make the system {\em less coupled}. This leads us to expect that bifurcation phenomena may exist regarding the behavior of $\epsilon_0$. The possible blow-up for the velocity field $v$ is triggered by the singularity of the orientation field $u$ due to the strong coupling. The {\em inner--outer gluing method} carried out in this paper is parabolic in nature and does not rely on any symmetry of the solution, which enables us to construct non-radial blow-up at multiple points in this challenging setting.

\medskip

The paper is organized as follows. In Section \ref{sec-HMF}, we shall discuss the singularity formation for the two dimensional harmonic map heat flow, and develop a new method to improve the linear theory of the inner problem. In Section \ref{sec-SS}, we develop the linear theory for the Stokes system. In Section \ref{sec-LCF}, using
the {\em inner--outer gluing method}, we construct a finite time blow-up solution to the nematic liquid crystal flow by the fixed point argument.

\bigskip
\noindent{{\bf Notation}}. \, Throughout the paper, we shall use the symbol  $``\,\lesssim\,"$ to denote $``\,\leq\, C\,"$ for a positive constant $C$ independent of $t$ and $T$. Here $C$ might be different from line to line.

\medskip


\section{Singularity formation for the harmonic map heat flow in dimension two}\label{sec-HMF}

\medskip

Closely related to the harmonic map heat flow in dimension two, the equation for the orientation field $u$ can be regarded as a transported version with drift term. For the two dimensional harmonic map heat flow $u:\Omega\times [0,T)\to\mathbb S^2$:
\begin{equation*}
\begin{cases}
\partial_t u=\Delta u+|\nabla u|^2 u,~&\mbox{ in }~\Omega\times(0,T),\\
u=u_{0},~&\mbox{ on }~\partial\Omega\times(0,T),\\
u(\cdot,0)=u_0,~&\mbox{ in }~\Omega,
\end{cases}
\end{equation*}
we first introduce some notations and preliminaries.


\subsection{Stationary problem: the equation of harmonic maps and its linearization}

\medskip

The equation of harmonic maps for $U: \R^2\to \mathbb S^2$ is the quasilinear elliptic system
\begin{equation}\label{HM}
\Delta U +|\nabla U|^2 U=0~\mbox{ in }~\R^2.
\end{equation}
For $\la>0$, $\xi\in\R^2$, $\omega,\alpha,\beta\in\R$, we consider the family of solutions to \eqref{HM} given by the following $1$-corotational harmonic maps
\begin{equation*}
U_{\la,\xi,\omega,\alpha,\beta}(x)=Q^1_{\omega}Q^2_{\alpha}Q^3_{\beta}W\left(\frac{x-\xi}{\la}\right), ~ x\in\R^2,
\end{equation*}
where
\begin{equation}
\label{def-Q}
Q^1_{\omega}:=\begin{bmatrix}
       \cos\omega & -\sin\omega & 0 \\[0.3em]
       \sin \omega & \cos\omega  & 0 \\[0.3em]
       0 & 0 & 1\\
     \end{bmatrix}
,~~
Q^2_{\alpha}:=\begin{bmatrix}
       1 & 0 & 0 \\[0.3em]
       0 & \cos\alpha & -\sin\alpha \\[0.3em]
       0 & \sin\alpha & \cos\alpha\\
     \end{bmatrix}
,~~
Q^3_{\beta}:=\begin{bmatrix}
       \cos\beta & 0 & \sin\beta \\[0.3em]
       0 & 1 & 0 \\[0.3em]
       -\sin\beta & 0 & \cos\beta\\
     \end{bmatrix}
\end{equation}
are the rotation matrices about $z$, $x$ and $y$ axes, respectively, and $W$ is the least energy harmonic map
$$W(y)=\frac{1}{1+|y|^2}\begin{bmatrix}
2y\\
|y|^2-1\\
\end{bmatrix},~y\in\R^2.$$
In the polar coordinates $y=\rho e^{i\theta}$, $W(y)$ can be represented as
\begin{equation*}
W(y)=\begin{bmatrix}
e^{i\theta} \sin w(\rho)\\
\cos w(\rho)\\
\end{bmatrix}
,~w(\rho)=\pi-2 \arctan (\rho),
\end{equation*}
and we have
\begin{equation*}
w_{\rho}=-\frac{2}{\rho^2+1},~~\sin w=-\rho w_{\rho}=\frac{2\rho}{\rho^2+1},~~\cos w=\frac{\rho^2-1}{\rho^2+1}.
\end{equation*}
For simplicity, we write
\begin{equation*}
Q_{\omega,\alpha,\beta}:=Q^1_{\omega}Q^2_{\alpha}Q^3_{\beta}.
\end{equation*}
The linearization of the harmonic map operator around $W$ is the elliptic operator
\begin{equation}\label{def-linearization}
L_W[\phi]=\Delta_y \phi + |\nabla W(y)|^2 \phi +2(\nabla W(y)\cdot\nabla\phi) W(y),
\end{equation}
whose kernel functions are given by
\begin{equation}\label{kernels}
\left\{
\begin{aligned}
&Z_{0,1}(y)=\rho w_{\rho}(\rho) E_1(y),\\
&Z_{0,2}(y)=\rho w_{\rho}(\rho) E_2(y),\\
&Z_{1,1}(y)=w_{\rho}(\rho)[\cos \theta E_1(y)+\sin\theta E_2(y)],\\
&Z_{1,2}(y)=w_{\rho}(\rho)[\sin \theta E_1(y)-\cos\theta E_2(y)],\\
&Z_{-1,1}(y)=\rho^2 w_{\rho}(\rho)[\cos \theta E_1(y)-\sin \theta E_2(y)],\\
&Z_{-1,2}(y)=\rho^2 w_{\rho}(\rho)[\sin \theta E_1(y)+\cos \theta E_2(y)],\\
\end{aligned}
\right.
\end{equation}
where the vectors
\begin{equation*}
E_1(y)=\begin{bmatrix}e^{i\theta}\cos w(\rho)\\ -\sin w(\rho)\\ \end{bmatrix},~~E_2(y)=\begin{bmatrix}i e^{i\theta}\\ 0\\ \end{bmatrix}
\end{equation*}
form an orthonormal basis of the tangent space $T_{W(y)} \mathbb{S}^2$. We see that
$$L_W[Z_{i,j}]=0 ~~ {~\rm{ for }} ~ i=\pm 1,0, ~j=1,2.$$
Note that
\begin{equation*}
L_U[\varphi]=\la^{-2} Q_{\omega,\alpha,\beta} L_W[\phi],~~\varphi(x)=\phi(y),~~y=\frac{x-\xi}{\la}.
\end{equation*}
In the sequel, it is of significance to compute the action of $L_U$ on functions whose value is orthogonal to $U$ pointwisely. Define
\begin{equation*}
\Pi_{U^{\perp}}\varphi:=\varphi-(\varphi\cdot U)U.
\end{equation*}
We invoke several useful formulas proved in \cite[Section 3]{17HMF}:
\begin{equation*}
L_U[\Pi_{U^{\perp}}\Phi]=\Pi_{U^{\perp}}\Delta \Phi +\tilde L_{U}[\Phi],
\end{equation*}
where
\begin{equation}\label{def-tildeL}
\tilde L_{U}[\Phi]:=|\nabla U|^2\Pi_{U^{\perp}}\Phi-2\nabla(\Phi\cdot U)\nabla U,
\end{equation}
with
$$\nabla(\Phi\cdot U)\nabla U=\partial_{x_j} (\Phi\cdot U) \partial_{x_j} U.$$
In the polar coordinates
$$\Phi(x)=\Phi(r,\theta),~~x=\xi+r e^{i\theta},$$
\eqref{def-tildeL} can be expressed as (see \cite[Section 3]{17HMF})
\begin{equation*}\label{form1-tildeL}
\tilde L_U[\Phi]=-\frac{2}{\la} w_{\rho}(\rho)\left[(\Phi_r\cdot U) Q_{\omega,\alpha,\beta} E_1 -\frac1r (\Phi_{\theta}\cdot U)Q_{\omega,\alpha,\beta} E_2\right],~~r=\la \rho.
\end{equation*}
Assume that $\Phi(x): \Omega\to \mathbb C\times \R$ is a $C^1$ function in the form
\begin{equation}\label{notanota}
\Phi(x)=\begin{bmatrix}
\varphi_1(x)+i\varphi_2(x)\\
\varphi_3(x)\\
\end{bmatrix}.
\end{equation}
If we write
$$\varphi=\varphi_1 +i\varphi_2,~~\bar\varphi=\varphi_1 -i\varphi_2$$
and
$${\rm div}\varphi=\partial_{x_1} \varphi_1 +\partial_{x_2} \varphi_2,~~{\rm curl} \varphi=\partial_{x_1} \varphi_2 -\partial_{x_2} \varphi_1,$$
then we have the following formula (see \cite[Section 3]{17HMF})
\EQ{\label{Ltilde211}
\tilde L_U[\Phi]=[\tilde L_U]_0[\Phi]+[\tilde L_U]_1[\Phi]+[\tilde L_U]_2[\Phi],
}
where
\begin{equation}
\left\{\label{Ltilde222}
\begin{aligned}
~ [\tilde L_U]_0[\Phi]=&~ \la^{-1} \rho w^2_{\rho} [{\rm div}(e^{-i\omega}\varphi) Q_{\omega,\alpha,\beta} E_1+{\rm curl}(e^{-i\omega}\varphi) Q_{\omega,\alpha,\beta} E_2],\\
[\tilde L_U]_1[\Phi]=&~ -2\la^{-1} w_{\rho}\cos w [(\partial_{x_1}\varphi_3)\cos\theta+(\partial_{x_2}\varphi_3)\sin\theta] Q_{\omega,\alpha,\beta} E_1\\
&~-2\la^{-1} w_{\rho}\cos w [(\partial_{x_1}\varphi_3)\sin\theta-(\partial_{x_2}\varphi_3)\cos\theta] Q_{\omega,\alpha,\beta} E_2,\\
[\tilde L_U]_2[\Phi]=&~ \la^{-1} \rho w^2_{\rho}  [{\rm div}(e^{i\omega}\bar\varphi)\cos 2\theta-{\rm curl}(e^{i\omega}\bar\varphi)\sin 2\theta] Q_{\omega,\alpha,\beta} E_1\\
&~+\la^{-1}\rho w^2_{\rho}  [{\rm div}(e^{i\omega}\bar\varphi)\sin 2\theta+{\rm curl}(e^{i\omega}\bar\varphi)\cos 2\theta] Q_{\omega,\alpha,\beta} E_2.
\end{aligned}
\right.
\end{equation}
If we assume
\begin{equation*}
\Phi(x)=\begin{bmatrix}
\phi(r) e^{i\theta}\\
0\\
\end{bmatrix}
,~~x=\xi+r e^{i\theta},~~r=\la \rho,
\end{equation*}
where $\phi(r)$ is complex-valued, then we have the following formula
\begin{equation*}
\tilde L_U[\Phi]=\frac{2}{\la} w^2_{\rho}(\rho) \left[{\rm Re}(e^{-i\omega} \partial_r \phi(r)) Q_{\omega,\alpha,\beta} E_1 + \frac1r {\rm Im}(e^{-i\omega}\phi(r)) Q_{\omega,\alpha,\beta} E_2\right].
\end{equation*}
If $\Phi$ is of the form
\begin{equation*}
\Phi(x)=\varphi_1(\rho,\theta) Q_{\omega,\alpha,\beta} E_1+\varphi_2(\rho,\theta) Q_{\omega,\alpha,\beta} E_2,~x=\xi+\la\rho e^{i\theta}
\end{equation*}
in the polar coordinates, then the linearized operator $L_U$ acting on $\Phi$ can be expressed as (see \cite[Section 3]{17HMF})
\begin{equation*}
\begin{aligned}
L_U[\Phi]=&~\la^{-2}\left(\partial_{\rho\rho}\varphi_1+\frac{\partial_{\rho} \varphi_1}{\rho}+\frac{\partial_{\theta\theta} \varphi_1}{\rho^2}+(2w_{\rho}^2-\frac{1}{\rho^2}) \varphi_1-\frac{2}{\rho^2}\partial_{\theta} \varphi_2 \cos w\right) Q_{\omega,\alpha,\beta} E_1\\
&~+\la^{-2}\left(\partial_{\rho\rho}\varphi_2+\frac{\partial_{\rho} \varphi_2}{\rho}+\frac{\partial_{\theta\theta} \varphi_2}{\rho^2}+(2w_{\rho}^2-\frac{1}{\rho^2}) \varphi_2+\frac{2}{\rho^2}\partial_{\theta} \varphi_1 \cos w\right) Q_{\omega,\alpha,\beta} E_2.\\
\end{aligned}
\end{equation*}
In next section, we shall find proper approximate solutions to the harmonic map heat flow
based on the $1$-corotational harmonic maps, and evaluate the error.

\medskip

\subsection{Approximate solution and error estimates}

\medskip

We now consider the harmonic map heat flow
\begin{equation}\label{HMF}
\begin{cases}
\partial_t u=\Delta u+|\nabla u|^2 u,~&\mbox{ in }~\Omega\times(0,T),\\
u=u_{0},~&\mbox{ on }~\partial\Omega\times(0,T),\\
u(\cdot,0)=u_0,~&\mbox{ in }~\Omega,\\
\end{cases}
\end{equation}
where $u:\bar\Omega\times(0,T)\to \mathbb S^2$, and $u_0:\bar\Omega\to\mathbb S^2$ is a given smooth map. For notational simplicity, we shall only carry out the construction in the single bubble case $k=1$ and mention the minor changes for the general case when needed.
We define the error operator
\begin{equation*}\label{def-mS}
S[u]=-\partial_t u +\Delta u + |\nabla u|^2 u.
\end{equation*}
We shall look for solution $u(x,t)$ to problem \eqref{HMF} which at leading order takes the form
\begin{equation}\label{def-approx}
U(x,t):=U_{\la(t),\xi(t),\omega(t),\alpha(t),\beta(t)}=Q_{\omega(t),\alpha(t),\beta(t)}W\left(\frac{x-\xi(t)}{\la(t)}\right).
\end{equation}
Here $\la(t)$, $\xi(t)$, $\omega(t)$, $\alpha(t)$ and $\beta(t)$ are parameter functions of class $C^1((0,T))$ to be determined later.
To get a desired blow-up solution, we assume
\begin{equation*}
\la(t)\to 0,~~\xi(t)\to q ~\mbox{ as }~t\to T,
\end{equation*}
where $q$ is a given point in $\Omega$.


A useful observation is that as long as the constraint $|u|=1$ is kept for all $t\in (0,T)$ and $u=U+v$
where the perturbation $v$ is uniformly small, say, $|v|\leq \frac12$,  then for $u$ to solve \eqref{HMF}, it suffices that
\begin{equation}\label{usefulobserv}
S(U+v)=b(x,t) U
\end{equation}
for some scalar function $b$. Indeed, since $|u|=1$, we get
\begin{equation*}
b(U\cdot u)= S(u)\cdot u=-\frac12 \frac{d}{dt} |u|^2+\frac12 \Delta |u|^2=0.
\end{equation*}
Thus $b\equiv 0$ follows from $U\cdot u\geq \frac12$.

We look for the small perturbation $v(x,t)$ with $|U+v|=1$ in the form
\begin{equation*}
v=\Pi_{U^{\perp}}\varphi +a(\Pi_{U^{\perp}} \varphi) U,
\end{equation*}
where $\varphi$ is an arbitrarily small perturbation with values in $\R^3$, and
$$\Pi_{U^{\perp}}\varphi:=\varphi-(\varphi\cdot U)U,\quad a(\zeta)=\sqrt{1-|\zeta|^2}-1.$$
By $\Delta U+|\nabla U|^2 U=0$, we compute
$$S(U+\Pi_{U^{\perp}}\varphi+aU)=-U_t-\partial_t \Pi_{U^{\perp}}\varphi + L_U(\Pi_{U^{\perp}}\varphi) +N_U(\Pi_{U^{\perp}}\varphi) + c(\Pi_{U^{\perp}}\varphi) U,$$
where for $\zeta =\Pi_{U^{\perp}}\varphi$, $a=a(\zeta)$,
\begin{align*}
L_U(\zeta)=&~\Delta \zeta+|\nabla U|^2\zeta+2(\nabla U\cdot \nabla \zeta)U,
\\
N_U( \zeta )
=&~
\big[
2 \nabla (aU)\cdot \nabla (U+ \zeta  )  + 2 \nabla U \cdot \nabla \zeta   + |\nabla \zeta  |^2
+ |\nabla (a U ) |^2 \,
\big] \zeta
- aU_t + 2\nabla a \cdot\nabla U,
\\
c(\zeta)=&~\Delta a -a_t +(|\nabla(U+\zeta+aU)|^2-|\nabla U|^2)(1+a)-2\nabla U\cdot \nabla \zeta.
\end{align*}
Since we just need to have an equation in the form \eqref{usefulobserv} satisfied, we obtain that
\begin{equation}\label{HMF-u}
u=U+\Pi_{U^{\perp}}\varphi +a(\Pi_{U^{\perp}}\varphi) U
\end{equation}
solves \eqref{HMF} if $\varphi$ satisfies
\begin{equation}\label{eqn-varphi1}
-U_t-\partial_t \Pi_{U^{\perp}}\varphi + L_U(\Pi_{U^{\perp}}\varphi) +N_U(\Pi_{U^{\perp}}\varphi) + b(x,t) U=0
\end{equation}
for some scalar function $b(x,t)$. The strategy for constructing $\varphi$ is based on the {\em inner--outer gluing method}. We decompose $\varphi$ in \eqref{HMF-u} into inner and outer profiles
$$\varphi=\varphi_{in}+\varphi_{out},$$
where $\varphi_{in}$, $\varphi_{out}$ solve the inner and outer problems we shall describe below.
In terms of $\varphi_{in}$ and $\varphi_{out}$, equation \eqref{eqn-varphi1} is reduced to
\begin{equation}\label{331331331}
-\partial_t \varphi_{in} + L_U[\varphi_{in}]+\tilde L_U[\varphi_{out}]-\Pi_{U^{\perp}}[\partial_t \varphi_{out}-\Delta \varphi_{out}+U_t]+N_U(\varphi_{in}+\Pi_{U^{\perp}}\varphi_{out})+(\varphi_{out}\cdot U) U_t +b U=0.
\end{equation}
The inner solution $\varphi_{in}$ will be assumed to be supported only near $x=\xi(t)$ and better expressed in the scaled variable $y=\frac{x-\xi(t)}{\la(t)}$ with zero initial condition and $\varphi_{in}\cdot U=0$ so that $\Pi_{U^{\perp}}\varphi_{in}=\varphi_{in}$, while the outer solution $\varphi_{out}$ will consist of several parts whose role is essentially to satisfy \eqref{331331331} in the region away from the concentration point $x=\xi(t)$.

For the outer problem, since we want the size of the error to be small, we shall add three corrections $\Phi^0$, $\Phi^{\alpha}$ and $\Phi^{\beta}$ which depend on the parameter functions $\la(t)$, $\xi(t)$, $\omega(t)$, $\alpha(t)$, $\beta(t)$ such that
$$\Pi_{U^{\perp}}[\partial_t (\Phi^0+\Phi^{\alpha}+\Phi^{\beta})-\Delta (\Phi^0+\Phi^{\alpha}+\Phi^{\beta})+U_t]$$
gets concentrated near $x=\xi(t)$ by eliminating the leading orders in the first error $U_t$ associated to the dilation and rotations about $x$, $y$ and $z$ axes. We write
\begin{equation*}
\varphi_{out}(x,t)=\Psi^*(x,t)+\Phi^0(x,t)+\Phi^{\alpha}(x,t)+\Phi^{\beta}(x,t),
\end{equation*}
where
$$\Psi^*=\psi+Z^*$$
with $Z^*:\Omega\times(0,\infty)\to\mathbb R^3$ satisfying
\begin{equation*}
\begin{cases}
\partial_t Z^*=\Delta Z^*,~&\mbox{ in }\Omega\times(0,\infty),\\
Z^*(\cdot,t)=0,~&\mbox{ on }\partial\Omega\times(0,\infty),\\
Z^*(\cdot,0)=Z^*_0,~&\mbox{ in }\Omega.\\
\end{cases}
\end{equation*}
For the inner problem, we define
$$\varphi_{in}(x,t)=\eta_R Q_{\omega,\alpha,\beta}\phi(y,t)$$
with
$$\eta_R(x,t)=\eta\left(\frac{|x-\xi(t)|}{\la(t)R(t)}\right),~y=\frac{x-\xi(t)}{\la(t)},~\eta(s)=\begin{cases}1,~&\mbox{ for }s<1,\\0,~&\mbox{ for }s>2,\end{cases}$$
where $\phi(y,t)$ satisfies $\phi(\cdot,0)=0$ and $\phi(\cdot,t)\cdot W=0$, and $R(t)>0$ is determined later. Then equation \eqref{eqn-varphi1} becomes
\begin{align}
 \label{eqsys1}
0 & =  \la^{-2}  \eta_R  Q_{\omega,\alpha,\beta}  [- \la^2 \phi_t +  L_ W  [\phi ] + \la^2  Q_{\omega,\alpha,\beta}^{-1} \tilde L_U [\Psi^*] ]
\\
\nonumber
& \quad
+ \eta_R Q_{\omega,\alpha,\beta}( \la^{-1}\dot\la  y\cdot \nabla_y \phi
+ \la^{-1} \dot\xi \cdot\nabla_y \phi  -  (Q_{\omega,\alpha,\beta}^{-1}\frac{d}{dt}Q_{\omega, \alpha,\beta}) \phi
)
\\
\nonumber
& \quad
+      \tilde L_U  [ \Phi^0+\Phi^{\alpha}+\Phi^{\beta}  ]  -
 \Pi_{U ^\perp} [\pd_t (\Phi^0+\Phi^{\alpha}+\Phi^{\beta}) -\Delta_x (\Phi^0+\Phi^{\alpha}+\Phi^{\beta})  + U_t ]
\\
\nonumber
&\quad- \pd_t \Psi^* +\Delta \Psi^* +  (1-\eta_R) \tilde L_U [\Psi^*] +  Q_{\omega,\alpha,\beta}[(\Delta_x \eta_R) \phi + 2  \nabla_x \eta_R \nabla_x \phi - (\partial_t  \eta_R) \phi]
\\
\nonumber
& \quad +
N_U( \eta_R Q_{\omega,\alpha,\beta} \phi  + \Pi_{U^\perp}( \Phi^0+\Phi^{\alpha}+\Phi^{\beta} +\Psi^*) ) + ((\Psi^*+ \Phi^0+\Phi^{\alpha}+\Phi^{\beta})\cdot U)U_t + b U .
\end{align}

We now give the precise definitions of $\Phi^0$, $\Phi^{\alpha}$, $\Phi^{\beta}$, and estimate the error
$$ \tilde L_U  [ \Phi^0+\Phi^{\alpha}+\Phi^{\beta}  ]  -
 \Pi_{U ^\perp} [\pd_t (\Phi^0+\Phi^{\alpha}+\Phi^{\beta}) -\Delta_x (\Phi^0+\Phi^{\alpha}+\Phi^{\beta})  + U_t ].$$
We shall choose $\Phi^0$, $\Phi^{\alpha}$, $\Phi^{\beta}$ in a way such that
$$\pd_t (\Phi^0+\Phi^{\alpha}+\Phi^{\beta}) -\Delta_x (\Phi^0+\Phi^{\alpha}+\Phi^{\beta})  + U_t \approx 0
~~{\rm{for}}~~ |x-\xi|\gg \la$$
so that the error in the outer problem is of smaller order.

The error of the approximate solution defined in \eqref{def-approx} is
\begin{equation*}
\mathcal S[U]=-\partial_t U=-\big[\underbrace{\dot\la \pd_{\la} U + \dot \omega \pd_{\omega} U}_{:=\mathcal E_0} +\underbrace{ \dot\xi \cdot \pd_{\xi} U}_{:=\mathcal E_1} +\underbrace{\dot\alpha \pd_{\alpha} U +\dot\beta \pd_{\beta}U}_{:=\mathcal E_{-1}} \big]
\end{equation*}
where
\begin{equation*}
\left\{
\begin{aligned}
&\pd_{\la} U(x)=\la^{-1} Q_{\omega,\alpha,\beta} Z_{0,1}(y)\\
&\pd_{\omega} U(x)=  Q_{\omega,\alpha,\beta} Z_{0,2}(y)+Q_{\omega,\alpha,\beta}(A_{\alpha,\beta}-J_1)W(y)\\
&\pd_{\xi_1} U(x) = \la^{-1}  Q_{\omega,\alpha,\beta} Z_{1,1}(y)\\
&\pd_{\xi_2} U(x) = \la^{-1}  Q_{\omega,\alpha,\beta} Z_{1,2}(y)\\
&\pd_{\alpha} U(x)=  \frac12 Q_{\omega,\alpha,\beta} \left[ Z_{-1,2}(y)+ Z_{1,2}(y)\right]+Q_{\omega,\alpha,\beta}(A_{\beta}-J_2)W(y)\\
&\pd_{\beta} U(x)=  -\frac12 Q_{\omega,\alpha,\beta} \left[Z_{-1,1}(y)+Z_{1,1}(y)\right]\\
\end{aligned}
\right.
\end{equation*}
with $Z_{i,j}$ defined in \eqref{kernels} for $i=0,\pm 1$, $j=1,2$,
\begin{equation}\label{def-J1J1}
A_{\alpha,\beta}=\begin{bmatrix}
0 & -\cos\alpha\cos\beta & \sin\alpha \\
\cos\alpha\cos\beta & 0 & \cos\alpha\sin\beta \\
-\sin\alpha & -\cos\alpha\sin\beta & 0 \\
\end{bmatrix}
,~J_1=\begin{bmatrix}
0 & -1 & 0 \\
1 & 0 & 0 \\
0 & 0 & 0 \\
\end{bmatrix},
\end{equation}
and
\begin{equation*}
A_{\beta}=\begin{bmatrix}
0 & -\sin\beta & 0 \\
\sin\beta & 0 & -\cos\beta \\
0 & \cos\beta & 0 \\
\end{bmatrix}
,~J_2=\begin{bmatrix}
0 & 0 & 0 \\
0 & 0 & -1 \\
0 & 1 & 0 \\
\end{bmatrix}.
\end{equation*}
It is worth mentioning that $A_{\alpha,\beta}-J_1=o(1)$ and $A_\beta-J_2=o(1)$ as $\alpha,\beta\ll1$. Writing $y=\frac{x-\xi}{\la}=\rho e^{i\theta}$, we have
\begin{align*}
\mathcal E_0(x,t) =&~ -Q_{\omega,\alpha,\beta}\left[\dot\la \la^{-1} \rho w_{\rho}(\rho) E_1(y)+\dot\omega \rho w_{\rho}(\rho) E_2(y)\right],\\
\mathcal E_1(x,t) =&~ -\dot \xi_1 \la^{-1} w_{\rho}(\rho) Q_{\omega,\alpha,\beta} \left[\cos\theta E_1(y)+\sin\theta E_2(y)\right]\\
&~ -\dot\xi_2 \la^{-1} w_{\rho}(\rho) Q_{\omega,\alpha,\beta}\left[\sin\theta E_1(y)-\cos\theta E_2(y)\right].
\end{align*}
Notice that the slow decaying part of the error $\mathcal S[U]$ consists of
\begin{equation*}
\mathcal E_0(x,t)=-\frac{2r}{r^2+\la^2}\left(\dot\la Q_{\omega,\alpha,\beta} E_1+\la\dot\omega Q_{\omega,\alpha,\beta} E_2\right)\approx -\frac{2r}{r^2+\la^2}\begin{bmatrix} (\dot\la+i \la\dot \omega) e^{i(\theta+\omega)}\\ 0\\ \end{bmatrix}
\end{equation*}
and
\begin{equation*}
\begin{aligned}
\mathcal E_{-1}(x,t) =&~ Q_{\omega,\alpha,\beta}\left[ \frac{\dot\alpha}2[Z_{-1,2}(y)+Z_{1,2}(y)]+\dot\alpha(A_\beta-J_2)W-\frac{\dot\beta}2[Z_{-1,1}(y)+Z_{1,1}(y)]\right]\\
:=&~ \mathcal E_{-1,2}+\mathcal E_{-1,1},
\end{aligned}
\end{equation*}
where
\EQN{
\mathcal E_{-1,2}&=Q_{\omega,\alpha,\beta} \frac{\dot\alpha}{1+\rho^2} \begin{bmatrix}
-2\rho\sin\beta\sin\theta\\
2\rho\sin\beta\cos\theta-(\rho^2-1)\cos\beta\\
2\rho\cos\beta\sin\theta\\
\end{bmatrix}
}
and
\begin{equation*}
\mathcal E_{-1,1}=Q_{\omega,\alpha,\beta}\frac{\dot\beta}{1+\rho^2} \begin{bmatrix}
\rho^2-1\\
0\\
-2\rho\cos\theta\\
\end{bmatrix}.
\end{equation*}
In the sequel, we write
\EQN{
p(t)=\la(t) e^{i\omega(t)}.
}
Then
\EQN{
-\frac{2r}{r^2+\la^2}\begin{bmatrix} (\dot\la+i \la\dot \omega) e^{i(\theta+\omega)}\\ 0\\ \end{bmatrix}=-\frac{2r}{r^2+\la^2}\begin{bmatrix} \dot p(t) e^{i\theta}\\ 0\\ \end{bmatrix}:=\tilde{\mathcal E}_0(x,t).
}
To reduce the size of $\mathcal S[U]$, we add corrections
\begin{equation}\label{def-corrections}
\Phi^0[p,\xi]:=\begin{bmatrix}
\varphi^0(r,t) e^{i\theta}\\
0\\
\end{bmatrix}
,~\Phi^{\alpha}=Q_{\omega,\alpha,\beta}\begin{bmatrix}0\\\alpha(t)\\0\end{bmatrix}
,~\Phi^{\beta}=Q_{\omega,\alpha,\beta}\begin{bmatrix}-\beta(t)\\0\\1\end{bmatrix},
\end{equation}
where
\begin{equation*}
\varphi^0(r,t)=-\int_{-T}^t r \dot p(s)k(z(r),t-s) ds
\end{equation*}
with
\begin{equation*}
z(r)=\sqrt{r^2+\la^2},~k(z,t)=2\frac{1-e^{-\frac{z^2}{4t}}}{z^2}.
\end{equation*}
By direct computations, the new error produced by $\Phi^0$ is
$$
\Phi^0_t - \Delta_x \Phi^0  + \ttt \EE_0    =
\ttt \RR_0 +\ttt \RR_1 ,  \quad \ttt \RR_0 =  \begin{bmatrix} \RR_0   \\ 0 \end{bmatrix}  ,\quad
\ttt\RR_1 =  \begin{bmatrix} \RR_1   \\ 0 \end{bmatrix}
$$
where
\[
\RR_0:=  - re^{i\theta}   \frac {\la^2}{z^4} \int_{-T}^t  \dot p(s)  ( z{k_z} - z^2 k_{zz}) (z(r),t-s) \, ds
\]
and
\begin{align*}
\RR_1 & :=
- e^{ i\theta}  {\rm Re}\,( e^{-i\theta} \dot \xi(t))
 \int_{-T}^t  \dot p(s) \, k(z(r),t-s) \, ds
\\
&\qquad
+  \frac r{z^2} e^{i\theta} \, (\la\dot\la(t)  -  {\rm Re}\,( re^{i\theta} \dot\xi(t)) )
\int_{-T}^t  \dot p(s) \ {zk_z}(z(r),t-s)\,  ds.
\end{align*}
Observe that $\RR_1$ is of smaller order. Moreover, we can evaluate
\begin{align*}
& ~\quad
\ttt L_U[\Phi^0]  + \Pi_{U^\perp} [ -U_t + \Delta \Phi^0 -\Phi^0_t ]
\\
 & =   \ttt L_U[\Phi^0]  -\EE_1 + \Pi_{U^\perp} [\ttt \EE_0] - \EE_0  -
\Pi_{U^\perp} [\ttt \RR_0]   -  \Pi_{U^\perp} [\ttt \RR_1]  -\EE_{-1}
\\
&=  \KK_{0}[p,\xi]  + \KK_1[p,\xi] -\Pi_{U^\perp} [\ttt \RR_1]  -\EE_{-1}
\end{align*}
where
\begin{align*}
\KK_{0}[p,\xi] =  \KK_{01}[p,\xi] + \KK_{02}[p,\xi]
\end{align*}
with
\begin{align}
\label{K01}
\KK_{01}[p,\xi]
:= - \frac {2}{\la} \rho w_\rho^2
\int_{-T} ^t  \left [ {\rm Re  } \,( \dot p(s) e^{-i\omega(t)} )   Q_{\omega,\alpha,\beta} E_1+
 {\rm Im  } \,( \dot p(s) e^{-i\omega(t)} ) Q_{\omega,\alpha,\beta} E_2   \right ]
\cdot  k(z,t-s)  \, ds
\end{align}
\begin{align}
\nonumber
\KK_{02}[p,\xi]
&
:=  \frac 1{\la} \rho w_\rho^2  \left [  {\dot\la}
-
\int_{-T} ^t  {\rm Re  } \,( \dot p(s) e^{-i\omega(t)} ) r k_z(z,t-s) z_r \, ds\, \right]  Q_{\omega,\alpha,\beta} E_1
\\
\nonumber
&~\quad
-   \frac{1}{4\lambda} \rho w_\rho^2 \cos w  \left [  \int_{-T}^t   {\rm Re}\, ( \dot p(s)e^{-i\omega(t) }  )
\, ( z{k_z} - z^2 k_{zz}) (z,t-s)\, ds\, \right ] Q_{\omega,\alpha,\beta} E_1
\\
\label{K02}
&~\quad
-    \frac{1}{4\lambda} \rho w_\rho^2  \left [  \int_{-T}^t   {\rm Im }\, ( \dot p(s)e^{-i\omega(t) }  )
\, ( z{k_z} - z^2 k_{zz}) (z,t-s)\, ds\,  \right ]  Q_{\omega,\alpha,\beta} E_2  ,
\end{align}
\begin{align}
\label{K1}
\KK_{1}[p,\xi]
& :=
\frac 1\la  w_\rho \, \big [
\Re \big (  (\dot  \xi_1 - i \dot \xi_2)  e^{i\theta } \big ) Q_{\omega,\alpha,\beta} E_1
+ \Im \big(  (\dot  \xi_1 - i \dot \xi_2)  e^{i\theta } \big ) Q_{\omega,\alpha,\beta} E_2       \big ].
\end{align}

Next we consider the new error estimates produced by $\Phi^{\alpha}$ and $\Phi^{\beta}$. It is obvious that $\tilde L_U[\Phi^\alpha]=0$ and $\tilde L_U[\Phi^\beta]=0$. Direct computations show that
\begin{equation*}
Q_{\omega,\alpha,\beta}^{-1}\left(\frac{d}{dt} Q_{\omega,\alpha,\beta}\right)\begin{bmatrix}
0\\ \alpha\\ 0\\
\end{bmatrix}
=\begin{bmatrix}
-\dot\omega \alpha \cos\alpha \cos\beta -\alpha\dot\alpha \sin\beta\\
0\\
\dot\alpha \alpha \cos\beta-\dot\omega \alpha \cos\alpha\sin\beta
\end{bmatrix},
\end{equation*}
\begin{equation*}
Q_{\omega,\alpha,\beta}^{-1}\left(\frac{d}{dt} Q_{\omega,\alpha,\beta}\right)\begin{bmatrix}
-\beta\\ 0\\ 1\\
\end{bmatrix}
=\begin{bmatrix}
\dot\omega \sin\alpha+\dot\beta\\
\dot\omega(\cos\alpha\sin\beta-\beta\cos\alpha\cos\beta)-\dot\alpha (\beta \sin\beta+\cos\beta)\\
\dot\omega \beta \sin \alpha+\dot\beta \beta
\end{bmatrix},
\end{equation*}
and thus
\EQ{\label{R-12}
-\pd_t\Phi^\alpha+\Delta\Phi^\alpha-\mathcal E_{-1,2}&=Q_{\omega,\alpha,\beta}\begin{bmatrix}\dot\omega\alpha\cos\alpha\cos\beta+\dot\alpha\sin\beta\left(\alpha+\frac{2\rho }{1+\rho^2}\sin\theta\right)\\
-\dot\alpha\left(1-\frac{\rho^2-1}{1+\rho^2}\cos\beta+\frac{2\rho}{1+\rho^2}\sin\beta\cos\theta\right)
\\
\dot\omega\alpha\cos\alpha\sin\beta-\dot\alpha\cos\beta\left(\alpha+\frac{2\rho}{1+\rho^2}\sin\theta\right)\end{bmatrix}\\
&:=\mathcal R_{-1,2}[\alpha,\beta]
}
and
\EQ{\label{R-11}
-\pd_t\Phi^\beta+\Delta\Phi^\beta-\mathcal E_{-1,1}&=Q_{\omega,\alpha,\beta}\begin{bmatrix}\frac{2}{1+\rho^2}\dot\beta-\dot\omega\sin\alpha-\dot\beta
\\
-\dot\omega(\cos\alpha\sin\beta-\beta\cos\alpha\cos\beta)+\dot\alpha (\beta \sin\beta+\cos\beta)
\\
-\dot\omega\beta\sin\alpha-\dot\beta\left(\beta-\frac{2\rho}{1+\rho^2}\cos\theta\right)\end{bmatrix}\\
&:=\mathcal R_{-1,1}[\alpha,\beta].
}
Consequently, we obtain
\EQN{
-\pd_t(\Phi^\alpha+\Phi^\beta)+\Delta(\Phi^\alpha+\Phi^\beta)-\mathcal E_{-1}=\mathcal R_{-1}[\alpha,\beta],
}
where
\EQ{\label{R-1}
\mathcal R_{-1}[\alpha,\beta]:=\mathcal R_{-1,1}[\alpha,\beta]+\mathcal R_{-1,2}[\alpha,\beta].
}
\medskip

\subsection{Inner--outer gluing system}

\medskip

Collecting the error estimates in the previous section, we will get a solution solving \eqref{eqsys1} if the pair $(\phi,\Psi^*)$ solves the {\em inner--outer gluing system}
\begin{equation}\label{eqn-inner}
\left\{
\begin{aligned}
&\la^2\partial_t \phi= L_W[\phi]+\la^2 Q^{-1}_{\omega,\alpha,\beta}\left[\tilde L_U[\Psi^*]+\mathcal K_0[p,\xi]+\mathcal K_1[p,\xi]+\Pi_{U^{\perp}}[\RR_{-1}[\alpha,\beta]]\right],~\mbox{ in }~\mathcal D_{2R}\\
&\phi(\cdot,0)=0,~\mbox{ in }~B_{2R(0)}\\
&\phi\cdot W=0,~\mbox{ in }~\mathcal D_{2R}\\
\end{aligned}
\right.
\end{equation}

\begin{equation}\label{eqn-outer}
\partial_t \Psi^*=\Delta_x\Psi^*+\mathcal G[p,\xi,\Psi^*,\alpha,\beta,\phi]~\mbox{ in }~\Omega\times(0,T),
\end{equation}
where
\begin{equation*}
\begin{aligned}
\mathcal G[p,\xi,\Psi^*,\alpha,\beta,\phi] :=&~(1-\eta_R) \tilde L_U[\Psi^*]+(\Psi^*\cdot U) U_t+Q_{\omega,\alpha,\beta}(\phi\Delta_x \eta_R+2\nabla_x \eta_R\cdot \nabla_x \phi-\phi\partial_t \eta_R)\\
&~+\eta_R Q_{\omega,\alpha,\beta}(- (Q_{\omega,\alpha,\beta}^{-1}\frac{d}{dt}Q_{\omega, \alpha,\beta}) \phi+\la^{-1}\dot\la y\cdot \nabla_y \phi +\la^{-1}\dot\xi\cdot \nabla_y \phi)\\
&~+(1-\eta_R)\left(\mathcal K_0[p,\xi]+\mathcal K_1[p,\xi]+\Pi_{U^{\perp}}[\RR_{-1}[\alpha,\beta]]\right)-\Pi_{U^{\perp}}[\tilde{\mathcal R}_1]\\
&~+N_U[\eta_R Q_{\omega,\alpha,\beta} \phi+\Pi_{U^{\perp}}(\Phi^0+\Phi^{\alpha}+\Phi^{\beta}+\Psi^*)]\\
&~+\left((\Phi^0+\Phi^{\alpha}+\Phi^{\beta})\cdot U\right) U_t,
\end{aligned}
\end{equation*}
 the linearization $L_W[\phi]$ is defined in \eqref{def-linearization}, and
$$\mathcal D_{2R}:=\left\{(y,t):y\in B_{2R(t)}, t\in(0,T)\right\}$$
with the radius
\begin{equation}\label{choice-R}
R=R(t)=\la_*(t)^{-\gamma_*},~\mbox{ with }~\la_*(t)=\frac{|\log T|(T-t)}{|\log(T-t)|^2}~\mbox{ and }~\gamma_*\in(0,1/2).
\end{equation}
The reason for choosing such $R(t)$ and $\la_*(t)$ will be made clear later on. If the pair $(\phi,\Psi^*)$ solves the inner--outer gluing system \eqref{eqn-inner}--\eqref{eqn-outer}, then we get a desired solution
\begin{equation*}
u(x,t)=U+ \Pi_{U^{\perp}}[\eta_R Q_{\omega,\alpha,\beta}\phi+\Psi^*+\Phi^0+\Phi^{\alpha}+\Phi^{\beta}]+a(\Pi_{U^{\perp}}[\eta_R Q_{\omega,\alpha,\beta}\phi+\Psi^*+\Phi^0+\Phi^{\alpha}+\Phi^{\beta}]) U
\end{equation*}
which solves problem \eqref{HMF}. We take the boundary condition $u\big|_{\partial\Omega}={\bf e_3}:=\begin{bmatrix}0\\ 0\\ 1\\\end{bmatrix}$, which amounts to
\begin{equation*}
\Pi_{U^{\perp}}[\Psi^*+\Phi^0+\Phi^{\alpha}+\Phi^{\beta}]+a(\Pi_{U^{\perp}}[U+\Psi^*+\Phi^0+\Phi^{\alpha}+\Phi^{\beta}]) U={\bf e_3} -U~\mbox{ on }~\pd \Omega\times(0,T).
\end{equation*}
So it suffices to take the boundary condition for the outer problem \eqref{eqn-outer} as
\begin{equation*}
\Psi^*\big|_{\pd\Omega} ={\bf e_3}-U-\Phi^0-\Phi^{\alpha}-\Phi^{\beta}.
\end{equation*}

\medskip

\subsection{Reduced equations for parameter functions}\label{sec-redu}

\medskip

In this section, we will derive the parameter functions $\la(t)$, $\xi(t)$, $\omega(t)$, $\alpha(t)$ and $\beta(t)$ at leading order as  $t\to T$.

The inner problem  \eqref{eqn-inner} has the form
\begin{equation}\label{eqn-redu}
\begin{cases}
\la^2 \phi_t   =  L_ W  [\phi ] +   h[p,\xi,\alpha,\beta, \Psi^*] (y,t)  &\inn \mathcal D_{2R},
\\
\phi\cdot  W   =  0   &\inn \mathcal D_{2R},
\\
\phi(\cdot, 0)  = 0 &\inn B_{2R(0)} .
\end{cases}
\end{equation}
Here we recall that we write $p(t)=\la(t) e^{i\omega(t)}$. For convenience, we assume that $h(y,t)$ is defined for all $y\in \R^2$ extending outside $\mathcal D_{2R}$ as
\begin{equation*}
\label{HH2}
h[p,\xi,\alpha,\beta, \Psi^*] = \la^2  Q^{-1}_{\omega,\alpha,\beta}\chi_{\mathcal D_{2R} }\left[\tilde L_U  [\Psi^* ]
+   \KK_{0}[p,\xi]
+    \KK_{1}[p,\xi]  +\Pi_{U^{\perp}}[\RR_{-1}[\alpha,\beta]]\right]    ,
\end{equation*}
where $\chi_A$ denotes the characteristic function of a set $A$,
$\KK_0$ is defined in \eqref{K01}, \eqref{K02}, $\KK_1$ in \eqref{K1} and $\mathcal R_{-1}$ in \eqref{R-1}. If $\lambda(t)$ has a relatively smooth vanishing as $t\to T$, it is then natural that the term $\la^2 \phi_t $ is of smaller order and  the equation \eqref{eqn-redu} is
approximated by the elliptic problem
\begin{align}
\label{linearized-elliptic}
L_ W  [\phi ] +   h[p,\xi,\alpha,\beta, \Psi^*]=0, \quad \phi\cdot W  =0  \inn B_{2R} .
\end{align}
We consider the kernel functions $Z_{l,j}(y)$ defined in \eqref{kernels}, which satisfy $L_ W [Z_{l,j}]=0$ for $l=0,\pm 1$, $j=1,2$. If there is a solution $\phi(y,t)$ to \eqref{linearized-elliptic} with sufficient decay, then necessarily
\begin{equation}\label{ww1}
\int_{B_{2R} }   h[p,\xi,\alpha,\beta, \Psi^*](y,t)\cdot Z_{l,j} (y)\, dy = 0  \quad \mbox{ for all } t\in (0,T) ,
\end{equation}
for $l=0,\pm 1$, $j=1,2$.
These orthogonality conditions \eqref{ww1} amount to an integro-differential system of equations for $p(t)$, $\xi(t)$, $\alpha(t)$, $\beta(t)$, which, as a matter of fact, {\em determine} the correct values of the parameter functions so that the solution pair $(\phi,\Psi^*)$ with appropriate asymptotics exists.

\medskip

For the reduced equations of $p(t)$ and $\xi(t)$ which correspond to mode $l=0$ and mode $l=1$, respectively, we invoke some useful expressions and results in \cite[Section 5]{17HMF}. Let
\begin{align*}
\label{defB0j}
\mathcal B_{0j} [p] (t) :=  \frac{\la}{2\pi} \int_{\R^2}   Q^{-1}_{\omega,\alpha,\beta} [ \KK_{0}[p,\xi]+ \KK_{1}[p,\xi]+\Pi_{U^{\perp}}[\mathcal R_{-1}[\alpha,\beta]]] \cdot Z_{0,j} (y)\, dy,~j=1,2.
\end{align*}
From \eqref{R-1}, \eqref{R-11} and \eqref{R-12}, direct computations yield
\EQ{\label{366366366}
&\quad\int_{B_{2R}}Q_{\omega,\alpha,\beta}^{-1}\Pi_{U^{\perp}}[\mathcal R_{-1}[\alpha,\beta]]\cdot Z_{0,1}(y)\,dy\\
&=\pi\left(-\frac{16R^2}{4R^2+1}+4\log(4R^2+1)\right)(\dot\omega\alpha\cos\alpha\sin\beta-\alpha\dot\alpha\cos\beta-\dot\omega\beta\sin\alpha-\beta\dot\beta),
}
and
\EQ{\label{367367367}
\int_{B_{2R}}Q_{\omega,\alpha,\beta}^{-1}\Pi_{U^{\perp}}[\mathcal R_{-1}[\alpha,\beta]]\cdot Z_{0,2}(y)\,dy=\pi\left(-\frac{16R^2}{4R^2+1}+4\log(4R^2+1)\right) \dot\alpha \sin\beta.
}
Combining \eqref{K01}, \eqref{K02} with \eqref{366366366} and \eqref{367367367},
 the following expressions for $\mathcal B_{01}$, $\mathcal B_{02}$ are readily obtained
\begin{align*}
\mathcal B_{01} [p](t)
&=
  \int_{-T} ^t  {\rm Re  } \,(\dot p(s) e^{-i\omega(t)} )\,
\Gamma_1 \left ( \frac {\la(t)^2}{t-s}   \right )  \,\frac{ ds}{t-s}\,  -2 \dot\la (t) +o(1)
\\
 \mathcal B_{02}[p](t)
& =
 \int_{-T} ^t  {\rm Im  } \,(\dot p(s) e^{-i\omega(t)} )\,
\Gamma_2 \left ( \frac {\la(t)^2}{t-s}   \right )  \,\frac{ ds}{t-s}\, ,
\end{align*}
where $o(1)\to0$ as $t\to T$, and $\Gamma_j(\tau)$  are smooth functions
defined as
\begin{align*}
\Gamma_1 (\tau)
&
=  - \int_0^{\infty} \rho^3 w^3_\rho \left [  K ( \zeta )
+ 2 \zeta K_\zeta (\zeta ) \frac {\rho^2} { 1+ \rho^2}
-4\cos(w) \zeta^2 K_{\zeta\zeta} (\zeta)
\right ]_{\zeta = \tau(1+\rho^2)}   \, d\rho,
\\
\Gamma_2 (\tau) & =
- \int_0^{\infty} \rho^3 w^3_\rho  \left [K(\zeta)   - \zeta^2 K_{\zeta\zeta}(\zeta) \right ]_{\zeta = \tau(1+\rho^2)}
\, d\rho\, ,
\end{align*}
where
\[
 K(\zeta)  =  2\frac {1- e^{-\frac{\zeta}4}} {\zeta}.
\]
Using the expressions of $\Gamma_j (\tau)$, we get
\begin{equation}
\left\{
\begin{aligned}
\label{GammaNear0}
| \Gamma_j (\tau)- 1|
& \le   C \tau(1+  |\log\tau|)~~&\hbox{ for }\tau<1 ,
\\
\nonumber
|\Gamma_j (\tau)|
& \le  \frac C\tau~~&\hbox{ for }\tau> 1.
\end{aligned}
\right.
\end{equation}
Define
\begin{align}
\label{defB0-new}
\mathcal B_0[p ] :=
\frac{1}{2}e^{i \omega(t) }
\left(
\mathcal B_{01}[p ]
+  i\mathcal B_{02}[p ] \right)
\end{align}
and
\begin{align*}
\nonumber
a_{0j}[p,\xi,\alpha,\beta, \Psi^*]
&:=
- \frac{ \la}{2\pi} \int_{B_{2R}}   Q^{-1}_{\omega,\alpha,\beta} \tilde L_U  [\Psi^* ]  \cdot Z_{0,j} (y)\, dy, ~~ j=1,2,
\\
a_{0}[p,\xi, \alpha,\beta,\Psi^*]
& :=
\frac{1}{2}
e^{i \omega(t)} \left( a_{01}[p,\xi,\alpha,\beta, \Psi^*] + i a_{02}[p,\xi,\alpha,\beta, \Psi^*] \right).
\end{align*}
Similarly, we let
\begin{align*}
\mathcal B_{1j} [\xi ] (t)  & :=  \frac{\la}{2\pi} \int_{\R^2}   Q^{-1}_{\omega,\alpha,\beta} [ \KK_{0}[p,\xi]+ \KK_{1}[p,\xi]+\RR_{-1}[\alpha,\beta]] \cdot Z_{1,j} (y)\, dy, ~~ j=1,2,\\
\mathcal B_{1} [\xi ] (t) & :=   \mathcal B_{11}[\xi](t) + i \mathcal B_{12}[\xi](t) .
\end{align*}
Directly using the expressions \eqref{R-1}, \eqref{R-11} and \eqref{R-12}, we have
\EQN{
\int_{B_{2R}}Q_{\omega,\alpha,\beta}^{-1}\Pi_{U^{\perp}}[\mathcal R_{-1}[\alpha,\beta]]\cdot Z_{1,1}(y)\,dy
=&~\frac{8\pi R^2}{4R^2+1}(\dot\omega\alpha\cos\alpha\cos\beta+\dot\alpha\alpha\sin\beta-\dot\omega\sin\alpha+\dot\beta),
}
and
\EQN{
\int_{B_{2R}}Q_{\omega,\alpha,\beta}^{-1}\Pi_{U^{\perp}}[\mathcal R_{-1}[\alpha,\beta]]\cdot Z_{1,2}(y)\,dy=-\frac{8\pi R^2}{4R^2+1}(\dot\alpha-\dot\omega\beta\cos\alpha\cos\beta-\dot\alpha\beta\sin\beta+\dot\omega\cos\alpha\sin\beta).\\
}
Therefore, by \eqref{K1}, \eqref{kernels} and the fact that $\int_0^{\infty} \rho w_\rho^2 d\rho\, =2$, we obtain
\[
\mathcal B_{1} [\xi ](t)\,  =  \, 2[\, \dot \xi_1(t) + i\dot \xi_2(t)+o(1)\,] ~ \mbox{ as }~ t\to T .
\]
At last, we let
\begin{align*}
a_{1j} [p,\xi,\alpha,\beta, \Psi^* ] &:= \frac{\lambda}{2\pi}
\int_{B_{2R}} Q^{-1}_{\omega,\alpha,\beta} \tilde L_U[\Psi^*]\cdot Z_{1,j}(y)  \,dy, ~~ j=1,2,
\\
a_1[p,\xi,\alpha,\beta, \Psi^* ] & := - e^{i \omega(t) } ( a_{11}[p,\xi,\alpha,\beta, \Psi^* ] + i a_{12} [p,\xi,\alpha,\beta, \Psi^* ] ) .
\end{align*}
We thus obtain that
the four conditions \eqref{ww1} for $l=0,1$ are reduced to the system of two complex equations
\begin{align}
\mathcal B_0[p ] & = a_0[p,\xi,\alpha,\beta,\Psi^* ],\label{eqB0}\\
\mathcal B_1[\xi ] & =    a_1[p,\xi,\alpha,\beta,\Psi^* ].\label{eqB1}
\end{align}
We observe that
\begin{align*}
\mathcal B_0[p ]
=   \int_{-T} ^{t-\la^2}    \frac{\dot p(s)}{t-s}ds\, + O\big( \|\dot p\|_\infty \big)+o(1) ~ \mbox{ as }~ t\to T.
\end{align*}
To get an approximation for $a_0$, we need to analyze the operator $\tilde L_U$ in $a_0$. To this end, we write
$$
\Psi^* =  \left [ \begin{matrix}\psi^* \\  \psi^*_3  \end{matrix}   \right ] , \quad \psi^* = \psi^*_1 + i \psi^*_2 .
$$
From \eqref{Ltilde211} and \eqref{Ltilde222}, we have
$$\tilde L_U[\Psi^*](y,t)=[\tilde L_U]_0[\Psi^*]+[\tilde L_U]_1[\Psi^*]+[\tilde L_U]_2[\Psi^*],$$
where
\begin{equation*}
\left\{
\begin{aligned}
~ [\tilde L_U]_0[\Psi^*]=&~ \la^{-1} Q_{\omega,\alpha,\beta}\rho w^2_{\rho} [{\rm div}(e^{-i\omega}\psi^*)E_1+{\rm curl}(e^{-i\omega}\psi^*)E_2],\\
[\tilde L_U]_1[\Psi^*]=&~ -2\la^{-1} Q_{\omega,\alpha,\beta} w_{\rho}\cos w [(\partial_{x_1}\psi^*_3)\cos\theta+(\partial_{x_2}\psi^*_3)\sin\theta]E_1\\
&~-2\la^{-1} Q_{\omega,\alpha,\beta} w_{\rho}\cos w [(\partial_{x_1}\psi^*_3)\sin\theta-(\partial_{x_2}\psi^*_3)\cos\theta]E_2,\\
[\tilde L_U]_2[\Psi^*]=&~ \la^{-1} Q_{\omega,\alpha,\beta} \rho w^2_{\rho}  [{\rm div}(e^{i\omega}\bar\psi^*)\cos 2\theta-{\rm curl}(e^{i\omega}\bar\psi^*)\sin 2\theta]E_1\\
&~+\la^{-1} Q_{\omega,\alpha,\beta} \rho w^2_{\rho}  [{\rm div}(e^{i\omega}\bar\psi^*)\sin 2\theta+{\rm curl}(e^{i\omega}\bar\psi^*)\cos 2\theta]E_2,\\
\end{aligned}
\right.
\end{equation*}
and the differential operators in $\Psi^*$  on the right hand sides
are evaluated at $(x,t)$ with   $x= \xi(t)+ \la(t) y$,  $y = \rho e^{i\theta}$ while $E_j= E_j(y)$ for $j=1,2$. From the above decomposition, assuming that $\Psi^*$ is of class $C^1$ in the space variable, we then get
\[
a_{0}[p,\xi,\alpha,\beta, \Psi^*] =   [   \div \psi^*+  i\curl \psi^*](\xi,t )  + o(1) ~ \mbox{ as }~ t\to T.
\]

\medskip

Similarly, since  $\int_0^\infty w_\rho^2 \cos w \rho \, d\rho = 0 $, we get
\begin{align*}
a_1[p,\xi,\alpha,\beta, \Psi^*] & =   2  ( \pd_{x_1} \psi^*_3 + i\pd_{x_2} \psi^*_3) (\xi, t) \int_{0}^\infty \cos w \,w_\rho^2 \rho \, d\rho  + o(1)
\\
&  =  o(1) ~\mbox{ as } ~t\to T.
\end{align*}

 \medskip

We now simplify the system \eqref{eqB0}--\eqref{eqB1} in the form
\begin{align}
\nonumber
\int_{-T} ^{t-\la^2}     \frac{\dot p(s)}{t-s}ds
& =
[ \div \psi^*+  i\curl \psi^*](\xi(t),t )  + o(1) + O(\|\dot p\|_\infty)  \\
\dot \xi(t) &  =  o(1)~ \mbox{ as }~ t\to T. \label{equB1}
 \end{align}
For the moment, we assume  that the function $\Psi^*(x,t)$ is fixed and sufficiently regular, and we regard $T$ as a parameter that will always be taken smaller if necessary. We recall that we need $\xi(T)=q$ where $q\in \Omega$ is given, and $\la(T)=0$.  Equation  \eqref{equB1} immediately suggests us
to take  $\xi(t) \equiv q$ as the first approximation.
Neglecting lower order terms, $p(t)= \lambda (t) e^{i\omega(t)}$ satisfies the following integro-differential system
\begin{align}
\label{kuj}
\int_{-T} ^{t-\la^2(t)} \frac{ \dot p(s)}{t-s}ds   =
\div \psi^*(q,0 ) + i\curl \psi^*(q,0 ) =: a_0^*.
\end{align}
At this point, we make the following assumption
\begin{align}
\label{negativeDiv}
\div \psi^*(q,0 ) <0,
\end{align}
which implies that
$a_0^*  =  -|a_0^*| e^{i\omega_0}$ for a unique $\omega_0\in (-\frac \pi2 , \frac \pi 2)$. Let us take
\smallskip
$$\omega(t)\equiv \omega_0.$$
\smallskip
Equation \eqref{kuj} then becomes
\begin{equation}\label{cccc4}
 \int_{-T} ^{t-\la^2(t)} \frac{ \dot \la(s)}{t-s}ds   =
 - |a_0^*| .
\end{equation}
 We claim that a good approximate solution to \eqref{cccc4} as $t\to T$  is given by
\[
\dot \la(t) =  -\frac {\kappa} {\log^2(T-t)}
\]
for a suitable $\kappa>0$. Indeed, we have
\begin{align*}
\int_{-T}^{t-\la^2(t)} \frac {\dot\la(s)}{t-s}\, ds \,  =& \
\int_{-T}^{t-  (T-t)  } \frac{ \dot\la  (s)}{t-s} \, ds +     \, \dot \la (t)\left [ \log (T-t)   - 2\log (\la(t)) \right ]\nonumber \\  &~+   \int_{t-(T-t)   } ^{ t- \la^2(t)}\frac{\dot \la(s)-\dot \la(t)}{t-s} ds    \nonumber \\
  \approx & \
\int_{-T}^{t } \frac{ \dot\la  (s)}{T-s}\, ds\,  - \, \dot \la (t) \log (T-t) \, := \Upsilon(t)
\end{align*}
as $t\to T$. We see that
$$
\log(T-t) \frac {d\Upsilon(t)} {dt}  =
   \frac d{dt} (\log^2(T-t) \, \dot\la(t))= 0
$$
from the explicit form of  $\dot\la(t)$. Thus $\Upsilon(t)$ is a  constant. As a consequence, equation \eqref{cccc4}
is approximately satisfied if $\kappa$ is such that
$$
\kappa \int_{-T}^{T} \frac{ \dot\la  (s)}{T-s}\,ds\ =\ -|a_0^*| ,
$$
which finally gives us the approximate expression
$$
\dot\la (t)= -  |\div \psi^*(q,0) + i \curl\psi^* (q,0) |\, \dot \la_* (t) ,
$$
where
\[
 \dot \la_* (t) = -\frac { |\log T|}{\log^2(T-t)}.
\]
Naturally, imposing $\la_*(T) =0$, we then have
\begin{equation}\label{def-lalala}
 \la_* (t) =  \frac { |\log T|}{\log^2(T-t)}(T-t)\, (1+ o(1)) ~\mbox{ as } ~t\to T.
\end{equation}

\medskip

Next, we consider \eqref{ww1} for the case of mode $l=-1$, which gives the reduced equations of $\alpha(t)$ and $\beta(t)$. By \eqref{R-1}, \eqref{R-11} and \eqref{R-12}, we evaluate
\EQN{
&\quad\int_{B_{2R}} Q_{\omega,\alpha,\beta}^{-1}\Pi_{U^{\perp}}[\mathcal R_{-1}[\alpha,\beta]]\cdot Z_{-1,1}(y)\,dy\\
&=4\pi\left(-\frac{4R^2(2R^2+1)}{4R^2+1}+\log(4R^2+1)\right)(-\dot\beta-\dot\omega\sin\alpha+\dot\omega\alpha \cos\alpha\cos\beta+\dot\alpha \alpha \sin\beta)\\
&=8\pi\left[(R^2-\log R)\dot\beta(1+o(1))\right],
}
and
\EQN{
&\quad\int_{B_{2R}} Q_{\omega,\alpha,\beta}^{-1}\Pi_{U^{\perp}}[\mathcal R_{-1}[\alpha,\beta]]\cdot Z_{-1,2}(y)\,dy\\
&=4\pi\left(\frac{4R^2(2R^2+1)}{4R^2+1}-\log(4R^2+1)\right)(\dot\alpha(1-\beta\sin\beta-2\cos\beta)+\dot\omega\cos\alpha(\sin\beta-\beta\cos\beta))\\
&=8\pi\left[(-R^2+\log R)\dot\alpha(1+o(1))\right],
}
where we recall that $\omega(t)\equiv \omega_0$.
Since
\EQN{
\int_{B_{2R}}\la^2 Q_{\omega,\alpha,\beta}^{-1}\left[\tilde{L}_U[\Psi^*]+\mathcal K_0+\mathcal K_1\right]\cdot Z_{-1,j}(y)\,dy=c_j\la
}
for some $c_j\in\R$, for $j=1,2$, the orthogonality condition \eqref{ww1} with $l=-1$ gives
\EQN{
8\pi\la^2(-R^2+\log R)\,\dot\beta(1+o(1))&=c_1\la,\\
8\pi\la^2(R^2-\log R)\,\dot\alpha(1+o(1))&=c_2\la.
}
Thus, by \eqref{def-lalala} and the definition of $R=R(t)$ in \eqref{choice-R}, good choices for $\alpha(t)$ and $\beta(t)$ at leading orders are
$$\alpha(t)=c_{\alpha} (T-t)^{\de_1}(1+o(1)),~~ \beta(t)=c_{\beta} (T-t)^{\de_2}(1+o(1)), ~~\mbox{ as }~~ t\to T$$
for some $\de_1,\,\de_2>0$ and $c_{\alpha},~c_{\beta}\in\R$.

\medskip

\subsection{Linear theory for the inner problem}\label{sec-ltinner}

\medskip

To capture the heart of the singularity formation, a linear theory of the inner problem \eqref{eqn-inner} is required. We consider
\begin{equation}\label{eqn-linearinner}
\begin{cases}
\la^2 \partial_t \phi =L_W[\phi]+h(y,t),~&\mbox{ in }~\mathcal D_{2R},\\
\phi(\cdot,0)=0,~&\mbox{ in }~B_{2R(0)},\\
\phi\cdot W=0,~&\mbox{ in }~\mathcal D_{2R},\\
\end{cases}
\end{equation}
where we recall from \eqref{choice-R} that
$$R=R(t)=\la_*(t)^{-\gamma_*},~\mbox{ with }~\la_*(t)=\frac{|\log T|(T-t)}{|\log(T-t)|^2}~\mbox{ and }~\gamma_*\in(0,1/2).$$

We regard $h(y,t)$ as a function defined in $\R^2\times(0,T)$ with compact support, and assume that $h(y,t)$ has the space-time decay of the following type
$$|h(y,t)|\lesssim \frac{\la_*^{\nu}(t)}{1+|y|^a},\quad h\cdot W=0,$$
where $\nu>0$ and $a\in(2,3)$. Define the norm
\begin{equation*}
\|h\|_{\nu,a}:=\sup_{(y,t)\in\R^2\times(0,T)} \la_*^{-\nu}(t)(1+|y|^a)|h(y,t)|.
\end{equation*}
In the polar coordinates, $h(y,t)$ can be written as
\begin{equation*}
h(y,t)=h^1(\rho,\theta,t) E_1(y) +h^2(\rho,\theta,t) E_2(y),~y=\rho e^{i\theta}
\end{equation*}
since $h\cdot W=0$. Expanding in the Fourier series, we write
\begin{equation}\label{def-Fourier1}
\tilde h(\rho,\theta,t):=h^1+i h^2=\sum_{k=-\infty}^{\infty} \tilde h_k(\rho,t) e^{ik\theta},~\quad\tilde h_k=\tilde h_{k1}+i \tilde h_{k2}
\end{equation}
such that
\begin{equation}\label{def-Fourier2}
h(y,t)=\sum_{k=-\infty}^{\infty} h_k(y,t) :=h_0(y,t)+h_1(y,t)+h_{-1}(y,t)+h_{\perp}(y,t)
\end{equation}
with
\begin{equation}\label{def-Fourier3}
h_k(y,t)={\rm Re}(\tilde h_k(\rho,t)e^{ik\theta})E_1+{\rm Im}(\tilde h_k(\rho,t)e^{i k\theta}) E_2, ~~ k\in\mathbb Z.
\end{equation}
We consider the kernel functions $Z_{k,j}$ defined in \eqref{kernels}, and define
\begin{equation}\label{def-hbar}
\bar h_{k}(y,t)   :=       \sum_{j=1}^2   \frac {\chi Z_{k,j}(y)} { \int_{\R^2} \chi  |Z_{k,j} |^2  }  \, \int_{ \R^2 }h(z  , t )  \cdot Z_{k,j}(z)\, dz, \quad~k=0,\pm 1,~j=1,2,
\end{equation}
where
$$
\chi(y,t) = \begin{cases}
w_\rho^2(|y|)  & \hbox{ if }  |y|< 2R(t),\\
0& \hbox{ if }  |y|\ge 2R(t).
\end{cases}
$$
The main result of this section  is stated as follows.

\begin{prop}
\label{prop-lt}
Assume that $a\in(2,3)$, $\nu>0$, $\delta\in(0,1)$ and $\|h\|_{\nu,a} <+\infty$.
Let us write $$h = h_0 + h_1 + h_{-1} + h_\perp ~\mbox{ with }~ h_{\perp} = \sum_{k\not=0,\pm 1} h_k.$$
Then there exists a solution $\phi[h]$ of problem $\eqref{eqn-linearinner}$, which defines
a linear operator of $h$, and satisfies the following estimate in $\mathcal{D}_{2R}$
\begin{align*}
&\quad
|\phi(y,t)|+(1+|y|)\left |\nabla_y \phi(y,t)\right |  +  (1+|y|)^2\left |\nabla^2_y \phi(y,t)\right |
\\
& \lesssim
   \lambda^\nu_*(t)   \,
 \min\left\{\frac{R^{\delta(5-a)}(t)}{1+|y|^3},  \frac{1}{1+|y|^{a-2}} \right\}
 \, \| h_0 -\bar h_{0} \|_{\nu,a}
+
\frac{ \lambda^\nu_*(t)  R^2(t)} {1+ |y|}  \|\bar h_0\|_{\nu,a}
\\
& \quad
+   \frac{ \lambda^\nu_*(t) }{ 1+ |y|^{a-2} }\, \left \| h_1 - \bar h_1\right  \|_{\nu,a}
+   \frac{ \lambda^\nu_*(t)  R^4(t)} {1+ |y|^2} \left \| \bar h_{1} \right  \|_{\nu,a}
\\
& \quad
+
  \lambda^\nu_*(t)
 \, \| h_{-1} -\bar h_{-1} \|_{\nu,a}
 + \lambda^\nu_*(t) \log R(t) \,   \| \bar h_{-1} \|_{\nu,a}
\\
& \quad
+
\frac{ \lambda^\nu_*(t)  }{ 1+ |y|^{a-2} }\,   \|h_\perp \|_{\nu,a} .
\end{align*}
\end{prop}

The construction of the solution $\phi$ to problem \eqref{eqn-linearinner} will be carried out in each Fourier mode. Write
\begin{equation*}
\phi=\sum_{k=-\infty}^{\infty} \phi_k,~~\quad \phi_k(y,t)={\rm Re} (\varphi_k(\rho,t) e^{ik\theta}) E_1 + {\rm Im} (\varphi_k(\rho,t) e^{ik\theta}) E_2.
\end{equation*}
In each mode $k$, the pair $(\phi_k,h_k)$ satisfies
\begin{equation}\label{eqn-mode}
\begin{cases}
\la^2\partial_t \phi_k = L_W[\phi_k]+h_k(y,t),~&\mbox{ in }~\mathcal D_{4R},\\
\phi_k(y,0)=0,~&\mbox{ in }~B_{4R(0)},\\
\end{cases}
\end{equation}
which is equivalent to the following problem
\begin{equation*}
\begin{cases}
\la^2\partial_t \varphi_k = \mathcal L_k [\varphi_k]+\tilde h_k(\rho,t),~&\mbox{ in }~\tilde{\mathcal D}_{4R},\\
\varphi_k(\rho,0)=0,~&\mbox{ in }~(0,4R(0)),\\
\end{cases}
\end{equation*}
where $\tilde{\mathcal D}_{4R}=\left\{(\rho,t):~t\in(0,T),~\rho\in(0,4R(t))\right\}$, and
\begin{equation*}
\mathcal L_k [\varphi_k] := \partial_{\rho\rho} \varphi_k +\frac{\partial_{\rho}\varphi_k}{\rho}-(k^2+2k\cos w +\cos (2w))\frac{\varphi_k}{\rho^2}.
\end{equation*}
It is direct to see that the kernel functions for $\mathcal L_k$ such that $\mathcal L_k[Z_k]=0$ at modes $k=0,\pm 1$ are given by
\begin{equation}\label{kernels''}
Z_0(\rho)=\frac{\rho}{1+\rho^2},~~Z_1(\rho)=\frac{1}{1+\rho^2},~~Z_{-1}(\rho)=\frac{2\rho^2}{1+\rho^2}.
\end{equation}

We have the following lemma proved in \cite[Section 7]{17HMF}.
\begin{lemma}[\cite{17HMF}]\label{lem-lt1}
Suppose $\nu>0$, $0<a<3$,~$a\neq 1,2$ and
$$\|h_k(y,t)\|_{\nu,a}<+\infty.$$
Then problem \eqref{eqn-mode} has a unique solution which takes the form
$$\phi_k(y,t)={\rm Re} (\varphi_k(\rho,t) e^{ik\theta}) E_1 + {\rm Im} (\varphi_k(\rho,t) e^{ik\theta}) E_2$$
and satisfies the boundary condition
$$\phi_k (y,t)=0,~y\in\partial B_{4R(t)}(0),~\forall~ t\in(0,T).$$
Moreover, the following estimates hold
\begin{equation*}
|\phi_k(y,t)|\lesssim \la_*^{\nu} k^{-2}\|h\|_{\nu,a} \begin{cases} R^{2-a},~&\mbox{ for }~a<2\\ (1+\rho)^{2-a},~&\mbox{ for }~a>2\\ \end{cases} ~\mbox{ for } k\geq 2,
\end{equation*}
\begin{equation*}
|\phi_{-1}(y,t)|\lesssim \la_*^{\nu} \|h\|_{\nu,a} \begin{cases} R^{2-a},~&\mbox{ for }~a<2\\ \log R,~&\mbox{ for }~a>2,\\ \end{cases}
\end{equation*}
\begin{equation*}
|\phi_0(y,t)|\lesssim \frac{\la_*^{\nu}\|h\|_{\nu,a}}{1+\rho} \begin{cases} R^{3-a},~&\mbox{ for }~a<1\\ R^2,~&\mbox{ for }~a>1,\\ \end{cases}
\end{equation*}
\begin{equation*}
|\phi_1(y,t)|\lesssim \frac{\la_*^{\nu}R^4 \|h\|_{\nu,a}}{(1+\rho)^2}.
\end{equation*}
\end{lemma}
The higher regularity estimates for solutions constructed in Lemma \ref{lem-lt1} are given by the following lemma. Before we state the lemma, we first introduce the H\"older semi-norm, which is better expressed in the $(y,\tau)$-variable. Define
\begin{equation}\label{def-tau}
\tau(t)=\int_0^t \frac{ds}{\la^2(s)}
\end{equation}
so that
\begin{equation*}
\begin{cases}
\partial_{\tau} \phi=L_W[\phi]+h(y,\tau)~&\mbox{ in }~\mathcal D_{4\gamma R},\\
\phi(\cdot,0)=0~&\mbox{ in }~B_{4\gamma R(0)}.\\
\end{cases}
\end{equation*}
We denote the parabolic ball
$$\mathcal B_{\ell}(y,\tau)=\{(y',\tau'):|y-y'|^2+|\tau-\tau'|<\ell^2\},$$
and also introduce the H\"older semi-norm
\begin{equation*}
[g]_{\alpha,A}:=\sup_{(y,\tau),(y',\tau')\in A} \frac{|g(y,\tau)-g(y',\tau')|}{|y-y'|^{\alpha}+|\tau-\tau'|^{\alpha/2}}
\end{equation*}
for $\alpha\in(0,1)$ and a set $A$. We denote $C^{\alpha,\alpha/2}(A)$ by the set of functions on $A$ such that $[g]_{\alpha,A}<+\infty$, endowed with the norm
$$
\|g\|_{C^{\alpha,\alpha/2}(A)}=\|g\|_{L^{\infty}(A)}+[g]_{\alpha,A}.
$$

\medskip

\begin{lemma}\label{lem-grad}
Let $\phi$ be a solution to
\begin{equation}\label{eqn-mode''}
\begin{cases}
\la^2 \partial_t \phi = L_W[\phi]+h(y,t),~&\mbox{ in }~\mathcal D_{4\gamma R},\\
\phi(\cdot,0)=0,~&\mbox{ in }~B_{4\gamma R(0)},\\
\end{cases}
\end{equation}
where $h(y,t)\in C^{\alpha,\alpha/2}(\mathcal B_{\ell}(y,\tau)\cap\mathcal D_{4\gamma R})$ for some $\alpha>0$ and $\ell=\frac{|y|}{4}+1$.
If for some $a,b,\gamma,M>0$ we have
\begin{equation}\label{ass-lem2.2}
|\phi(y,t)|+(1+|y|)^2|h(y,t)|+(1+|y|)^{2+\alpha}[h(y,t)]_{\alpha,\mathcal B_{\ell}(y,\tau)\cap \mathcal D_{4\gamma R}} \leq M \frac{\la_*^b(t)}{(1+|y|)^a}~\mbox{ in }~\mathcal D_{4\gamma R},
\end{equation}
then there exists a constant $C$ such that
\begin{equation}\label{est-grad1}
(1+|y|)|\nabla_y \phi(y,t)|+(1+|y|)^{2}|\nabla_y^2 \phi(y,t)|\leq CM\frac{\la_*^b(t)}{(1+|y|)^{a}}~\mbox{ in }~\mathcal D_{3\gamma R}.
\end{equation}
Here
$$\mathcal D_{\gamma R}=\{(y,t): |y|<\gamma R(t),~t\in(0,T)\}.$$
Moreover, if $\phi$ satisfies the Dirichlet boundary condition $\phi(\cdot,t)=0$ on $\partial B_{4\gamma R(t)}$ for all $t\in(0,T)$, then the estimate \eqref{est-grad1} is valid in the entire region $\mathcal D_{4\gamma R}$.
\end{lemma}
\begin{proof}
In the $(y,\tau)$-variable with $\tau$ given by \eqref{def-tau}, problem \eqref{eqn-mode''} reads as
\begin{equation*}
\begin{cases}
\partial_{\tau} \phi=L_W[\phi]+h(y,\tau)~&\mbox{ in }~\mathcal D_{4\gamma R},\\
\phi(\cdot,0)=0~&\mbox{ in }~B_{4\gamma R(0)}.\\
\end{cases}
\end{equation*}
Let $\tau_1>0$ and $y_1\in B_{3\gamma R(\tau_1)}(0)$. Let $\rho=\frac{|y_1|}{4}+1$ so that $B_{\rho}(y_1)\subset B_{4\gamma R(\tau_1)(0)}$. We prove \eqref{est-grad1} by the scaling argument. Define
$$
\tilde \phi(z,s)=\phi(y_1+\rho z,\tau_1+\rho^2 s),~~z\in B_1(0),~~s>-\frac{\tau_1}{\rho^2}.
$$
For the case $\tau_1<\rho^2$, $\tilde \phi(z,s)$ satisfies the following equation
$$
\partial_s \tilde \phi=\Delta_z \tilde \phi+ A(z,s)\cdot \nabla_z \tilde\phi +B(z,s) \tilde \phi +\tilde h(z,s) ~~\mbox{ in }~~B_1(0)\times (-1,0],
$$
where the coefficients $A(z,s)$ and $B(z,s)$ are uniformly bounded by $O((1+\rho)^{-2})$ in $B_1(0)\times (-1,0]$ and
$$
\tilde h(z,s)=\rho^2 h(y_1+\rho z,\tau_1+\rho^2 s).
$$
Let $b'>0$ such that $\tau^{-b'}\sim \la_*^b(t)$ from \eqref{def-tau}. By the facts $\rho \leq C R(\tau_1)$ and $R^2(\tau_1)\ll \tau_1$ for $\tau_1$ large, we have
$$
C_1 \tau_1^{-b'}\leq (\tau_1+\rho^2 s)^{-b'}\leq C_2 \tau_1^{-b'}
$$
for some positive constants $C_1$, $C_2$ independent of $\tau_1$. Then standard interior gradient estimates together with the assumption \eqref{ass-lem2.2} imply
\begin{equation*}
\begin{aligned}
\|\nabla_z \tilde \phi\|_{L^{\infty}(B_{1/4}(0)\times (1,2))}\lesssim &~ \|\tilde \phi\|_{L^{\infty}(B_{1/2}(0)\times (0,2))}+\|\tilde h\|_{L^{\infty}(B_{1/2}(0)\times (0,2))}\\
\lesssim &~ \tau_1^{-b'} \rho^{2-a},
\end{aligned}
\end{equation*}
which in particular gives
$$
\rho|\nabla_y \phi(y_1,\tau_1)|=|\nabla_z \tilde \phi(0,1)|\lesssim  \tau_1^{-b'} \rho^{2-a}.
$$
On the other hand, from interior parabolic Schauder estimates and \eqref{ass-lem2.2}, we have
 \begin{equation*}
\begin{aligned}
\|\nabla^2_z \tilde \phi\|_{L^{\infty}(B_{1/4}(0)\times (1,2))}\lesssim &~ \|\tilde \phi\|_{L^{\infty}(B_{1/2}(0)\times (0,2))}+\|\tilde h\|_{C^{\alpha,\alpha/2}(B_{1/2}(0)\times (0,2))}\\
\lesssim &~ \tau_1^{-b'} \rho^{2-a},
\end{aligned}
\end{equation*}
and in particular
$$
\rho^2|\nabla^2_y \phi(y_1,\tau_1)|=|\nabla^2_z \tilde \phi(0,1)|\lesssim  \tau_1^{-b'} \rho^{2-a}.
$$
For the case $\tau_1\geq \rho^2$ the argument is similar. In this case $\tilde \phi$ satisfies the equation in $B_1(0)\times (-\frac{\tau_1}{\rho^2},0]$ and it has initial condition $0$ at $s=-\frac{\tau_1}{\rho^2}$. Then similarly by the standard boundary estimate, we get the desired bound. Finally, translating the above bounds into $(y,t)$-variable, we conclude the validity of \eqref{est-grad1}.
\end{proof}


As we can see from Lemma \ref{lem-lt1}, the estimates at modes $k=0,\pm 1$ are worse than high modes $k\geq 2$. In fact, if certain orthogonality conditions are imposed on $h(y,t)$, better estimates of $\phi$ can be obtained at modes $k=0,\pm 1$. In the sequel, we omit the subscript for each mode if there is no confusion.

\medskip

\subsubsection{Mode \texorpdfstring{$k=0$}{k=0}}

\medskip

We consider
\begin{equation}\label{eqn-projmode0}
\begin{cases}
\la^2 \partial_t \varphi =L_W[\varphi]+h(y,t)+\sum_{j=1,2} \tilde c_{0j} Z_{0,j} w_{\rho}^2~&\mbox{ in }~\mathcal D_{2R}\\
\varphi\cdot W=0~&\mbox{ in }~\mathcal D_{2R}\\
\varphi=0~&\mbox{ on }~\partial B_{2R}\times (0,T)\\
\varphi(\cdot,0)=0~&\mbox{ in }~B_{2R(0)}\\
\end{cases}
\end{equation}
at mode $0$. By carrying out another inner--gluing scheme for mode $0$, the following Lemma was proved in \cite[Proposition 7.2]{17HMF}.

\medskip

\begin{lemma}[\cite{17HMF}]\label{lem-mode0}
Let $\delta\in(0,1)$, $\nu>0$ and $a\in(2,3)$. Assume $\|h\|_{\nu,a}<+\infty$. Then there exists a solution $(\phi,\tilde c_{0j})$ of problem \eqref{eqn-projmode0} which defines a linear operator in $h(y,t)$ such that
\begin{equation*}
|\varphi(y,t)|+(1+|y|)|\nabla_y \varphi(y,t)|\lesssim \la_*^{\nu}(t) \|h\|_{\nu,a}\begin{cases}
\dfrac{R^{\delta(5-a)}}{(1+|y|)^3},~&\mbox{ for }~|y|\leq 2 R^{\delta}\\
\dfrac{1}{(1+|y|)^{a-2}},~&\mbox{ for }~2 R^{\delta}\leq |y|\leq 2R\\
\end{cases}
\end{equation*}
and
\begin{equation*}
\tilde c_{0j}=-\frac{\int_{\R^2} h Z_{0,j}}{\int_{\R^2} w_{\rho}^2|Z_{0,j}|^2}-G[h],
\end{equation*}
where $G$ is linear in $h$ satisfying
\begin{equation*}
|G[h]|\lesssim \la_*^{\nu}(t) R^{-\delta\sigma'}\|h\|_{\nu,a}
\end{equation*}
for $\sigma'\in(0,a-2)$.
\end{lemma}

\medskip

\subsubsection{Mode \texorpdfstring{$k=-1$}{k=-1}}

\medskip

We consider problem \eqref{eqn-mode} for $k=-1$ and the kernel functions defined in \eqref{kernels}. We first state a result proved in \cite[Lemma 7.5]{17HMF}.
\begin{lemma}[\cite{17HMF}]\label{lm-mode-1}
Let $a\in(2,3)$, $\nu>0$ and $k=-1$. If $h_{-1}$ in \eqref{eqn-mode} satisfies $\|h_{-1}\|_{\nu,a}<\infty$ and
$$\int_{\R^2} h_{-1}(y,t) Z_{-1,j}(y) dy=0~\mbox{ for }~j=1,2,~\forall~ t\in(0,T),$$
then there exists a solution $\phi_{-1}$ to problem \eqref{eqn-mode} at mode $-1$ which defines a linear operator of $h_{-1}$, and $\phi_{-1}$ satisfies
\begin{equation*}
|\phi_{-1}(y,t)|\lesssim \la_*^{\nu}(t)\|h_{-1}\|_{\nu,a}\min\left\{\log R, \frac{R^{4-a}}{1+|y|^2}\right\}.
\end{equation*}
\end{lemma}

\medskip

Since the incompressible Navier--Stokes equation is essentially coupled with the transported harmonic map heat flow through the inner problem, the linear theory required for mode $k=-1$ should be very refined, and Lemma \ref{lm-mode-1} cannot be applied to gain contraction when we finally show the existence of desired blow-up solution. Instead, we shall develop a new  linear theory at mode $-1$. The main result for mode $-1$ is stated as follows.

\medskip

\begin{lemma}\label{lem-mode-1}
Let $a\in(2,3)$, $\nu>0$ and $k=-1$. If $h_{-1}$ in \eqref{eqn-mode} satisfies $\|h_{-1}\|_{\nu,a}<\infty$ and
$$\int_{\R^2} h_{-1}(y,t) Z_{-1,j}(y) dy=0~\mbox{ for }~j=1,2,~\forall~ t\in(0,T),$$
then there exists a solution $\phi_{-1}$ to problem \eqref{eqn-mode} at mode $-1$ which defines a linear operator of $h_{-1}$, and $\phi_{-1}$ satisfies
\begin{equation*}
|\phi_{-1}(y,t)|\lesssim \la_*^{\nu}(t)\|h_{-1}\|_{\nu,a}.
\end{equation*}
\end{lemma}

\begin{proof}
For convenience, we change variable \eqref{def-tau} and consider
\begin{equation*}
\partial_{\tau} \varphi_{-1}= \mathcal L_{-1}[\varphi_{-1}]+\tilde h_{-1}.
\end{equation*}
By letting $\varphi_{-1}(\rho,\tau)=Z_{-1}(\rho) f_{-1}(\rho,\tau)$ and using $\mathcal L_{-1}[Z_{-1}]=0$, we obtain
\begin{equation}\label{eqn-mode-1f-1}
\partial_{\tau} f_{-1}=\frac{1}{Z_{-1}^2} {\rm div}(Z_{-1}^2\nabla f_{-1})+\frac{\tilde h_{-1}}{Z_{-1}},
\end{equation}
where $Z_{-1}(\rho)$ is defined in \eqref{kernels''}.
We first solve
\begin{equation}\label{step1-elliptic}
{\rm div}(Z_{-1}^2\nabla f_{0})= \tilde h_{-1} Z_{-1}.
\end{equation}
By the orthogonality condition $\int_{\R^2} h_{-1}(y,t) Z_{-1,j}(y) dy=0$, we get
\begin{equation}\label{est-gradf0}
|\nabla f_0|\lesssim \frac{\tau^{-\nu'}}{1+|y|^{a-1}}\|h_{-1}\|_{\nu,a},
\end{equation}
where $\nu'>0$ is the number such that $\la_*^{\nu}\sim\tau^{-\nu'}$ under the change of variable \eqref{def-tau}. Thus, by \eqref{step1-elliptic}, the problem \eqref{eqn-mode-1f-1} becomes
\begin{equation*}
\partial_{\tau} f_{-1}=\frac{1}{Z_{-1}^2} {\rm div}(Z_{-1}^2\nabla f_{-1})+\frac{1}{Z_{-1}^2} {\rm div}(Z_{-1}^2\nabla f_{0}).
\end{equation*}
In order to estimate $f_{-1}$, we need to estimate the fundamental solution $S$ to the problem
\begin{equation*}
\left\{
\begin{aligned}
&\partial_{\tau} S =\frac{1}{Z_{-1}^2} {\rm div} (Z_{-1}^2 \nabla S),\\
&S\big|_{\tau=0}=\delta_0,\\
\end{aligned}
\right.
\end{equation*}
where $\delta_0$ is the Dirac delta function at the origin. We consider
\begin{equation*}
\left\{
\begin{aligned}
&\partial_{\tau} S^{\epsilon}=\frac{1}{Z_{-1}^2} {\rm div} (Z_{-1}^2 \nabla S^{\epsilon}),\\
&S^{\epsilon}\big|_{\tau=0}=\frac{1}{2\pi \epsilon^2} e^{-\frac{|x|^2}{2\epsilon^2}}.\\
\end{aligned}
\right.
\end{equation*}
We note that as $\epsilon\to 0$, $S^{\epsilon}\big|_{\tau=0}\,dx\rightharpoonup \delta_0$.
Let $V^{\epsilon}=S^{\epsilon}_{\rho}$. Then differentiating the above equation with respect to $\rho$, we obtain
\begin{equation}\label{eqn-Srho}
\left\{
\begin{aligned}
&\partial_{\tau} V^{\epsilon}=\frac{1}{Z_{-1}^2} {\rm div} (Z_{-1}^2 \nabla V^{\epsilon})+\partial_{\rho \rho}(\log Z_{-1}^2) V^{\epsilon},\\
&V^{\epsilon}\big|_{\tau=0}=-\frac{|x|}{2\pi \epsilon^4} e^{-\frac{|x|^2}{2\epsilon^2}}.\\
\end{aligned}
\right.
\end{equation}
We claim that $V^{\epsilon}<0$. Indeed, we can easily check that $\partial_{\rho \rho}(\log Z_{-1}^2)<0$. Therefore, by $V^{\epsilon}\big|_{\tau=0}=-\frac{|x|}{2\pi \epsilon^4} e^{-\frac{|x|^2}{2\epsilon^2}}<0$ and the maximum principle, we have $V^{\epsilon}<0$. Then we can write
\begin{equation*}
\int_0^{\infty} |S^{\epsilon}_{\rho}(s,\rho)|ds =-\int_0^{\infty} V^{\epsilon}(s,\rho) ds := -M^{\epsilon}(\rho).
\end{equation*}
Integrating equation \eqref{eqn-Srho} over $\tau$ from $0$ to $\infty$, we get
\begin{equation*}
\frac{1}{Z_{-1}^2} {\rm div} (Z_{-1}^2 \nabla M^{\epsilon})+\partial_{\rho \rho} (\log Z_{-1}^2) M^{\epsilon} =-\frac{|x|}{2\pi \epsilon^4} e^{-\frac{|x|^2}{2\epsilon^2}}.
\end{equation*}
Let $M^{\epsilon}=\partial_{\rho} G^{\epsilon}$, where $G^{\epsilon}$ satisfies
\begin{equation}\label{eqn-Gepsilon1}
\frac{1}{Z_{-1}^2} {\rm div} (Z_{-1}^2 \nabla G^{\epsilon})=\frac{1}{2\pi \epsilon^2} e^{-\frac{|x|^2}{2\epsilon^2}}.
\end{equation}
By $Z_{-1}(\rho)=\frac{2\rho^2}{\rho^2+1}$, we write
\begin{equation}\label{eqn-Gepsilon2}
\begin{aligned}
\frac{1}{Z_{-1}^2} {\rm div} (Z_{-1}^2 \nabla G^{\epsilon})=&~\frac{1}{Z_{-1}^2(\rho) \rho} \partial_{\rho} (Z_{-1}^2(\rho) \rho \partial_{\rho} G^{\epsilon})\\
=&~\partial_{\rho \rho} G^{\epsilon} + \frac{\rho^2+5}{\rho(\rho^2+1)} \partial_{\rho} G^{\epsilon}.
\end{aligned}
\end{equation}
From \eqref{eqn-Gepsilon1} and \eqref{eqn-Gepsilon2}, we obtain
\begin{equation*}
\begin{aligned}
\int_0^{\infty} |S_{\rho}^{\epsilon}(s,\rho)| ds =&~ -M^{\epsilon}(\rho)=-\partial_{\rho} G^{\epsilon}(\rho)\\
=&~ \frac{1}{2\pi \epsilon^2} \frac{(1+\rho^2)^2}{\rho^5}\int_{\rho} ^{\infty} \frac{r^5}{(1+r^2)^2} e^{-\frac{r^2}{2 \epsilon^2}} dr\\
\leq&~ \frac{1}{2\pi \epsilon^2} \frac{(1+\rho^2)^2}{\rho^5}\int_{\rho} ^{\infty} r e^{-\frac{r^2}{2 \epsilon^2}} dr\\
\leq&~\frac{1}{2\pi} \frac{1+\rho^4}{\rho^5}.
\end{aligned}
\end{equation*}
Therefore, by letting $\epsilon\to 0$, we obtain
\begin{equation}\label{est-key-1}
\int_0^{\infty} |S_{\rho}(s,\rho)|ds \lesssim \frac{1+\rho^4}{\rho^5}.
\end{equation}
Duhamel's formula gives
\begin{equation*}
\begin{aligned}
f_{-1}(0,\tau)=&~\int_{\tau}^{\infty}\int_0^{\infty} S_{\rho}(s-\tau,\rho) \nabla f_0 Z^2_{-1}(\rho) \rho d\rho ds\\
\lesssim&~\int_0^{\infty}\left(\int_{\tau}^{\infty}  |S_{\rho}(s-\tau,\rho)| ds\right) |\nabla f_0| Z^2_{-1}(\rho) \rho d\rho.
\end{aligned}
\end{equation*}
By \eqref{est-gradf0} and \eqref{est-key-1}, we conclude
\begin{equation*}
|f_{-1}(0,\tau)|\lesssim \tau^{-\nu_1}.
\end{equation*}
In the original time variable $t$, we get
\begin{equation*}
|f_{-1}(0,t)|\lesssim \la_*^{\nu}(t),
\end{equation*}
and parabolic regularity theory readily yields
\begin{equation*}
|f_{-1}(\rho,t)|\lesssim \la_*^{\nu}(t).
\end{equation*}
Therefore, we obtain
$$|\phi_{-1}(y,t)|\lesssim \la_*^{\nu}(t)\|h_{-1}\|_{\nu,a}$$
as desired.
\end{proof}

\medskip

\subsubsection{Mode \texorpdfstring{$k=1$}{k=1}}

\medskip

We assume that $h_1(y,t)$ is defined in the entire space $\R^2\times (0,T)$ such that
\begin{equation}\label{h1divform}
h_1(y,t)={\rm div}_y G(y,t)
\end{equation}
with
\begin{equation}\label{boundG}
|G(y,t)|\lesssim \frac{\la_*^{\nu}(t)}{1+|y|^{a-1}},~(y,t)\in \R^2\times (0,T)
\end{equation}
for $\nu>0$ and $a\in(2,3).$ By the blow-up argument, the following lemma was proved in \cite[Lemma 7.6]{17HMF}.

\begin{lemma}[\cite{17HMF}]\label{lem-mode1}
Assume that $\nu>0$, $a\in(2,3)$ and $h_1$ takes the form \eqref{h1divform} such that \eqref{boundG} holds and
\begin{equation*}
\int_{\R^2} h_1(y,t) Z_{1,j}(y) dy=0~\mbox{ for all }~t\in(0,T)
\end{equation*}
for $j=1,2$. Then there exists a solution $\phi_1(y,t)$ to problem \eqref{eqn-mode} for $k=1$ which defines a linear operator of $h_1(y,t)$, and $\phi_1(y,t)$ satisfies
\begin{equation*}
|\phi_1(y,t)|\lesssim \frac{\la_*^{\nu} (t)\|h_1\|_{\nu,a}}{1+|y|^{a-2}}~\mbox{ in }~\mathcal D_{3R}.
\end{equation*}
\end{lemma}
A direct consequence of Lemma \ref{lem-mode1} is the following
\begin{lemma}[\cite{17HMF}]
Assume $\nu>0$, $a\in(2,3)$ and
\begin{equation*}
\int_{B_{2R}} h_1(y,t) Z_{1,j} (y) dy =0 ~\mbox{ for all }~t\in (0,T)
\end{equation*}
for $j=1,2$. Then there exists a solution $\phi_1(y,t)$ to problem \eqref{eqn-mode} with $k=1$ which defines a linear operator of $h_1(y,t)$, and $\phi_1(y,t)$ satisfies
\begin{equation*}
|\phi_1(y,t)|\lesssim \frac{\la_*^{\nu}(t)\|h_1\|_{\nu,a}}{1+|y|^{a-2}}.
\end{equation*}
\end{lemma}

By the construction in each mode, now we prove Proposition \ref{prop-lt}.
\begin{proof}[Proof of Proposition \ref{prop-lt}]
Let $h$ be defined in $\mathcal D_{2R}$ with $\|h\|_{\nu,a}<+\infty$.
We consider
\begin{equation*}
\begin{cases}
\lambda^2 \pd_t  \phi   =      L_{ W  }  [\phi]       + h  &\mbox{ in }~  \mathcal D_{4R},\\
\phi(\cdot,0 )  =   0   &\mbox{ in }~  B_{4R(0)}.
\end{cases}
\end{equation*}
Let $\phi_k$ be the solution estimated in Lemma \ref{lem-lt1}
to
\begin{equation*}
\begin{cases}
\lambda^2 \pd_t  \phi_k    =      L_{ W  }  [\phi_k]       + h_k&\inn  \mathcal D_{4R}, \\
 \phi_k(\cdot,t )   =    0 &\onn \pd B_{4R} \times (0,T) ,\\
 \phi_k(\cdot,0 )  =    0 &\inn B_{4 R(0)} .
 \end{cases}
\end{equation*}
In addition, we let $\phi_{0,1}$, $\phi_{1,1}$, $\phi_{-1,1}$ solve
\begin{equation*}
\begin{cases}
\lambda^2 \pd_t  \phi_{k,1}    =     L_{ W  }  [\phi_{k,1}]       +  \bar h_k &\inn  \mathcal D_{4R}, \\
 \phi_{k,1}(\cdot,t )  =    0 &\onn \pd B_{4R} \times (0,T) ,\\
 \phi_{k,1}(\cdot,0 )  =    0 &\inn B_{4 R(0)},
 \end{cases}
\end{equation*}
 for $k=0,\pm 1$, where $\bar h_k$ is defined in \eqref{def-hbar}.
Consider the functions $\phi_{0,2}$ constructed in Lemma \ref{lem-mode0},
$\phi_{-1,2}$ constructed in Lemma \ref{lem-mode-1},
and
$\phi_{1,2}$ constructed in Lemma \ref{lem-mode1},
 that solve for $k=0,\pm 1$
\begin{equation*}
\begin{cases}
\lambda^2 \pd_t  \phi_{k,2}    =      L_{ W  }  [\phi_{k,2}]       + h_k-  \bar h_k &\inn  \mathcal D_{3R}, \\
 \phi_{k,2}(\cdot,0 )  =    0 &\inn B_{3 R(0)} .
\end{cases}
\end{equation*}
Define
\[
\phi  : =   \sum_{k=0,\pm 1} (\phi_{k,1} +\phi_{k,2})  +  \sum_{k\neq 0,\pm 1}  \phi_k
\]
which is a bounded solution to the following equation
\[
\lambda^2 \partial_t \phi  =     L_ W  [\phi ]  + h (y,t)\inn \mathcal D_{3R}.
\]
Moreover, it defines a linear operator of $h$.  Applying the estimates for the components in Lemmas \ref{lem-lt1}, \ref{lem-mode0}, \ref{lem-mode-1}, and \ref{lem-mode1},
we obtain
\begin{align*}
  |\phi(y,t)|
 \lesssim &~
   \lambda^\nu_*(t)   \,
 \min\left\{\frac{R^{\delta(5-a)}(t)}{1+|y|^3},  \frac{1}{1+|y|^{a-2}} \right\}
 \, \| h_0 -\bar h_{0} \|_{\nu,a}
+
\frac{ \lambda^\nu_*(t)  R^2(t)} {1+ |y|}  \|\bar h_0\|_{\nu,a}
\\
& ~
+   \frac{ \lambda^\nu_*(t) }{ 1+ |y|^{a-2} }\, \left \| h_1 - \bar h_1\right  \|_{\nu,a}
+   \frac{ \lambda^\nu_*(t)  R^4(t)} {1+ |y|^2} \left \| \bar h_{1} \right  \|_{\nu,a}
\\
& ~
+
  \lambda^\nu_*(t)
 \, \| h_{-1} -\bar h_{-1} \|_{\nu,a}
 + \lambda^\nu_*(t) \log R(t) \,   \| \bar h_{-1} \|_{\nu,a}
\\
& ~
+
\frac{ \lambda^\nu_*(t)  }{ 1+ |y|^{a-2} }\,   \|h_\perp \|_{\nu,a} .
\end{align*}
in $\mathcal D_{3R}$. Finally, Lemma \ref{lem-grad} yields that the same bound holds for
$ (1+|y|)  |\nabla_y \phi | $ and $ (1+|y|)^2  |\nabla^2_y \phi | $
in $\mathcal D_{2R}$.
The function $\phi\big |_{\mathcal D_{2R}}$ solves equation \eqref{eqn-linearinner}, and it defines a linear operator of $h$ satisfying the desired estimates. The proof is complete.
\end{proof}

\medskip

\subsection{Linear theory for the outer problem}\label{sec-ltouter}

\medskip

In order to solve the outer problem \eqref{eqn-outer}, we need to develop a linear theory to the associated linear problem of \eqref{eqn-outer}, which is basically a heat equation.

For  $q\in \Omega$ and $T>0$ sufficiently small, we consider the problem
\begin{align}
\label{heat-eq0}
\begin{cases}
\psi_t   = \Delta_x \psi + f(x,t) &\inn \Omega \times (0,T), \\
\psi   = 0 &\onn \pd \Omega \times (0,T), \\
\psi(x,0)   =  0    &\inn  \Omega.
\end{cases}
\end{align}
The right hand side of \eqref{heat-eq0} is assumed to be bounded with respect to some weights that appear in the outer problem \eqref{eqn-outer}.
Thus we define the weights
\begin{align}\label{weights}
\left\{
\begin{aligned}
\varrho_1 & :=   \lambda_*^{\Theta}  (\lambda_* R)^{-1}  \chi_{ \{ r \leq 3\la_*R \} },
\\
\varrho_2 & := T^{-\sigma_0}  \frac{\lambda_*^{1-\sigma_0}}{r^2}  \chi_{ \{ r \geq  \la_*R \} },
\\
\varrho_3 & := T^{-\sigma_0},
\end{aligned}
\right.
\end{align}
where $r= |x-q|$, $\Theta>0$ and  $\sigma_0>0$ is  small.
For a function $f(x,t)$ we define the $L^\infty$-weighted norm
\begin{align}
\label{defNormRHSpsi}
\|f\|_{**} : =   \sup_{ \Omega \times (0,T)}  \Big ( 1 + \sum_{i=1}^3 \varrho_i(x,t)\, \Big )^{-1}  {|f(x,t)|} .
\end{align}
The factor $T^{\sigma_0}$ in front of $\varrho_2$ and $\varrho_3$ is a simple way to have parts of the error small in the outer problem. Also, we define the $L^\infty$-weighted norm for $\psi$
\begin{align}
\nonumber
\| \psi\|_{\sharp, \Theta,\gamma}
&:=
\lambda^{-\Theta}_*(0)
\frac{1}{|\log T|  \lambda_*(0) R(0) }\|\psi\|_{L^\infty(\Omega\times (0,T))}
+ \lambda^{-\Theta}_*(0) \|\nabla_x \psi\|_{L^\infty(\Omega\times (0,T))}
\\
\nonumber
&~\quad
+
\sup_{\Omega\times (0,T)}   \lambda^{-\Theta-1}_*(t) R^{-1}(t)
\frac{1}{|\log(T-t)|} |\psi(x,t)-\psi(x,T)|
\\
\nonumber
&~\quad
+ \sup_{\Omega\times (0,T)} \, \lambda^{-\Theta}_*(t)
|\nabla_x \psi(x,t)-\nabla_x \psi(x,T) |+ \|\nabla_x^2 \psi\|_{L^{\infty}(\Omega\times(0,T))}
\\
\label{normPsi}
& ~\quad
+ \sup_{}
\lambda^{-\Theta}_*(t)
(\lambda_*(t) R(t))^{2\gamma}  \frac {|\nabla_x \psi(x,t) -\nabla_x \psi(x',t') |}{ ( |x-x'|^2 + |t-t'|)^{\gamma   }} ,
\end{align}
where $\Theta>0$, $\gamma \in (0,\frac{1}{2})$, and the last supremum is taken in the region
\[
x,\, x'\in \Omega,\quad  t,\, t'\in (0,T), \quad |x-x'|\le 2 \la_*(t)R(t), \quad  |t-t'| < \frac 14 (T-t) .
\]

We shall measure the solution $\psi$ to the problem \eqref{heat-eq0}
in the norm $\| \ \|_{\sharp,\Theta,\gamma}$ defined in \eqref{normPsi} where $\gamma\in \Bigl( 0,\frac{1}{2} \Bigr)$, and we require that $\Theta$ and $\gamma_*$ (recall that $R= \lambda_*^{-\gamma_*}$ in \eqref{choice-R}) satisfy
\begin{align}
\label{assumpPar1}
\gamma_* \in \Bigl( 0,\frac{1}{2} \Bigr)  ,
\quad
\Theta \in
(0,\gamma_*).
\end{align}
The condition $\gamma_* \in (0,\frac{1}{2})$ is a basic assumption to have the singularity appear inside the self-similar region.
The condition $\Theta>0$ is needed for Lemma~\ref{lemma-heat1}.
The assumption $\Theta <  \gamma_* $ is made so that the estimates provided by Lemma~\ref{lemmaHeat6} are stronger than that of Lemma~\ref{lemma-heat1}.

\medskip

We invoke some useful estimates proved in \cite[Appendix A]{17HMF} as follows.
\begin{prop}[\cite{17HMF}]
\label{prop3}
Assume  \eqref{assumpPar1} holds.
For $T>0$ sufficiently small, there is a linear operator that maps a function $f:\Omega \times (0,T) \to \R^3$ with  $\|f\|_{**}<\infty$ into $\psi$ which solves problem \eqref{heat-eq0}.
Moreover, the following estimate holds
\begin{align*}
\| \psi\|_{\sharp, \Theta ,\gamma}
 \leq C \|f\|_{**} ,
\end{align*}
where $\gamma\in(0,\frac{1}{2})$.
\end{prop}

The proof of Proposition~\ref{prop3} was achieved in \cite{17HMF} by considering
\begin{align}
\label{heatEqOmega}
\left\{
\begin{aligned}
\psi_t &= \Delta \psi + f \quad \text{in }\Omega \times (0,T),
\\
\psi(x,0) &= 0 , \quad x \in \Omega,
\\
\psi(x,t) &= 0 ,\quad x \in \partial\Omega\times(0,T),
\end{aligned}
\right.
\end{align}
and decomposing the equation into three parts corresponding to the weights of the right hand side defined in \eqref{weights}.

\begin{lemma}[\cite{17HMF}]
\label{lemma-heat1}
Assume $\gamma_*\in (0,\frac{1}{2})$ and $\Theta>0 $.
Let $\psi$ solve \eqref{heatEqOmega} with $f$ satisfying
\[
|f(x,t)|\leq \lambda^{\Theta}_*(t) (\lambda_* (t) R(t))^{-1}
\chi_{ \{  |x-q| \leq 3 \lambda_*(t) R(t) \}} .
\]
Then the following estimates hold
\begin{align*}
|\psi(x,t)| & \leq  C \lambda^\Theta_*(0) \lambda_*(0) R(0) |\log T|,
\\
|\psi(x,t)-\psi(x,T)| & \leq  C  \lambda^\Theta_*(t)  \lambda_*(t) R(t)|\log(T-t)|,
\\
| \nabla \psi(x,t)|  &  \leq C \lambda^\Theta_*(0),
\\
| \nabla \psi(x,t) - \psi(x,T)|  &  \leq C \lambda^\Theta_*(t),\\
|\nabla_x^2 \psi(x,t)| & \leq C,\\
\end{align*}
and for any $\gamma \in (0,\frac{1}{2})$,
\begin{align*}
\frac{|\nabla \psi(x,t)-\nabla\psi(x,t')|}{ |t-t'|^{\gamma}}
\leq
C
\frac{ \lambda^\Theta_*(t) }{  ( \lambda_*(t) R(t) )^{2 \gamma} }
\end{align*}
for any $x,$ and $0\leq t'\leq t\leq T$ such that $t-t'\leq\frac{1}{10}(T-t)$,
\begin{align*}
\frac{|\nabla \psi(x,t)-\nabla\psi(x',t')|}{ |x-x'|^{2\gamma}}
\leq
C
\frac{ \lambda^\Theta_*(t) }{  ( \lambda_*(t) R(t) )^{2\gamma} }
\end{align*}
for any $|x - x'|\leq 2 \lambda_*(t) R(t)$ and $0\leq t\leq T$.
\end{lemma}

\begin{lemma}[\cite{17HMF}]
\label{lemmaHeat6}
Assume $\gamma_* \in (0,\frac{1}{2})$  and $ m \in (\frac{1}{2},1)$.
Let $\psi$ solve \eqref{heatEqOmega} with $f$ satisfying
\[
|f(x,t)|\leq \frac{\lambda^m_*(t)}{|z-q|^2}  \chi_{ \{  |x-q| \geq  \lambda_*(t) R(t) \}}  .
\]
Then the following estimates hold
\begin{align*}
|\psi(x,t)|
& \leq C T^m |\log T|^{2-m} ,
\\
|\psi(x,t) - \psi(x,T)|
& \leq
C
|\log T|^m (T-t)^m |\log(T-t)|^{2-2m} ,
\\
|\nabla\psi(x,t)|
& \leq
C \frac{ T^{m-1}  |\log T|^{2-m} }{ R( T )},
\\
|\nabla \psi(x,t)-\nabla\psi(x,T)|
&
\leq
C \frac{\lambda^{m-1}_*(t) |\log(T-t)|}{ R(t) },\\
|\nabla_x^2 \psi(x,t)| & \leq C,\\
\end{align*}
and for any $\gamma \in (0,\frac{1}{2})$,
\begin{align*}
\frac{|\nabla \psi(x,t)-\nabla\psi(x',t')|}{(|x-x'|^2 + |t-t'|)^{\gamma}}
\leq C
\frac{1}{(\lambda_*(t) R(t) )^{2\gamma} }
\frac{\lambda^{m-1}_*(t) |\log(T-t)|}{ R(t) }
\end{align*}
for any  $|x - x'|\leq 2 \lambda_*(t) R(t)$ and $0\leq t'\leq t\leq T$ such that $t-t'\leq\frac{1}{10}(T-t)$.
\end{lemma}

\begin{lemma}[\cite{17HMF}]	
\label{lemma-heat4}
Let $\psi$ solve \eqref{heatEqOmega} with $f$ such that
\[
|f(x,t)|\leq 1 ,
\]
Then the following estimates hold
\begin{align*}
|\psi(x,t) |
&\leq  C t ,
\\
|\psi(x,t)-\psi(x,T)|
& \leq C (T-t) |\log(T-t)|,
\\
|\nabla \psi (x,t) | & \leq C T^{1/2},
\\
|\nabla \psi(x,t) - \nabla \psi(x,T) |
&\leq C  (T-t)^{1/2},\\
|\nabla_x^2 \psi(x,t)| & \leq C,\\
|\nabla \psi(x,t_2) - \nabla \psi(x,t_1) |
&\leq C  |t_2-t_1|^{1/2}, \\
|\nabla \psi(x_1,t) - \nabla \psi(x_2,t) |
&\leq C |x_1-x_2| |\log(|x_1-x_2|)| .\\
\end{align*}
\end{lemma}	

\begin{remark}
We note that the estimates for $|\nabla_x^2 \psi(x,t)|$ in Lemmas \ref{lemma-heat1}--\ref{lemma-heat4} are achieved by  writing the original equation \eqref{heatEqOmega} in the self-similar variables $(y,\tau)$:
$$
\psi(x,t)=\tilde \psi\left(\frac{x-\xi}{\la},\tau(t)\right),
$$
where $y=\frac{x-\xi}{\la}$ and $\tau$ is defined in \eqref{def-tau}. Then $\tilde \psi(y,\tau)$ satisfies the equation
$$
\partial_{\tau} \tilde \psi = \Delta_y \tilde\psi +(\la \dot\xi+\dot\la \la y)\cdot \nabla_y \tilde\psi +\la^2 f(\la y +\xi, t(\tau)).
$$
By similar argument as in the proof of Lemma \ref{lem-grad}, we can show the boundedness of $|\nabla_x^2 \psi(x,t)|$ by the scaling argument and parabolic regularity estimates, which is sufficient for the final gluing procedure in Section \ref{sec-LCF} to work.
\end{remark}

\medskip


\section{Model problem: Stokes system}\label{sec-SS}

\medskip

In order to solve the incompressible Navier--Stokes equation in \eqref{LCF}, a linear theory of certain linearized problem is required. In this section, we consider the Stokes system
\begin{equation}\label{eqn-SS}
\begin{cases}
\pd_t v+\nabla P =\Delta v +\na\cdot F,~&\mbox{ in }~\Omega\times(0,T),\\
\na\cdot v=0,~&\mbox{ in }~\Omega\times(0,T),\\
v=0,~&\mbox{ on }~\pd\Omega\times(0,T),\\
v(\cdot,0)=v_0,~&\mbox{ in }~\Omega,\\
\end{cases}
\end{equation}
which is the linearized problem of the incompressible Navier--Stokes equation in \eqref{LCF}. The idea is the following. Apriori we assume that the nonlinearity $v\cdot \nabla v$ is a perturbation under certain topology. Then we develop a linear theory for the Stokes system under which we shall see that $v\cdot \nabla v$ is indeed a smaller perturbation.

Our aim is to find a velocity field $v$ solving \eqref{eqn-SS} with proper decay ensuring the inner--outer gluing scheme to be carried out. Suppose that $F(x,t)$ in \eqref{eqn-SS} has the space-time decay of the type
\EQ{\label{ineqineqineq}|F(x,t)|\leq C\frac{\la_*^{\nu-2}(t)}{1+\left|\frac{x-q}{\la_*(t)}\right|^{a+1}},~~|\nabla_x F(x,t)|\leq C\frac{\la_*^{\nu-3}(t)}{1+\left|\frac{x-q}{\la_*(t)}\right|^{a+2}}}
for $\nu>0$ and $a>1$. Here $q\in\Omega$ is the singular point for the orientation field
$u(x,t)$ and
$$\la_*(t)=\frac{|\log T|(T-t)}{|\log(T-t)|^2}.$$
We define the norm
\begin{equation}\label{def-normSSF}
\begin{aligned}
\|F\|_{S,\nu-2,a+1}:=&~\sup_{(x,t)\in\Omega\times(0,T)} \la_*^{2-\nu}(t)\left(1+\left|\frac{x-q}{\la_*(t)}\right|^{a+1}\right)|F(x,t)|\\
 &~+ \sup_{(x,t)\in\Omega\times(0,T)} \la_*^{3-\nu}(t)\left(1+\left|\frac{x-q}{\la_*(t)}\right|^{a+2}\right)|\nabla_x F(x,t)|.
\end{aligned}
\end{equation}
The main result of this section is stated as follows.

\begin{prop}\label{prop-SS}
Assume that $\|F\|_{S,\nu-2,a+1}<+\infty$ with $\nu>0$, $a>1$, and $\|v_0\|_{B^{2-2/p}_{p,p}}<+\infty$, where the Besov norm $\|\cdot\|_{B^{2-2/p}_{p,p}}$ is defined by \eqref{besov}. Then there exists a solution $(v,P)$ to the Stokes system \eqref{eqn-SS} satisfying
\begin{itemize}
\item in the region near $q$: $B_{2\delta}(q)=\{x\in\Omega: |x-q|<2\delta\}$ for $\delta>0$ fixed and small,
$$|v(x,t)|\lesssim \|F\|_{S,\nu-2,a+1} \frac{\la_*^{\nu-1}(t)}{1+\left|\frac{x-q}{\la_*(t)}\right|},$$
and
$$|P(x,t)|\lesssim \|F\|_{S,\nu-2,a+1} \left(\frac{\la_*^{\nu}(t)}{|x-q|^2}+\frac{\la_*^{\nu-2}(t)}{1+\left|\frac{x-q}{\la_*(t)}\right|^{a+1}}\right).$$

\item in the region away from $q$: $\Omega\setminus B_{\delta}(q)$
$$\|v\|_{W^{2,1}_p((\Omega\setminus B_{\delta}(q))\times(0,T))}+\|\na P\|_{L^p((\Omega\setminus B_{\delta}(q))\times(0,T))}\lesssim\|F\|_{S,\nu-2,a+1}+\|v_0\|_{B^{2-2/p}_{p,p}}$$
for $(\nu-1)p+1>0$. Moreover, if $\nu>1/2$, then
\begin{equation*}
\|v\|_{C^{\al,\al/2}((\Omega\setminus B_{\delta}(q))\times(0,T))}\lesssim\|F\|_{S,\nu-2,a+1}+\|v_0\|_{B^{2-2/p}_{p,p}}
\end{equation*}
for $0<\al\le2-4/p$.
\end{itemize}
\end{prop}

\medskip


To prove Proposition \ref{prop-SS}, we decompose the solution $v(x,t)$ to problem \eqref{eqn-SS} into inner and outer profiles
$$v(x,t)=\eta_{\delta} v_{in}(x,t)+v_{out}(x,t),$$
where the smooth cut-off function
\begin{equation*}
\eta_{\delta}(x)=
\begin{cases}
1,~&\mbox{ for }~|x-q|<\delta\\
0,~&\mbox{ for }~|x-q|>2\delta\\
\end{cases}
\end{equation*}
with $\delta>0$ fixed and sufficiently small such that $dist(q,\partial \Omega)>2\delta$. We denote
$$B_{2\delta}(q)=\{x\in\Omega: |x-q|<2\delta\}.$$
It is direct to see that a solution to problem \eqref{eqn-SS} is found if $v_{in}$ and $v_{out}$ satisfy
\begin{equation}\label{eqn-vin}
\begin{cases}
\partial_t v_{in}+\nabla P_1=\Delta v_{in}+\nabla\cdot F_{in},~&\mbox{ in }~\R^2\times(0,T),\\
\nabla\cdot v_{in}=0,~&\mbox{ in }~\R^2\times(0,T),\\
v_{in}(\cdot,0)=0,~&\mbox{ in }~\R^2,\\
\end{cases}
\end{equation}

\medskip

\begin{equation}\label{eqn-vout}
\left\{
\begin{aligned}
&\partial_t v_{out}+\nabla(P-\eta_{\delta}P_1)=\Delta v_{out}+(1-\eta_{\delta})\nabla\cdot F+2\,\nabla \eta_{\delta}\cdot \nabla v_{in}\\
&\qquad\qquad\qquad\qquad\qquad\qquad+(\Delta \eta_{\delta})v_{in}-P_1\nabla \eta_{\delta},~\mbox{ in }~\Omega\times(0,T),\\
&\nabla\cdot v_{out}=-\nabla \eta_{\delta}\cdot v_{in} ,~\mbox{ in }~\Omega\times(0,T),\\
&v_{out}=0,~\mbox{ on }~\partial\Omega\times(0,T),\\
&v_{out}(\cdot,0)=v_0,~\mbox{ in }~\Omega,
\end{aligned}
\right.
\end{equation}
where $F_{in}=F\chi_{\{B_{2\delta}(q)\times (0,T)\}}$. The estimate of the inner part \eqref{eqn-vin} is achieved by the representation formula in the entire space, while the outer part \eqref{eqn-vout} is done by $W^{2,1}_p$-theory of the Stokes system.

\begin{lemma}\label{apriori-SSin}
For $\|F\|_{S,\nu-2,a+1}<+\infty$, the solution $(v_{in},P_1)$ of the system \eqref{eqn-vin} satisfies
\begin{equation}\label{est-vin}
|v_{in}(x,t)|\lesssim \|F\|_{S,\nu-2,a+1} \frac{\la_*^{\nu-1}(t)}{1+\left|\frac{x-q}{\la_*(t)}\right|},
\end{equation}
and
\begin{equation}\label{est-P1}
|P_1(x,t)|\lesssim \|F\|_{S,\nu-2,a+1} \left(\frac{\la_*^{\nu}(t)}{|x-q|^2}+\frac{\la_*^{\nu-2}(t)}{1+\left|\frac{x-q}{\la_*(t)}\right|^{a+1}}\right).
\end{equation}
\end{lemma}
\begin{proof}
For simplicity, we shall write $v_{in}$ as $v$ in the following proof. Denote $v=\begin{bmatrix}v_1\\v_2\\ \end{bmatrix}$. The estimate \eqref{est-vin} is obtained by the well-known representation formula in the entire space
\begin{equation*}
v_i(x,t)=\int_{\R^2}S_{ij}(x-z,t)(v(\cdot, 0))_j(z)\,dz-\int_0^t \int_{\R^2}\pd_{z_k}S_{ij}(x-z,t-s) F_{jk}(z,s)\,dz ds,
\end{equation*}
where $S_{ij}$ is the Oseen tensor, which is the fundamental solution of the non-stationary Stokes system derived by Oseen \cite{oseen1927neuere}, defined by
\begin{equation*}
S_{ij}(x,t)=G(x,t)\delta_{ij}-\frac{1}{2\pi}\frac{\pd^2}{\pd x_i\pd x_j}\int_{\R^2} G(y,t)\log |x-y|\, dy
\end{equation*}
with $G(x,t)=\frac{e^{-\frac{|x|^2}{4t}}}{4\pi t}$, and $F=(F_{jk})_{2\times 2}$. It is well known (see \cite{Solo1964} for instance) that
\begin{equation}\label{est-Oseentensor}
|D_x^l D_t^k S(x,t)|\leq C_{k,l} \frac{1}{(|x|^2+t)^{k+\frac{2+l}{2}}}.
\end{equation}
Since $v(\cdot, 0)=0$, we then get for $i=1,2$,
\begin{equation}\label{est-vvvvvv}
\begin{aligned}
|v_i(x,t)|\lesssim &~ \|F\|_{S,\nu-2,a+1}\int_0^t \int_{\R^2} \frac{1}{(|x-z|+\sqrt{t-s})^3}\frac{\la_*^{\nu-2}(s)}{1+\left|\frac{z-q}{\la_*(s)}\right|^{a+1}}\,dz ds\\
:= &~ \|F\|_{S,\nu-2,a+1} (I_1+I_2),
\end{aligned}
\end{equation}
where we decompose
\begin{equation*}
I_1=\int_0^{t-(T-t)^2} \int_{\R^2} \frac{1}{(|x-z|+\sqrt{t-s})^3}\frac{\la_*^{\nu-2}(s)}{1+\left|\frac{z-q}{\la_*(s)}\right|^{a+1}}\,dz ds,
\end{equation*}
and
\begin{equation*}
I_2=\int_{t-(T-t)^2}^t \int_{\R^2} \frac{1}{(|x-z|+\sqrt{t-s})^3}\frac{\la_*^{\nu-2}(s)}{1+\left|\frac{z-q}{\la_*(s)}\right|^{a+1}}\,dz ds.
\end{equation*}

\medskip

\noindent {\bf Estimate of $I_1$.}

\medskip

To estimate $I_1$, we evaluate
\begin{equation}\label{est-I110}
\begin{aligned}
I_1\lesssim &~ \int_0^{t-(T-t)^2} \int_{\R^2} \frac{1}{(|x-z|^2+(t-s))^{3/2}}\frac{\la_*^{\nu+a-1}(s)}{\la_*^{a+1}(s)+\left|z-q\right|^{a+1}}\,dz ds\\
\lesssim&~ \la_*^{\nu+a-1}(t)\int_{\R^2} \frac{1}{\la_*^{a+1}(t)+\left|z-q\right|^{a+1}} \frac{1}{(|x-z|^2+(T-t)^2)^{1/2}}\,dz\\
\lesssim&~ \la_*^{\nu+a-1}(t)\left(\int_{D_1(x)}+\int_{D_2(x)}+\int_{D_3(x)}\right) \frac{1}{\la_*^{a+1}(t)+\left|z-q\right|^{a+1}} \frac{1}{(|x-z|^2+\la_*^2(t))^{1/2}}\,dz,\\
\end{aligned}
\end{equation}
where
\begin{align}
\label{def-DD1}
D_1(x):=&\left\{z\in\R^2: |z-q|\leq \frac{|x-q|}{2}\right\},
\\
\label{def-DD2}
D_2(x):=&\left\{z\in\R^2: \frac{|x-q|}{2}\leq |z-q|\leq 2|x-q|\right\},
\\
\label{def-DD3}
D_3(x):=&\left\{z\in\R^2: |z-q|\geq 2|x-q|\right\}.
\end{align}
We first compute
\begin{equation}\label{est-I111}
\begin{aligned}
&\quad\int_{D_1(x)} \frac{1}{\la_*^{a+1}(t)+\left|z-q\right|^{a+1}} \frac{1}{(|x-z|^2+\la_*^2(t))^{1/2}}\,dz\\
&\lesssim \frac{1}{|x-q|+\la_*(t)}\int_0^{\frac{|x-q|}{2}} \frac{r}{\la_*^{a+1}(t)+r^{a+1}} dr\\
&\lesssim \frac{\la_*^{1-a}(t)}{|x-q|+\la_*(t)}.
\end{aligned}
\end{equation}
Similarly, we have
\begin{equation}\label{est-I112}
\begin{aligned}
&\quad\int_{D_2(x)} \frac{1}{\la_*^{a+1}(t)+\left|z-q\right|^{a+1}} \frac{1}{(|x-z|^2+\la_*^2(t))^{1/2}}\,dz\\
&\lesssim \frac{1}{\la_*^{a+1}(t)+\left|x-q\right|^{a+1}}\int_0^{3|x-q|} \frac{r}{r+\la_*(t)}\,dr\\
&\lesssim \frac{1}{|x-q|^a+\la_*^a(t)},
\end{aligned}
\end{equation}
and
\begin{equation}\label{est-I113}
\begin{aligned}
&\quad\int_{D_3(x)} \frac{1}{\la_*^{a+1}(t)+\left|z-q\right|^{a+1}} \frac{1}{(|x-z|^2+\la_*^2(t))^{1/2}}\,dz\\
&\lesssim \frac{1}{|x-q|+\la_*(t)}\int_{2|x-q|}^{\infty} \frac{r}{\la_*^{a+1}(t)+r^{a+1}}\,dr\\
&\lesssim \frac{1}{|x-q|^a+\la_*^a(t)}.
\end{aligned}
\end{equation}
Collecting \eqref{est-I110}, \eqref{est-I111}, \eqref{est-I112} and \eqref{est-I113}, we obtain
\begin{equation}\label{est-I11}
I_1\lesssim \frac{\la_*^{\nu-1}(t)}{1+|y|},
\end{equation}
where we write $y=\frac{x-q}{\la_*(t)}$ for simplicity.

\medskip

\noindent {\bf Estimate of $I_2$.}

\medskip

To estimate $I_2$, we change variable
\begin{equation*}
\tilde s=\frac{|x-z|}{(t-s)^{1/2}},
\end{equation*}
and thus
\begin{equation}\label{est-I220}
\begin{aligned}
I_2 \lesssim&~ \int_{\R^2} \int_{\frac{|x-z|}{T-t}}^{\infty} \frac{1}{(1+\tilde s)^3 |x-z|} \frac{\la_*^{\nu+a-1}(t)}{\la_*^{a+1}(t)+|z-q|^{a+1}}\,d\tilde s dz\\
\lesssim&~ \la_*^{\nu+a+1}(t) \int_{\R^2} \frac{1}{\la_*^{a+1}(t)+|z-q|^{a+1}}\frac{1}{|x-z|}\frac{1}{\la_*^2(t)+|x-z|^2}\,dz\\
\lesssim&~ \la_*^{\nu+a+1}(t) \left(\int_{D_1(x)}+\int_{D_2(x)}+\int_{D_3(x)}\right) \frac{1}{\la_*^{a+1}(t)+|z-q|^{a+1}}\frac{1}{|x-z|}\frac{1}{\la_*^2(t)+|x-z|^2}\,dz,
\end{aligned}
\end{equation}
where $D_1(x)$, $D_2(x)$ and $D_3(x)$ are defined in \eqref{def-DD1}, \eqref{def-DD2} and \eqref{def-DD3}, respectively. For the above integral, we consider the following two cases.
\begin{itemize}
\item Case 1: $|x-q|\leq \la_*(t)$. We have
\begin{equation}\label{est-I221}
\begin{aligned}
&\quad\int_{D_1(x)} \frac{1}{\la_*^{a+1}(t)+|z-q|^{a+1}}\frac{1}{|x-z|}\frac{1}{\la_*^2(t)+|x-z|^2}\,dz\\
&\lesssim\frac{1}{|x-q|(\la_*^2(t)+|x-q|^2)}\int_0^{\frac{|x-q|}{2}} \frac{r}{\la_*^{a+1}(t)+r^{a+1}}\,dr\\
&\lesssim\frac{\la_*^{-a-1}(t)}{\la_*(t)|x-q|^2}|x-q|^2\\
&\lesssim \la_*^{-a-2}(t),
\end{aligned}
\end{equation}

\begin{equation}\label{est-I222}
\begin{aligned}
&\quad\int_{D_2(x)} \frac{1}{\la_*^{a+1}(t)+|z-q|^{a+1}}\frac{1}{|x-z|}\frac{1}{\la_*^2(t)+|x-z|^2}\,dz\\
&\lesssim \frac{1}{\la_*^{a+1}(t)+|x-q|^{a+1}}\int_0^{3|x-q|} \frac{1}{\la_*^2(t)+r^2}\,dr\\
&\lesssim \la_*^{-a-2}(t),
\end{aligned}
\end{equation}
and
\begin{equation}\label{est-I223}
\begin{aligned}
&\quad\int_{D_3(x)} \frac{1}{\la_*^{a+1}(t)+|z-q|^{a+1}}\frac{1}{|x-z|}\frac{1}{\la_*^2(t)+|x-z|^2}\,dz\\
&\lesssim \int_{2|x-q|}^{\infty} \frac{1}{\la_*^{a+1}(t)+r^{a+1}} \frac{1}{r-|x-q|}\frac{1}{\la_*^2(t)+(r-|x-q|)^2}\,r\,dr\\
&\lesssim \int_{|x-q|}^{\infty} \frac{1}{\la_*^{a+1}(t)+(\tilde r +|x-q|)^{a+1}}\frac{1}{\tilde r} \frac{1}{\la_*^2(t)+\tilde r^2}(\tilde r +|x-q|)\,d\tilde r\\
&\lesssim \la_*^{-a-2}(t).
\end{aligned}
\end{equation}
Observe that in this case $|x-q|\leq \la_*(t)$ we have $1\lesssim \frac{1}{1+|y|^{\epsilon}}$ for $\epsilon>0$, where $y=\frac{x-q}{\la_*(t)}$. Therefore, for the case $|x-q|\leq \la_*(t)$, we conclude
\begin{equation}\label{est-I21}
I_2\lesssim \frac{\la_*^{\nu-1}(t)}{1+|y|^2}
\end{equation}
by \eqref{est-I220}--\eqref{est-I223}.

\medskip

\item Case 2: $|x-q|\geq \la_*(t)$. In this case, we compute
\begin{equation}\label{est-I224}
\begin{aligned}
&\quad\int_{D_1(x)} \frac{1}{\la_*^{a+1}(t)+|z-q|^{a+1}}\frac{1}{|x-z|}\frac{1}{\la_*^2(t)+|x-z|^2}\,dz\\
&\lesssim \frac{1}{\la_*^2(t)+|x-q|^2}\frac{1}{|x-q|}\int_0^{\frac{|x-q|}{2}} \frac{r}{\la_*^{a+1}(t)+r^{a+1}}\,dr\\
&\lesssim \frac{\la_*^{-a-2}(t)}{1+|y|^2},
\end{aligned}
\end{equation}

\begin{equation}\label{est-I225}
\begin{aligned}
&\quad\int_{D_2(x)} \frac{1}{\la_*^{a+1}(t)+|z-q|^{a+1}}\frac{1}{|x-z|}\frac{1}{\la_*^2(t)+|x-z|^2}\,dz\\
&\lesssim \frac{1}{\la_*^{a+1}(t)+|x-q|^{a+1}}\int_0^{3|x-q|} \frac{1}{\la_*^2(t)+r^2}\,dr\\
&\lesssim \frac{\la_*^{-a-2}(t)}{1+|y|^{a+1}},
\end{aligned}
\end{equation}
and
\begin{equation}\label{est-I226}
\begin{aligned}
&\quad\int_{D_3(x)} \frac{1}{\la_*^{a+1}(t)+|z-q|^{a+1}}\frac{1}{|x-z|}\frac{1}{\la_*^2(t)+|x-z|^2}\,dz\\
&\lesssim \frac{1}{\la_*^2(t)+|x-q|^2} \frac{1}{|x-q|}\int_{2|x-q|}^{\infty} \frac{r}{\la_*^{a+1}(t)+r^{a+1}}\,dr\\
&\lesssim \frac{1}{\la_*^2(t)+|x-q|^2}\frac{1}{\la_*(t)} \frac{1}{\la_*^{a-1}(t)+|x-q|^{a-1}}\\
&\lesssim \frac{\la_*^{-a-2}(t)}{1+|y|^{a+1}}.
\end{aligned}
\end{equation}
From \eqref{est-I220}, \eqref{est-I224}, \eqref{est-I225} and \eqref{est-I226}, one has
\begin{equation}\label{est-I22}
I_2\lesssim \frac{\la_*^{-a-2}(t)}{1+|y|^2}
\end{equation}
for the case $|x-q|\geq \la_*(t)$.
\end{itemize}

\medskip

In conclusion, we get
\begin{equation*}
|v_{in}(x,t)|\lesssim \|F\|_{S,\nu-2,a+1} \frac{\la_*^{\nu-1}(t)}{1+|y|}
\end{equation*}
from \eqref{est-vvvvvv}, \eqref{est-I11}, \eqref{est-I21} and \eqref{est-I22}.

\medskip

We now derive the estimate \eqref{est-P1} for $P_1$. Recall the representation formula for $P_1$: \[P_1(x,t)=\int_0^t\int_{\R^2}Q_j(x-z,t-s)\pd_{z_k}F_{jk}(z,s)\,dzds,\]
where $Q_j$ is given by
\[Q_j(x,t)=\frac{\de(t)}{2\pi}\,\frac{x_j}{|x|^2}.\]
Thus,
\EQN{
P_1(x,t)=&\int_{\R^2}\frac1{2\pi}\,\frac{x_j-z_j}{|x-z|^2}\,\pd_{z_k}F_{jk}(z,t)\,dz\\
=&\left(\int_{D_1(x)}+\int_{D_2(x)}+\int_{D_3(x)}\right)\frac1{2\pi}\,\frac{x_j-z_j}{|x-z|^2}\,\pd_{z_k}F_{jk}(z,t)\,dz\\
:=&~\textup{I}+\textup{II}+\textup{III}
} where $D_1(x)$, $D_2(x)$, and $D_3(x)$ are defined in \eqref{def-DD1}, \eqref{def-DD2}, and \eqref{def-DD3}, respectively.

\medskip

We perform integration by parts to estimate $\textup{I}$. In fact, one has
\EQ{\label{est-P1I}
\textup{I}
\lesssim&~\|F\|_{S,\nu-2,a+1}\left(\int_{D_1(x)}\frac1{|x-z|^2}\,\frac{\la_*^{\nu-2}(t)}{1+\left|\frac{z-q}{\la_*(t)}\right|^{a+1}}\,dz+\int_{\pd D_1(x)}\frac1{|x-z|}\,\frac{\la_*^{\nu-2}(t)}{1+\left|\frac{z-q}{\la_*(t)}\right|^{a+1}}\,dz\right)\\
\lesssim&~\|F\|_{S,\nu-2,a+1}\left(\frac{\la_*^{\nu-2}(t)}{|x-q|^2}\int_0^{\frac{|x-q|}2}\frac{\la_*^{a+1}(t)}{\la_*^{a+1}(t)+r^{a+1}}\,r\,dr+\frac1{|x-q|}\,\frac{\la_*^{\nu-2}}{1+\left|\frac{x-q}{\la_*(t)}\right|^{a+1}}\,|x-q|\right)\\
\lesssim&~\|F\|_{S,\nu-2,a+1}\left(\frac{\la_*^{\nu}(t)}{|x-q|^2}+\frac{\la_*^{\nu-2}}{1+\left|\frac{x-q}{\la_*(t)}\right|^{a+1}}\right).
}

The way to estimate $\textup{II}$ and $\textup{III}$ is straightforward. More specifically, we have
\EQ{\label{est-P1II}
\textup{II}=&~\frac1{2\pi}\int_{D_2(x)}\frac{x_j-z_j}{|x-z|^2}\,\pd_{z_k}F_{jk}(z,t)\,dz\\
\lesssim&~\|F\|_{S,\nu-2,a+1}\int_{D_2(x)}\frac1{|x-z|}\,\frac{\la_*^{\nu-3}(t)}{1+\left|\frac{z-q}{\la_*(t)}\right|^{a+2}}\,dz\\
\lesssim&~\|F\|_{S,\nu-2,a+1}\, \frac{\la_*^{\nu-3}(t)}{1+\left|\frac{x-q}{\la_*(t)}\right|^{a+2}} \int_0^{3|x-q|}\frac1r\,r\,dr\\
\lesssim&~\|F\|_{S,\nu-2,a+1}\frac{\la_*^{\nu-2}}{1+\left|\frac{x-q}{\la_*(t)}\right|^{a+1}},
}
and
\EQN{
\textup{III}=&~\frac1{2\pi}\int_{D_3(x)}\frac{x_j-z_j}{|x-z|^2}\, \pd_{z_k}F_{jk}(z,t)\,dz\\
\lesssim&~\|F\|_{S,\nu-2,a+1}\int_{D_3(x)}\frac1{|x-z|}\,\frac{\la_*^{\nu-3}(t)}{1+\left|\frac{z-q}{\la_*(t)}\right|^{a+2}}\,dz\\
\lesssim&~\|F\|_{S,\nu-2,a+1}\,\la_*^{\nu+a-1}(t)\int_{2|x-q|}^\infty\frac1{r-|x-q|}\,\frac{1}{\la_*^{a+2}(t)+r^{a+2}}\,r\,dr\\
=&~\|F\|_{S,\nu-2,a+1}\,\la_*^{\nu+a-1}(t)\int_{|x-q|}^\infty\frac1u\,\frac{1}{\la_*^{a+2}(t)+(u+|x-q|)^{a+2}}\,(u+|x-q|)\,du,
}
where we changed the variables $u=r-|x-q|$. Hence $u\ge|x-q|$ implies that
\EQ{\label{est-P1III}
\textup{III}\lesssim&~\|F\|_{S,\nu-2,a+1}\,\la_*^{\nu+a-1}(t)\int_{|x-q|}^\infty\frac{1}{\la_*^{a+2}(t)+(u+|x-q|)^{a+2}}\,du\\
\lesssim&~\|F\|_{S,\nu-2,a+1}\,\la_*^{\nu+a-1}(t)\,\frac1{\la_*^{a+1}(t)+|x-q|^{a+1}}\\
=&~\|F\|_{S,\nu-2,a+1}\frac{\la_*^{\nu-2}}{1+\left|\frac{x-q}{\la_*(t)}\right|^{a+1}}.
}
Collecting \eqref{est-P1I}, \eqref{est-P1II}, and \eqref{est-P1III}, we obtain the estimate \eqref{est-P1}, and the proof is complete.
\end{proof}

In order to apply $W_p^{2,1}$-theory of the Stokes system to the outer part \eqref{eqn-vout}, the estimates for $\nabla v_{in}$ and $\partial_t (v_{in}\cdot \nabla \eta_{\delta})$ are further needed. We have the following lemma.

\begin{lemma}
Under the assumptions of Lemma \ref{apriori-SSin}, the following estimates hold
\begin{equation}\label{est-gradvin}
|\nabla_x v_{in}(x,t)|\lesssim \|F\|_{S,\nu-2,a+1} \frac{\la_*^{\nu-2}(t)}{1+\left|\frac{x-q}{\la_*(t)}\right|},
\end{equation}
and
\begin{equation}\label{est-dtvin}
\|\partial_t (v_{in}\cdot \nabla \eta_{\delta})\|_{L^p\left((B_{2\delta}(q)\setminus B_{\delta}(q))\times(0,T)\right)} \lesssim \|F\|_{S,\nu-2,a+1}
\end{equation}
for $(\nu-1)p+1>0$.
\end{lemma}

\begin{proof}
Since we impose zero initial condition on $v_{in}$, we have
\begin{equation*}
\begin{aligned}
|\partial_{x_l} v_i(x,t)|\lesssim &~ \|F\|_{S,\nu-2,a+1}\int_0^t \int_{\R^2} \frac{1}{\left(|x-z|^2+(t-s)\right)^{3/2}}\frac{\la_*^{\nu-3}(s)}{1+\left|\frac{z-q}{\la_*(s)}\right|^{a+2}}\,dz ds\\
\end{aligned}
\end{equation*}
where we have used \eqref{est-Oseentensor}. We decompose the above integral and first estimate
\begin{equation*}
\begin{aligned}
&\quad \int_0^{t-(T-t)^2} \int_{\R^2} \frac{1}{\left(|x-z|^2+(t-s)\right)^{3/2}}\frac{\la_*^{\nu-3}(s)}{1+\left|\frac{z-q}{\la_*(s)}\right|^{a+2}}\,dz ds\\
&\lesssim \la_*^{\nu+a-1}(t)\left(\int_{D_1(x)}+\int_{D_2(x)}+\int_{D_3(x)}\right) \frac{1}{\la_*^{a+2}(t)+|z-q|^{a+2}}\frac{1}{|x-z|+\la_*(t)}\, dz,
\end{aligned}
\end{equation*}
where $D_1(x)$, $D_2(x)$ and $D_3(x)$ are defined in \eqref{def-DD1}, \eqref{def-DD2} and \eqref{def-DD3}, respectively. Then we can easily check the following
\begin{equation*}
\begin{aligned}
\int_{D_1(x)} \frac{1}{\la_*^{a+2}(t)+|z-q|^{a+2}}\frac{1}{|x-z|+\la_*(t)}\, dz \lesssim&~ \frac{\la_*^{-a}(t)}{|x-q|+\la_*(t)}
\\
\int_{D_2(x)} \frac{1}{\la_*^{a+2}(t)+|z-q|^{a+2}}\frac{1}{|x-z|+\la_*(t)}\, dz \lesssim&~ \frac{1}{|x-q|^{a+1}+\la_*^{a+1}(t)}
\\
\int_{D_3(x)} \frac{1}{\la_*^{a+2}(t)+|z-q|^{a+2}}\frac{1}{|x-z|+\la_*(t)}\, dz \lesssim&~ \frac{1}{|x-q|^{a+1}+\la_*^{a+1}(t)}
\end{aligned}
\end{equation*}
and thus
\begin{equation*}
\int_0^{t-(T-t)^2} \int_{\R^2} \frac{1}{\left(|x-z|^2+(t-s)\right)^{3/2}}\frac{\la_*^{\nu-3}(s)}{1+\left|\frac{z-q}{\la_*(s)}\right|^{a+2}}\,dz ds \lesssim \frac{\la_*^{\nu-2}(t)}{1+|y|},
\end{equation*}
where we write $y=\frac{x-q}{\la_*}$. For the other part, we have
\begin{equation*}
\begin{aligned}
&\quad \int_{t-(T-t)^2}^t \int_{\R^2} \frac{1}{\left(|x-z|^2+(t-s)\right)^{3/2}}\frac{\la_*^{\nu-3}(s)}{1+\left|\frac{z-q}{\la_*(s)}\right|^{a+2}}\,dz ds\\
&\lesssim \int_{\R^2} \int_{\frac{|x-z|}{T-t}}^{\infty} \frac{1}{(1+\tilde s)^3 |x-z|} \frac{\la_*^{\nu+a-1}(t)}{\la_*^{a+2}(t)+|z-q|^{a+2}}\,d\tilde s dz\\
&\lesssim \la_*^{\nu+a+1}(t) \int_{\R^2} \frac{1}{\la_*^{a+2}(t)+|z-q|^{a+2}}\frac{1}{|x-z|}\frac{1}{\la_*^2(t)+|x-z|^2}\,dz,\\
\end{aligned}
\end{equation*}
where we have changed variable $\tilde s=\frac{|x-z|}{\sqrt{t-s}}$. Similar to the proof of Lemma \ref{apriori-SSin}, the following bound holds
\begin{equation*}
\int_{t-(T-t)^2}^t \int_{\R^2} \frac{1}{\left(|x-z|^2+(t-s)\right)^{3/2}}\frac{\la_*^{\nu-3}(s)}{1+\left|\frac{z-q}{\la_*(s)}\right|^{a+2}}\,dz ds \lesssim \frac{\la_*^{\nu-2}(t)}{1+|y|^2}.
\end{equation*}
Collecting the above estimates, we conclude the validity of \eqref{est-gradvin}.

\medskip

Next we prove \eqref{est-dtvin}. Multiplying equation \eqref{eqn-vin} by $\nabla \eta_{\delta}$, we obtain that $v_{in}\cdot \nabla \eta_{\delta}$ satisfies the equation
\begin{equation*}
\partial_t (v_{in}\cdot \nabla \eta_{\delta})=\Delta (v_{in}\cdot \nabla \eta_{\delta})-\Delta(\nabla \eta_{\delta})\cdot v_{in}-2\nabla^2 \eta_{\delta}\cdot \nabla v_{in}-\nabla P_1\cdot \nabla \eta_{\delta}+(\nabla\cdot F_{in})\cdot \nabla \eta_{\delta}.
\end{equation*}
Thanks to the cut-off function $\eta_{\delta}$, standard $W^{2,1}_p$-theory for parabolic equation yields
\begin{equation}\label{est-dtvin111}
\begin{aligned}
&~\quad \|\partial_t (v_{in}\cdot \nabla \eta_{\delta})\|_{L^p\left((B_{2\delta}(q)\setminus B_{\delta}(q))\times(0,T)\right)}\\
&\lesssim \|v_{in}\|_{L^p\left((B_{2\delta}(q)\setminus B_{\delta}(q))\times(0,T)\right)} + \|\nabla v_{in}\|_{L^p\left((B_{2\delta}(q)\setminus B_{\delta}(q))\times(0,T)\right)}\\
&~\quad+ \|\nabla P_1\|_{L^p\left((B_{2\delta}(q)\setminus B_{\delta}(q))\times(0,T)\right)}+ \|\nabla\cdot F\|_{L^p\left((B_{2\delta}(q)\setminus B_{\delta}(q))\times(0,T)\right)}.
\end{aligned}
\end{equation}
Using the $W^{2,1}_p$-theory for the Stokes system (see \cite{SoloRMS2003} for instance), we readily see that
\begin{equation}\label{est-dtvin222}
\|\nabla P_1\|_{L^p\left((B_{2\delta}(q)\setminus B_{\delta}(q))\times(0,T)\right)}\lesssim  \|\nabla\cdot F\|_{L^p\left((B_{2\delta}(q)\setminus B_{\delta}(q))\times(0,T)\right)}.
\end{equation}
From \eqref{est-dtvin111}, \eqref{est-dtvin222}, \eqref{est-vin}, \eqref{est-gradvin} and the assumption $\|F\|_{S,\nu-2,a+1}<+\infty$, we obtain
\begin{equation*}
\|\partial_t (v_{in}\cdot \nabla \eta_{\delta})\|_{L^p\left((B_{2\delta}(q)\setminus B_{\delta}(q))\times(0,T)\right)} \lesssim \|F\|_{S,\nu-2,a+1}
\end{equation*}
provided $(\nu-1)p+1>0$. The proof is complete.
\end{proof}

We are ready to estimate the outer part \eqref{eqn-vout}.
\begin{lemma}\label{apriori-SSout}
For $\|F\|_{S,\nu-2,a+1}<+\infty$ and $\|v_0\|_{B^{2-2/p}_{p,p}}<+\infty$, the solution $(v_{out},P)$ of the system \eqref{eqn-vout} satisfies
\begin{equation}\label{est-vout-W21}
\|v_{out}\|_{W^{2,1}_p(\Omega\times(0,T))}+\|\na(P-\eta_\de P_1)\|_{L^p(\Omega\times(0,T))}\lesssim\|F\|_{S,\nu-2,a+1}+\|v_0\|_{B^{2-2/p}_{p,p}}
\end{equation}
for $(\nu-1)p+1>0$. If we further assume $\nu\in(1/2,1)$, then we have
\begin{equation}\label{est-vout-holder}
\|v_{out}\|_{C^{\al,\al/2}(\Omega\times(0,T))}\lesssim\|F\|_{S,\nu-2,a+1}+\|v_0\|_{B^{2-2/p}_{p,p}}
\end{equation}
for $0<\al\le2-4/p$.
\end{lemma}
\begin{proof}
The $W^{2,1}_p$ estimate of solutions to Stokes system with non-zero divergence derived in \cite[Theorem 3.1]{SoloRMS2003} shows that
\EQ{\label{est-solo2003}
&~\quad\|v_{out}\|_{W^{2,1}_p(\Omega\times(0,T))}+\|\na(P-\eta_\de P_1)\|_{L^p(\Omega\times(0,T))}\\
&\lesssim\|(1-\eta_\de)\na\cdot F+2\na\eta_\de\cdot\na v_{in}+(\Delta\eta_\de)v_{in}-P_1\na\eta_\de\|_{L^p(\Omega\times(0,T))}+\|\na\eta_\de\cdot v_{in}\|_{L^p(0,T;W^1_p(\Omega))}\\
&~\quad+\|\pd_t(\na\eta_\de\cdot v_{in})\|_{L^p(0,T;W^{-1}_p(\Omega))}+\|v_0\|_{B^{2-2/p}_{p,p}},
}
where $\|\cdot\|_{B^{2-2/p}_{p,p}}$ is the Besov norm defined in \eqref{besov}. Thanks to the cut-off function $\eta_{\delta}$, we get
\[\left|(1-\eta_\de)\na\cdot F \right|\lesssim\|F\|_{S,\nu-2,a+1}\la_*^{\nu+a-1},\]
and from \eqref{est-vin}, \eqref{est-P1}, \eqref{est-gradvin} and \eqref{est-dtvin}, one has
\[\left|\na\eta_\de\cdot \na v_{in}\right|+\left|(\De\eta_\de)v_{in}\right|+\left|P_1\na\eta_\de\right|\lesssim\|F\|_{S,\nu-2,a+1}\la_*^{\nu-1},\]
and also
\begin{align*}
\left|\na\eta_\de\cdot v_{in}\right|\lesssim&~\|F\|_{S,\nu-2,a+1}\la_*^\nu,\\
\|\pd_t(\na\eta_\de\cdot v_{in})\|_{L^p(0,T;W^{-1}_p(\Omega))}\lesssim&~ \|F\|_{S,\nu-2,a+1}.
\end{align*}
It is worth noting that $\|\cdot\|_{L^p(0,T;W^{-1}_p(\Omega))}\le\|\cdot\|_{L^p(0,T;L^p(\Omega))}$ (see \cite{UBCAdams} for instance). Therefore, estimate \eqref{est-solo2003} together with the above bounds imply \eqref{est-vout-W21} for $(\nu-1)p+1>0$.
The H\"{o}lder estimate \eqref{est-vout-holder} then follows from a standard parabolic version of Morrey type inequality (see \cite{Lieberman} for instance). The proof is complete.
\end{proof}

The proof of Proposition \ref{prop-SS} is a direct consequence of Lemma \ref{apriori-SSin} and Lemma \ref{apriori-SSout}.

\medskip

For the behavior of the velocity field $v$, we further make several remarks:

\begin{remark}\label{rmk-SS}
\noindent
\begin{itemize}
\item From \eqref{def-normSSF}, Proposition \ref{prop-SS} implies
\begin{equation*}
\|v\|_{S,\nu-1,1}\lesssim \|F\|_{S,\nu-2,a+1}.
\end{equation*}

\item Since $v$ is divergence-free, we can write $v\cdot \nabla v=\nabla\cdot(v\otimes v)$, where $\otimes$ is the tensor product defined by $(v\otimes w)_{ij}=v_iw_j$. If we solve $v$ in the class $\|v\|_{S,\nu-1,1}<\infty$, then the nonlinearity in the Navier--Stokes equation
 $$|v\cdot \nabla v|\lesssim \frac{\la_*^{2\nu-3}(t)}{1+\left|\frac{x-q}{\la_*(t)}\right|^3}$$
is indeed a perturbation compared to $\nabla\cdot F$, which enables us to solve $v$ by the fixed point argument in Section \ref{sec-LCF}.

\item The initial velocity $v_0$ in the outer problem \eqref{eqn-vout} can be chosen arbitrarily in the Besov space $B^{2-2/p}_{p,p}$ in which the norm is defined by
\EQ{\label{besov}
&\quad\|v_0\|_{B^{2-2/p}_{p,p}}\\
&:=\left(\int_{|z|<1}|z|^{-2p}\int_{\Omega(z)}|v_0(x+2z)-v_0(x+z)+v_0(x)|^pdxdz\right)^{1/p}+\|v_0\|_{L^p(\Omega)},
}
where $\Omega(z)=\{x\in\Omega:x+tz\in\Omega,\,t\in[0,1]\}$, as long as it agrees with zero at the boundary and satisfies the condition
\[\na\cdot v_0=-\nabla \eta_{\delta}\cdot v_{in}(x,0)=0.\]
To ensure the non-triviality of velocity field, i.e, $v\not\equiv0$, we choose a non-trivial solenoidal (divergence-free) initial velocity $v_0\in B^{2-2/p}_{p,p}$, where $p>1/\nu$.
\end{itemize}
\end{remark}

\medskip


\section{Solving the nematic liquid crystal flow}\label{sec-LCF}

\medskip

In this section, we shall apply  the linear theories developed in Section \ref{sec-HMF} and Section \ref{sec-SS} to show the existence of the desired blow-up solution to \eqref{LCF}--\eqref{BCs} by means of the fixed point argument. Apriori we need some assumptions on the behavior of the parameter functions $p(t)=\la(t) e^{i\omega(t)}$ and $\xi(t)$
\begin{align*}
c_1 |\dot\la_*(t)|\leq |\dot p(t)|\leq&~ c_2|\dot\la_*(t)| ~\mbox{ for all }~t\in(0,T),\\
|\dot\xi(t)|\leq&~ \la^{\sigma}_*(t)~\mbox{ for all }~t\in(0,T),
\end{align*}
where $c_1$, $c_2$ and $\sigma$ are some positive constants independent of $T$. We recall that
$$R=R(t)=\la_*^{-\gamma_*}(t)~\mbox{ with }~\la_*(t)=\frac{|\log T|(T-t)}{|\log(T-t)|^2}~\mbox{ and }~\gamma_*\in(0,1/2).$$

Similar to the harmonic map heat flow, we look for solution $u$ solving problem \eqref{LCF} in the form
\begin{equation*}
u=U+\Pi_{U^{\perp}}\varphi +a(\Pi_{U^{\perp}}\varphi) U,
\end{equation*}
with
\begin{equation*}
\varphi=\eta_R Q_{\omega,\alpha,\beta} \phi(y,t) +\Psi^*(x,t)+\Phi^0(x,t)+\Phi^{\alpha}(x,t)+\Phi^{\beta}(x,t),
\end{equation*}
where we decompose $\Psi^*$ into
$$\Psi^*=Z^*+\psi.$$
Here $Z^*$ satisfies
\begin{equation*}
\begin{cases}
\partial_t Z^*=\Delta Z^*,~&\mbox{ in }\Omega\times(0,\infty)\\
Z^*(\cdot,t)=0,~&\mbox{ on }\partial\Omega\times(0,\infty)\\
Z^*(\cdot,0)=Z^*_0,~&\mbox{ in }\Omega\\
\end{cases}
\end{equation*}
For the same technical reasons as shown in \cite{17HMF}, we make some assumptions on $Z_0^*(x)$ as follows.
Let us write
\begin{align*}
Z_0^*(x) =  \left [ \begin{matrix}z_0^*(x) \\  z_{03}^* (x) \end{matrix}   \right ] , \quad z_0^*(x) = z^*_{01}(x)  + i z^*_{02}(x)  .
\end{align*}
Consistent with \eqref{negativeDiv}, the first condition that we need is
\begin{equation*}
 \div  z^*_0(q) <0.
 \end{equation*}
In addition, we require that $Z_0^*(q)\approx 0 $ in a non-degenerate way.

We will get a desired solution $(v,u)$ to problem \eqref{LCF} if $(v,\phi,\Psi^*,p,\xi,\alpha,\beta)$ solves the following {\em inner--outer gluing system}
\begin{equation}\label{eqn-NS}
\left\{
\begin{aligned}
&\partial_t v+v\cdot \nabla v+\nabla P=\Delta v -\epsilon_0\nabla\cdot \mathcal F[p,\xi,\alpha,\beta,\Psi^*,\phi,v],~\mbox{ in }~\Omega\times(0,T),\\
&\nabla\cdot v=0,~\mbox{ in }~\Omega\times(0,T),\\
&v=0,~\mbox{ on }~\partial\Omega\times(0,T),\\
&v(\cdot,0)=v_0,~\mbox{ in }~\Omega,\\
\end{aligned}
\right.
\end{equation}
\begin{equation}\label{eqn-dinner}
\left\{
\begin{aligned}
&\la^2\partial_t \phi= L_W[\phi]+\mathcal H[p,\xi,\alpha,\beta,\Psi^*,\phi,v],~\mbox{ in }~\mathcal D_{2R},\\
&\phi(\cdot,0)=0,~\mbox{ in }~B_{2R(0)},\\
&\phi\cdot W=0,~\mbox{ in }~\mathcal D_{2R},\\
\end{aligned}
\right.
\end{equation}
\begin{equation}\label{eqn-douter}
\left\{
\begin{aligned}
&\partial_t \Psi^*=\Delta_x\Psi^*+\mathcal G[p,\xi,\alpha,\beta,\Psi^*,\phi,v]~\mbox{ in }~\Omega\times(0,T),\\
&\Psi^*={\bf e_3}-U-\Phi^0-\Phi^{\alpha}-\Phi^{\beta}~\mbox{ on }~\pd\Omega\times(0,T),\\
&\Psi^*(\cdot,0)=(1-\chi)\left({\bf e_3}-U-\Phi^0-\Phi^{\alpha}-\Phi^{\beta}\right)~\mbox{ in }~\Omega,\\
\end{aligned}
\right.
\end{equation}
where
\begin{equation}\label{def-mF}
\mathcal F[p,\xi,\alpha,\beta,\Psi^*,\phi,v]=\left(\nabla u\odot \nabla u-\frac12|\nabla u|^2\mathbb I_2\right),
\end{equation}
with
\begin{equation*}
u=U+\Pi_{U^{\perp}}[\eta_R Q_{\omega,\alpha,\beta} \phi +\Psi^*+\Phi^0+\Phi^{\alpha}+\Phi^{\beta}] +a(\Pi_{U^{\perp}}[\eta_R Q_{\omega,\alpha,\beta} \phi +\Psi^*+\Phi^0+\Phi^{\alpha}+\Phi^{\beta}]) U,
\end{equation*}
\begin{equation*}
\begin{aligned}
\mathcal H[p,\xi,\alpha,\beta,\Psi^*,\phi,v]=&~
\la^2 Q^{-1}_{\omega,\alpha,\beta}\Bigg[\tilde L_U[\Psi^*]+\mathcal K_0[p,\xi]+\mathcal K_1[p,\xi]+\Pi_{U^{\perp}}[\RR_{-1}]-\la^{-1}\Pi_{U^{\perp}}(v \cdot \nabla_y U)\\
&~-\la^{-1}\Pi_{U^{\perp}}\left(v \cdot \nabla_y \left(\Pi_{U^{\perp}}[\eta_R Q_{\omega,\alpha,\beta} \phi  +\Psi^* +\Phi^0 +\Phi^{\alpha} +\Phi^{\beta}]\right)\right)\\
&~-\la^{-1}\Pi_{U^{\perp}}\left(v \cdot \nabla_y \left(a(\Pi_{U^{\perp}}[\eta_R Q_{\omega,\alpha,\beta} \phi  +\Psi^* +\Phi^0 +\Phi^{\alpha} +\Phi^{\beta}]) U\right)\right)\Bigg],
\end{aligned}
\end{equation*}
and
\begin{equation*}
\begin{aligned}
\mathcal G[p,\xi,\alpha,\beta,\Psi^*,\phi,v]:=&~(1-\eta_R) \tilde L_U[\Psi^*]+(\Psi^*\cdot U) U_t+Q_{\omega,\alpha,\beta}(\phi\Delta_x \eta_R+2\nabla_x \eta_R\cdot \nabla_x \phi-\phi\partial_t \eta_R)\\
&~+\eta_R Q_{\omega,\alpha,\beta}\left(-\left(Q_{\omega,\alpha,\beta}^{-1}\frac{d}{dt}Q_{\omega,\alpha,\beta}\right)\phi+\la^{-1}\dot\la y\cdot \nabla_y \phi +\la^{-1}\dot\xi\cdot \nabla_y \phi\right)\\
&~+(1-\eta_R)(\mathcal K_0[p,\xi]+\mathcal K_1[p,\xi]+\Pi_{U^{\perp}}[\RR_{-1}])-\Pi_{U^{\perp}}[\tilde{\mathcal R}_1]\\
&~+N_U[\eta_R Q_{\omega,\alpha,\beta} \phi+\Pi_{U^{\perp}}(\Phi^0+\Phi^{\alpha}+\Phi^{\beta}+\Psi^*)]\\
&~+\left((\Phi^0+\Phi^{\alpha}+\Phi^{\beta})\cdot U\right) U_t-(1-\eta_R)v\cdot \nabla U\\
&~-(1-\eta_R) v\cdot \nabla\left( \Pi_{U^{\perp}}\left[\eta_R Q_{\omega,\alpha,\beta} \phi  +\Psi^* +\Phi^0 +\Phi^{\alpha} +\Phi^{\beta} \right]\right)\\
&~-(1-\eta_R) v\cdot \nabla\left( a\left(\Pi_{U^{\perp}}\left[\eta_R Q_{\omega,\alpha,\beta} \phi  +\Psi^* +\Phi^0 +\Phi^{\alpha} +\Phi^{\beta} \right]\right)U\right).\\
\end{aligned}
\end{equation*}
Here $\chi$ in \eqref{eqn-douter} is a smooth cut-off function which is supported near a fixed neighborhood of $q$ independent of $T$.

As discussed in Section \ref{sec-ltinner}, suitable inner solution with space-time decay can be obtained under certain orthogonality conditions, which will be achieved by adjusting the parameter functions $p(t)$, $\xi(t)$, $\alpha(t)$ and $\beta(t)$. In order to solve the inner problem \eqref{eqn-dinner}, we further decompose it based on the Fourier modes
$$\mathcal H=\mathcal H_1+\mathcal H_2+\mathcal H_3+\mathcal H_4,$$
with
\begin{equation*}
\begin{aligned}
\mathcal H_1[p,\xi,\alpha,\beta,\Psi^*,\phi,v]
&=\left(
\lambda^2  Q^{-1}_{\omega,\alpha,\beta} \left(
\tilde L_U  [\Psi^* ]_0
+\tilde L_U  [\Psi^* ]_2 +\KK_{0}[p,\xi]\right)
 +\lambda  Q^{-1}_{\omega,\alpha,\beta} [\Pi_{U^{\perp}}(v\cdot\nabla u)]_0\right)\chi_{\mathcal D_{2R} },
\\
\mathcal H_2[p,\xi,\alpha,\beta,\Psi^*,\phi,v] &=\left(
\lambda^2  Q^{-1}_{\omega,\alpha,\beta} \left(\tilde L_U  [\Psi^* ]_1^{(0)}
+\KK_{1}[p,\xi]\right)\right.
\\
&\qquad\left.+\lambda  Q^{-1}_{\omega,\alpha,\beta}\left([\Pi_{U^{\perp}}(v\cdot\nabla u)]_{1}+ [\Pi_{U^{\perp}}(v\cdot\nabla u)]_{\perp}\right)\right) \chi_{\mathcal D_{2R} },
\\
\mathcal H_3[p,\xi,\alpha,\beta,\Psi^*,\phi,v] &= \lambda^2  Q^{-1}_{\omega,\alpha,\beta}
\left( \tilde L_U [\Psi^* ]_1 - \tilde L_U  [\Psi^* ]_1 ^{(0)}\right) \chi_{\mathcal D_{2R} },\\
\mathcal H_4[p,\xi,\alpha,\beta,\Psi^*,\phi,v] &=  \left( \lambda^2  Q^{-1}_{\omega,\alpha,\beta}
(\Pi_{U^{\perp}}[\RR_{-1,1}]+\Pi_{U^{\perp}}[\RR_{-1,2}])+\lambda  Q^{-1}_{\omega,\alpha,\beta} [\Pi_{U^{\perp}}(v\cdot\nabla u)]_{-1}\right) \chi_{\mathcal D_{2R} },
\end{aligned}
\end{equation*}
where $[\Pi_{U^\perp}(v\cdot\nabla u)]_{0}$, $[\Pi_{U^\perp}(v\cdot\nabla u)]_{-1}$, $[\Pi_{U^\perp}(v\cdot\nabla u)]_{1}$ and $[\Pi_{U^\perp}(v\cdot\nabla u)]_{\perp}$ correspond respectively to modes $0$, $-1$, $1$ and higher modes $k\geq 2$ defined in \eqref{def-Fourier1}--\eqref{def-Fourier3}, and
 \begin{align}
\nonumber
\tilde L_U [\Phi ]_1^{(0)} & =
   -\,
  2\la^{-1} w_\rho  \cos w \, \big [\,(\partial_{x_1} \varphi_3(\xi(t),t)) \cos \theta +    (\partial_{x_2} \varphi_3(\xi(t),t))) \sin  \theta \, \big ]\,Q_{\omega,\alpha,\beta} E_1
\\
\nonumber
& \quad - 2\la^{-1}  w_\rho  \cos w  \, \big [\, (\partial_{x_1} \varphi_3(\xi(t),t))) \sin \theta -    (\partial_{x_2} \varphi_3(\xi(t),t))) \cos  \theta \, \big ]\, Q_{\omega,\alpha,\beta}E_2
\end{align}
in the notation \eqref{notanota}.
Then by decomposing $\phi=\phi_1+\phi_2+\phi_3+\phi_4$ in a similar manner as $\mathcal H_i$'s, the inner problem \eqref{eqn-dinner} becomes
\begin{equation}
\left\{
\begin{aligned}
\lambda^2  \partial_t \phi_1
&= L_W [\phi_1] + \mathcal H_1[p,\xi,\alpha,\beta,\Psi^*,\phi,v]
-  \sum_{  j=1,2} \tilde c_{0j}[ \mathcal H_1[p,\xi,\alpha,\beta,\Psi^*,\phi,v]] w_\rho^2 Z_{0,j}
\\
\label{eqphi1}
& \qquad
-  \sum_{  j=1,2} c_{1j}[ \mathcal H_1[p,\xi,\alpha,\beta,\Psi^*,\phi,v]  ] w_\rho^2  Z_{1,j}
\inn \mathcal D_{2R}
\\
\phi_1\cdot  W  &=  0   \inn \mathcal D_{2R}
\\
\phi_1(\cdot, 0) &=0 \inn B_{2R(0)}
\end{aligned}
\right.
\end{equation}

\begin{align}
\left\{
\begin{aligned}
\label{eqphi2}
\lambda^2 \partial_t \phi_2
&= L_W [\phi _2]  + \mathcal H_2[p,\xi,\alpha,\beta,\Psi^*,\phi,v]
 -  \sum_{  j=1,2} c_{1j}[  \mathcal H_2[p,\xi,\alpha,\beta,\Psi^*,\phi,v]    ] w_\rho^2  Z_{1,j}    \inn \mathcal D_{2R}
\\
\phi_2\cdot  W  & = 0   \inn \mathcal D_{2R}
\\
\phi_2(\cdot, 0) & =0 \inn B_{2R(0)}
\end{aligned}
\right.
\end{align}

\begin{align}
\label{eqphi3}
\left\{
\begin{aligned}
\lambda^2  \partial_t  \phi_3 &= L_W [\phi_3] +
\mathcal H_3[p,\xi,\alpha,\beta,\Psi^*,\phi,v]  -  \sum_{  j=1,2} c_{1j}[  \mathcal H_3[p,\xi,\alpha,\beta,\Psi^*,\phi,v]   ] w_\rho^2 Z_{1,j}
\\
& \quad
+ \sum_{j=1,2} c_{0j}^*[p,\xi,\alpha,\beta,\Psi^*,\phi,v] w_\rho^2 Z_{0,j} \inn \mathcal D_{2R}
\\
\phi_3\cdot  W  & = 0   \inn \mathcal D_{2R}
\\
\phi_3(\cdot, 0) & =0 \inn B_{2R(0)}
\end{aligned}
\right.
\end{align}

\begin{align}
\label{eqphi4}
\left\{
\begin{aligned}
\lambda^2  \partial_t  \phi_4 &= L_W [\phi_4 ] +
\mathcal H_4[p,\xi,\alpha,\beta,\Psi^*,\phi,v]-  \sum_{j=1,2}  c_{-1,j}[   \mathcal H_4[p,\xi,\alpha,\beta,\Psi^*,\phi,v]  ] w_\rho^2 Z_{-1,j}
\\
\phi_4\cdot  W  & = 0   \inn \mathcal D_{2R}
\\
\phi_4(\cdot, 0) & =0 \inn B_{2R(0)}
\end{aligned}
\right.
\end{align}

\begin{align}
\label{redueqn-m0}
c^*_{0j}(t) - \tilde c_{0j} (t) &= 0 ~\mbox{ for all }~ t\in (0,T), \quad j=1,2,
\\
\label{redueqn-m1}
c_{1j}(t)  &=  0 ~\mbox{ for all }~ t\in (0,T), \quad j=1,2,
\\
\label{redueqn-m-1}
c_{-1,j}(t)  &=  0 ~\mbox{ for all }~ t\in (0,T), \quad j=1,2.
\end{align}

\medskip

Based on the linear theory developed in Section \ref{sec-ltinner}, we shall solve the inner problems \eqref{eqphi1}--\eqref{eqphi4} in the  norms below.

\medskip

\begin{itemize}
\item We use the norm $\|\cdot\|_{\nu_i,a_i}$ to measure the right hand side $\mathcal H_i$ with $i=1,\cdots,4$, where
\begin{align}
\label{norm-h}
\|h\|_{\nu_i,a_i}  =
\sup_{\R^2 \times (0,T)} \  \frac{ |h(y,t)| }{ \lambda_*^{\nu_i}(t) (1+|y|)^{-a_i}}
\end{align}
with $\nu_i>0$, $a_i\in(2,3)$ for $i=1,\,2,\,4$, and $a_3\in (1,3)$.

\item We use the norm $\|\cdot\|_{*,\nu_1,a_1,\delta}$ to measure the solution $\phi_1$ solving \eqref{eqphi1}, where
\begin{align*}
\| \phi \|_{*,\nu_1,a_1,\delta}
=
\sup_{\mathcal D_{2R}}
\frac{| \phi(y,t) | + (1+|y|) |\nabla_y \phi(y,t)|+(1+|y|)^2 |\nabla^2_y \phi(y,t)|}{ \lambda_*^{\nu_1}(t) \max\left\{\frac{R^{\delta(5-a_1)}}{(1+|y|)^3} , \frac{1}{(1+|y|)^{a_1-2} }\right\}}
\end{align*}
with $\nu_1\in(0,1)$, $a_1\in(2,3)$, $\delta>0$ fixed small.

\item We use the norm $\|\cdot\|_{{\rm in},\nu_2,a_2-2}$ to measure the solution $\phi_2$ solving \eqref{eqphi2}, where
$$
\|\phi\|_{{\rm in},\nu_2,a_2-2}  =
\sup_{\mathcal D_{2R}} \  \frac{| \phi(y,t) | + (1+|y|) |\nabla_y \phi(y,t)|+(1+|y|)^2 |\nabla^2_y \phi(y,t)|}{ \lambda_*^{\nu_2}(t) (1+|y|)^{2-a_2}}
$$
with $\nu_2\in(0,1)$, $a_2\in(2,3)$.

\item We use the norm $\|\cdot\|_{**,\nu_3}$ to measure the solution $\phi_3$ solving \eqref{eqphi3}, where
\begin{align*}
\|\phi\|_{**,\nu_3}
= \sup_{\mathcal D_{2R}} \
\frac{ |\phi(y,t)| + (1+|y|)\left |\nabla_y \phi(y,t)\right | +(1+|y|)^2 |\nabla^2_y \phi(y,t)|}
{ \la^{\nu_3}_*(t)  R^{2}(t) ( 1+|y| )^{-1}  }
\end{align*}
with $\nu_3>0$.

\item We use the norm $\|\cdot\|_{***,\nu_4}$ to measure the solution $\phi_4$ solving \eqref{eqphi4}, where
\begin{align}
\nonumber
\|\phi\|_{***,\nu_4}
= \sup_{\mathcal D_{2R}} \
\frac{ |\phi(y,t)| + (1+|y|)\left |\nabla_y \phi(y,t)\right |+(1+|y|)^2 |\nabla^2_y \phi(y,t)| }
{ \la^{\nu_4}_*(t)  }
\end{align}
with $\nu_4>0$.

\end{itemize}

\medskip

Based on the linear theory in Section \ref{sec-ltouter}, we shall solve the outer problem \eqref{eqn-douter} in the following norms.

\medskip

\begin{itemize}
\item We use the norm $\|\cdot\|_{**}$ defined in \eqref{defNormRHSpsi} to measure the right hand side $\mathcal G$ in the outer problem \eqref{eqn-douter}.

\item We use the norm $\| \cdot\|_{\sharp, \Theta,\gamma}$ defined in \eqref{normPsi} to measure the solution $\psi$ solving the outer problem \eqref{eqn-douter}, where $\Theta>0$ and $\gamma\in(0,1/2)$.
\end{itemize}

\medskip

Based on the linear theory developed in Section \ref{sec-SS}, we shall solve the incompressible Navier--Stokes equation \eqref{eqn-NS} in the following norms.

\medskip

\begin{itemize}
\item We use the norm $\|\cdot\|_{S,\nu-2,a+1}$ defined in \eqref{def-normSSF} to measure the forcing $\mathcal F$, where $\nu>0$ and $a\in(1,2)$.

\item We use the norm $\|\cdot\|_{S,\nu-1,1}$ defined in \eqref{def-normSSF} to measure the velocity field $v$ solving problem \eqref{eqn-NS}, where $\nu>0$.
\end{itemize}

\medskip

We then define
\begin{align*}
E_1 &= \{ \phi_1 \in L^\infty(\mathcal D_{2R}) :
\nabla_y \phi_1  \in L^\infty(\mathcal D_{2R}), \
\|\phi_1\|_{*,\nu_1,a_1,\delta} <\infty \}
\\
E_2 &= \{ \phi_2 \in L^\infty(\mathcal D_{2R}) :
\nabla_y \phi_2  \in L^\infty(\mathcal D_{2R}), \
\|\phi_2\|_{{\rm in},\nu_2,a_2-2}<\infty \}
\\
E_3 &= \{ \phi_3 \in L^\infty(\mathcal D_{2R}) :
\nabla_y \phi_3  \in L^\infty(\mathcal D_{2R}), \
\|\phi_3\|_{**,\nu_3} < \infty \}
\\
E_4 &= \{ \phi_4 \in L^\infty(\mathcal D_{2R}) :
\nabla_y \phi_4  \in L^\infty(\mathcal D_{2R}), \
\|\phi_4\|_{***,\nu_4} <\infty \}
\end{align*}
and use the notation
\begin{equation*}
E_{\phi} = E_1\times E_2 \times E_3 \times E_4,\quad \Phi = ( \phi_1,\phi_2,\phi_3,\phi_4) \in E_{\phi}
\end{equation*}
$$
\|\Phi\|_{E_{\phi}} =
\|\phi_1\|_{*,\nu_1,a_1,\delta}
+\|\phi_2\|_{{\rm in},\nu_2,a_2-2}
+\|\phi_3\|_{**,\nu_3}
+\|\phi_4\|_{***,\nu_4}.
$$
We define the closed ball
\[
\mathcal B = \{  \Phi \in E_{\phi} : \| \Phi\|_{E_{\phi}} \leq 1 \} .
\]
For the outer problem \eqref{eqn-douter}, we shall solve $\psi$ in the space
\begin{equation*}
E_{\psi}=\left\{\psi\in L^{\infty}(\Omega\times(0,T)):\|\psi\|_{\sharp,\Theta,\gamma}<\infty\right\}.
\end{equation*}
For the incompressible Navier--Stokes equation \eqref{eqn-NS},
we shall solve the velocity field $v$ in the space
\begin{equation}\label{class-v}
E_v=\left\{v\in L^2(\Omega;\R^2): \nabla\cdot v=0,~\|v\|_{S,\nu-1,1}<M \epsilon_0\right\},
\end{equation}
where $\epsilon_0>0$ is the number in \eqref{LCF} which is fixed sufficiently small, and $M>0$ is some fixed number.

To introduce the space for the parameter function $p(t)$,
we recall the  integral operator $\mathcal B_0$  defined in \eqref{defB0-new} of the approximate form
\begin{align*}
\mathcal B_0[p ]
=   \int_{-T} ^{t-\la^2}    \frac{\dot p(s)}{t-s}ds\, + O\big( \|\dot p\|_\infty \big).
\end{align*}
For  $\Theta\in (0,1)$, $l\in \R$ and a continuous function $g:I\to \mathbb C$, we define the norm
\begin{align}
\nonumber
\|g\|_{\Theta,l} = \sup_{t\in [-T,T]} \, (T-t)^{-\Theta} |\log(T-t)|^{l} |g(t)| ,
\end{align}
and for  $\gamma \in (0,1)$, $m \in (0,\infty) $, $l \in \R$, we define the semi-norm
\begin{align}
\nonumber
[ g]_{\gamma,m,l} = \sup \, (T-t)^{-m}  |\log(T-t)|^{l} \frac{|g(t)-g(s)|}{(t-s)^\gamma} ,
\end{align}
where the supremum is taken over $s \leq t$ in $[-T,T]$  such that $t-s \leq \frac{1}{10}(T-t)$.

\medskip

The following result was proved in \cite[Section 8]{17HMF}.

\begin{prop}
\label{keyprop}
Let $\alpha ,  \gamma \in (0,\frac{1}{2})$, $l\in \R$, $C_1>1$.
If $\alpha_0\in(0,1]$, $\Theta \in (0,\alpha_0)$,
$m \in(0, \Theta - \gamma]$,
and  $a(t) :[0,T]\to \mathbb C$ satisfies
\begin{align}
\label{hypA00}
\left\{
\begin{aligned}
& \frac{1}{C_1} \leq | a(T) | \leq C_1 ,
\\
& T^\Theta |\log T|^{1+\sigma-l} \| a(\cdot) - a(T) \|_{\Theta,l-1}
+ [a]_{\gamma,m,l-1}
\leq C_1 ,
\end{aligned}
\right.
\end{align}
for some $\sigma>0$,
then for $T>0$ sufficiently small there exist two operators $\mathcal P $ and $\RR_0$ so that $p = \mathcal P[a]: [-T,T]\to \mathbb C$ satisfies
\begin{align*}
\mathcal B_0[p](t)
= a(t) + \RR_0[a](t) , \quad t \in [0,T]
\end{align*}
with
\begin{align}
\nonumber
& |\RR_0[a](t) | \leq  C
\Bigl( T^{\sigma}
+ T^\Theta  \frac{\log |\log T|}{|\log T|}  \| a(\cdot) - a(T) \|_{\Theta,l-1}
+ [a]_{\gamma,m,l-1} \Bigr)
\frac{(T-t)^{m+(1+\alpha ) \gamma}}{  |\log(T-t)|^{l}} ,
\end{align}
for some $\sigma>0$.
\end{prop}

Proposition~\ref{keyprop} gives an approximate inverse $\mathcal P$  of the operator $\mathcal B_0$, so that given $a(t)$ satisfying \eqref{hypA00},  $ p := \mathcal P  \left[ a \right] $
satisfies
\[
\mathcal B_0[ p ]   = a +\RR_0[ a] , \quad \text{in }~[0,T],
\]
for a small remainder $\RR_0[ a]$. Moreover, the proof of Proposition \ref{keyprop} in \cite{17HMF} gives the decomposition
\begin{align*}
\mathcal P[a] = p_{0,\kappa} + \mathcal P_1[a],
\end{align*}
with
\begin{align}
\nonumber
p_{0,\kappa}(t) =
\kappa |\log T|
\int_t^T \frac{1}{|\log(T-s)|^2}
\,ds
, \quad  t\leq T  ,
\end{align}
$\kappa = \kappa[a] \in \mathbb C$,  and the function $ p_1 = \mathcal P_1[a]$ has the estimate
\[
\| p_1 \|_{*,3-\sigma} \leq C |\log T|^{1-\sigma} \log^2(|\log T|) .
\]
Here the semi-norm $\| \ \|_{*,3-\sigma}$ is defined by
\begin{align}
\nonumber
\|g\|_{*,3-\sigma} = \sup_{t\in [-T,T]}  |\log(T-t)|^{3-\sigma} |\dot g(t)|,
\end{align}
and $\sigma \in (0,1)$.
This leads us to define the space
\begin{equation*}
 X_{p} := \{  p_1 \in C([-T,T;\mathbb C]) \cap C^1([-T,T;\mathbb C]) \ : \ p_1(T) = 0 , \ \|p_1\|_{*,3-\sigma}<\infty \} ,
\end{equation*}
where we represent $p$ by the pair $(\kappa,p_1)$ in the form $p = p_{0,\kappa} + p_1$.

We define the space for $\xi(t)$ as
\begin{equation*}
X_{\xi}=\left\{\xi\in C^1((0,T);\R^2):\dot\xi(T)=0,\|\xi\|_{X_{\xi}}<\infty\right\}
\end{equation*}
where
\begin{equation*}
\|\xi\|_{X_{\xi}}=\|\xi\|_{L^{\infty}(0,T)}+\sup_{t\in(0,T)} \la_*^{-\sigma}(t)|\dot\xi(t)|
\end{equation*}
for some $\sigma\in(0,1)$, and we define the spaces for $\alpha(t)$, $\beta(t)$ as follows
\begin{equation*}
X_{\alpha}=\left\{\xi\in C^1((0,T)):\alpha(T)=0,\|\alpha\|_{X_{\alpha}}<\infty\right\}
\end{equation*}
where
\begin{equation*}
\|\alpha\|_{X_{\alpha}}=\sup_{t\in(0,T)} \la_*^{-\delta_1}(t)|\alpha(t)|+\sup_{t\in(0,T)} \la_*^{1-\delta_1}(t)|\dot\alpha(t)|
\end{equation*}
and
\begin{equation*}
X_{\beta}=\left\{\beta\in C^1((0,T)):\beta(T)=0,\|\beta\|_{X_{\beta}}<\infty\right\}
\end{equation*}
where
\begin{equation*}
\|\beta\|_{X_{\beta}}=\sup_{t\in(0,T)} \la_*^{-\delta_2}(t)|\beta(t)|+\sup_{t\in(0,T)} \la_*^{1-\delta_2}(t)|\dot\beta(t)|.
\end{equation*}
 Here $\delta_1,\delta_2\in(0,1)$.

In conclusion, we will solve the inner--outer gluing system \eqref{eqn-NS}, \eqref{eqn-douter}, \eqref{eqphi1}, \eqref{eqphi2}, \eqref{eqphi3}, \eqref{eqphi4}, \eqref{redueqn-m0}, \eqref{redueqn-m1} and \eqref{redueqn-m-1} in the space
\begin{equation}\label{def-mX}
\mathcal X= E_v \times E_{\psi} \times E_{\phi} \times X_p\times X_{\xi}\times X_{\alpha}\times X_{\beta}
\end{equation}
by means of fixed point argument.

\medskip

\subsection{Estimates of the orientation field \texorpdfstring{$u$}{u}}\label{sec-d}

\medskip

The equation for the orientation field $u$ is close in spirit to the harmonic map heat flow \eqref{HMF}. To get the desired blow-up, we only need to show the drift term $v\cdot \nabla u$ is a small perturbation in the topology chosen above. Then the construction of the orientation field $u$ is a direct consequence of \cite{17HMF} with slight modifications.

\medskip

\noindent{{\bf Effect of the drift term $v\cdot \nabla u$ in the outer problem}}

\medskip

In the outer problem \eqref{eqn-douter}, it is direct to see that the main contribution in the drift term $v\cdot \nabla u$ comes from $v\cdot\nabla U$ since all the other terms are of smaller orders. We get that for some positive constant $\epsilon$,
\begin{equation}\label{est-vinouter}
\begin{aligned}
\left|(1-\eta_R)v\cdot\nabla u\right|=&~\bigg|(1-\eta_R)v\cdot \nabla U+(1-\eta_R) v\cdot \nabla\left( \Pi_{U^{\perp}}\left[\eta_R Q_{\omega,\alpha,\beta} \phi  +\Psi^* +\Phi^0 +\Phi^{\alpha} +\Phi^{\beta} \right]\right)\\
&~+(1-\eta_R) v\cdot \nabla\left( a\left(\Pi_{U^{\perp}}\left[\eta_R Q_{\omega,\alpha,\beta} \phi  +\Psi^* +\Phi^0 +\Phi^{\alpha} +\Phi^{\beta} \right]\right)U\right)\bigg|\\
\lesssim &~|(1-\eta_R)v\cdot\nabla U|\\
\lesssim &~\frac{\la_*^{\nu-1}(t)\|v\|_{S,\nu-1,1}}{1+\left|\frac{x-q}{\la_*(t)}\right|}\frac{\la_*(t)}{|x-q|^2+\la_*^2(t)}\chi_{ \{  |x-q| \geq  \lambda_*(t) R(t) \}}\\
\lesssim&~ T^{\epsilon} \varrho_2
\end{aligned}
\end{equation}
provided $\nu>m$ with $m\in(1/2,1)$ obtained in Lemma \ref{lemmaHeat6}, where $\varrho_2$ is the weight of the $\|\cdot\|_{**}$-norm (see \eqref{weights}) for the right hand side of the outer problem. Therefore, as long as $\nu$ is chosen sufficiently close to $1$, the influence of the drift term $v\cdot \nabla u$ in the outer problem is negligible, and it is indeed a perturbation compared to the rest terms already estimated in the harmonic map heat flow \cite[Section 6.6]{17HMF}.

\medskip

\noindent{{\bf Effect of the drift term $v\cdot \nabla u$ in the inner problem}}

\medskip

Since the inner problem is decomposed into different modes \eqref{eqphi1}--\eqref{eqphi4}, the drift term $v\cdot \nabla u$ will get coupled in each mode. We now analyze the projections of $v\cdot \nabla u$ on different modes. Recall that
$$v\cdot\nabla u=v\cdot \nabla [U+\varphi_{in}+\Pi_{U^{\perp}} \varphi_{out}+a(\Pi_{U^{\perp}}(\varphi_{in}+\varphi_{out}))U]$$
where
$$\varphi_{in}=\eta_R Q_{\omega,\alpha,\beta}(\phi_1+\phi_{2}+\phi_{3}+\phi_{4}),\quad \varphi_{out}=\Psi^*+\Phi^0+\Phi^{\alpha}+\Phi^{\beta}.$$
Notice that the leading term in $v\cdot\nabla u$ is $v\cdot\nabla U$.
Since $(v\cdot \nabla U,U)=0$, we have
$$\Pi_{U^{\perp}} (v\cdot \nabla U)=v\cdot \nabla U.$$
Denote $v=\begin{bmatrix} v_1\\ v_2\\ \end{bmatrix}$. We write $U$ in the polar coordinates
\begin{equation*}
\nabla U=\la^{-1}\begin{bmatrix}
\cos\theta w_{\rho} E_1-\frac{\sin\theta}{\rho}\sin w E_2\\
\sin\theta w_{\rho} E_1 +\frac{\cos\theta}{\rho}\sin w E_2\\
\end{bmatrix}.
\end{equation*}
Therefore, the projection of $v\cdot\nabla u$ on mode $k$ ($k\in\mathbb Z$) is of the following size
\begin{equation*}
\begin{aligned}
&\big|[\Pi_{U^\perp}(v\cdot\nabla u)]_k\big|\lesssim \big|[\Pi_{U^\perp}(v\cdot\nabla U)]_k\big|\lesssim\\
&\Bigg| \int_0^{2\pi}\left(v_1 \cos\theta \cos (k\theta) w_{\rho}+v_2\sin\theta \cos (k\theta) w_{\rho}+v_1\frac{\sin\theta \sin (k\theta)}{\rho}\sin w-v_2\frac{\cos\theta \sin (k\theta)}{\rho}\sin w\right)d\theta \\
&+i \int_0^{2\pi}\left( v_1 \cos\theta \sin (k\theta) w_{\rho}+v_2\sin\theta \sin (k\theta) w_{\rho}-v_1\frac{\sin\theta \cos (k\theta)}{\rho}\sin w+v_2\frac{\cos\theta \cos (k\theta)}{\rho}\sin w\right)d\theta \Bigg|\\
\end{aligned}
\end{equation*}
from which we obtain
\begin{equation}\label{proj-vininner}
\big|\la  [\Pi_{U^\perp}(v\cdot\nabla u)]_k\big|\leq \frac{M\epsilon_0\la_*^{\nu}}{1+|y|^3}
\end{equation}
where $M$ and $\epsilon_0$ are given in \eqref{class-v}.
Thus, it holds that
$$\|\la [\Pi_{U^\perp}(v\cdot\nabla u)]_k\|_{\nu,a}\leq  M\epsilon_0.$$
Since $\epsilon_0$ is a sufficiently small number, we find that the projection $[\Pi_{U^\perp}(v\cdot\nabla u)]_k$ can be regarded as a perturbation compared to the rest terms in the right hand sides of the inner problems \eqref{eqphi1}--\eqref{eqphi4}.

\medskip

In summary, the coupling of the drift term $v\cdot\nabla u$ in the inner and outer problems of the harmonic map heat flow is essentially negligible under the topology chosen above. Therefore, with slight modifications, the fixed point formulation for
$$\partial_t u+ v\cdot\nabla u=\Delta u+|\nabla u|^2u$$
can be carried out in a similar manner as in \cite{17HMF}.

For the outer problem \eqref{eqn-douter}, it was already estimated in \cite{17HMF} that in the space $\mathcal X$ defined in \eqref{def-mX}, it holds that for some $\epsilon>0$
\begin{equation*}
\begin{aligned}
\|\mathcal G[p,\xi,\alpha,\beta,\Psi^*,\phi,v]-(1-\eta_R)v\cdot \nabla u\|_{**}\lesssim&~ T^{\epsilon} (\|\Phi\|_{E_{\phi}}+\|\psi\|_{\sharp,\Theta,\gamma}+\|p\|_{X_{p}}+\|\xi\|_{X_{\xi}}\\
&~+\|\alpha\|_{X_{\alpha}}+\|\beta\|_{X_{\beta}}+1)
\end{aligned}
\end{equation*}
provided
\begin{equation}\label{choice-outer}
\left\{
\begin{aligned}
&0<\Theta<\min\left\{\gamma_*,\frac12-\gamma_*,\nu_1-1+\gamma_*(a_1-1),\nu_2-1+\gamma_*(a_2-1),\nu_3-1,\nu_4-1+\gamma_*\right\},\\
&\Theta<\min\left\{\nu_1-\delta\gamma_*(5-a_1)-\gamma_*,\nu_2-\gamma_*,\nu_3-3\gamma_*,\nu_4-\gamma_*\right\},\\
&\delta\ll 1.\\
\end{aligned}
\right.
\end{equation}
On the other hand, from \eqref{est-vinouter}, we find that
\begin{equation*}
\|(1-\eta_R)v\cdot \nabla u\|_{**}\lesssim T^{\epsilon} (\|v\|_{S,\nu-1,1}+\|\Phi\|_{E_{\phi}}+\|\psi\|_{\sharp,\Theta,\gamma}+\|p\|_{X_{p}}+\|\xi\|_{X_{\xi}}+\|\alpha\|_{X_{\alpha}}+\|\beta\|_{X_{\beta}}+1)
\end{equation*}
provided
\begin{equation}\label{choice-outer1}
\nu>\frac12.
\end{equation}
Therefore, we conclude the validity of the following proposition by Proposition \ref{prop3}.
\begin{prop}\label{prop-outerop}
Assume \eqref{choice-outer} and \eqref{choice-outer1} hold. If $T>0$ is sufficiently small, then there exists a solution $\psi=\Psi(v,\Phi,p,\xi,\alpha,\beta)$ to problem \eqref{eqn-douter} with
$$\|\Psi(v,\Phi,p,\xi,\alpha,\beta)\|_{\sharp,\Theta,\gamma}\lesssim T^{\epsilon} (\|v\|_{S,\nu-1,1}+\|\Phi\|_{E_{\phi}}+\|p\|_{X_{p}}+\|\xi\|_{X_{\xi}}+\|\alpha\|_{X_{\alpha}}+\|\beta\|_{X_{\beta}}+1),$$
for some $\epsilon>0$.
\end{prop}
We denote $\mathcal T_{\psi}$ by the operator which returns $\psi$ given in Proposition \ref{prop-outerop}.

For the inner problems \eqref{eqphi1}--\eqref{eqphi4}, our next step is to take $\Phi\in E_{\phi}$ and substitute  
$$\Psi^*(v,\Phi,p,\xi,\alpha,\beta) = Z^*+ \Psi(v,\Phi,p,\xi,\alpha,\beta)$$
into \eqref{eqn-dinner}. We can then write equations \eqref{eqphi1}--\eqref{eqphi4}
as the fixed point problem
\begin{align}
\label{ptofijo}
\Phi = \mathcal A (\Phi)
\end{align}
where
\begin{align*}
\mathcal A (\Phi) = ( \mathcal A_1(\Phi) ,  \mathcal A_2(\Phi) ,  \mathcal A_3(\Phi) , \mathcal A_4(\Phi)  ) , \quad \mathcal A : \bar{\mathcal B}_1\subset E \to E
\end{align*}
with
\begin{align*}
\mathcal A_1(\Phi) &=   \mathcal T_{1}  (
\mathcal H_1[v,\Psi^*(v,\Phi,p,\xi,\alpha,\beta),p,\xi, \alpha,\beta  ] )
\\
\mathcal A_2(\Phi) &=   \mathcal T_{2}  (
\mathcal H_2[v,\Psi^*(v,\Phi,p,\xi,\alpha,\beta),p,\xi, \alpha,\beta  ] )
\\
\mathcal A_3(\Phi) &=   \mathcal T_{3 }
\Bigl(
\mathcal H_3[v,\Psi^*(v,\Phi,p,\xi,\alpha,\beta),p,\xi, \alpha,\beta  ]+
\sum_{j=1}^2 c_{0j}^*[v,\Psi^*(v,\Phi,p,\xi,\alpha,\beta),p,\xi, \alpha,\beta  ] w_\rho^2 Z_{0,j}
\Bigr)
\\
\mathcal A_4(\Phi) &=   \mathcal T_{4 } \Bigl(\mathcal H_4[v,\Psi^*(v,\Phi,p,\xi,\alpha,\beta),p,\xi, \alpha,\beta  ]\Bigr)	  .
\end{align*}
Neglecting $\Pi_{U^{\perp}}(v\cdot\nabla u)$, the contraction for the inner problem was shown in \cite[Section 6.7]{17HMF} under the conditions
\begin{equation}\label{choice-inner}
\left\{
\begin{aligned}
&\nu_1<1\\
&\nu_2<1-\gamma_*(a_2-2)\\
&\nu_3<\min\left\{1+\Theta+2\gamma_*\gamma,\nu_1+\frac12 \delta\gamma_*(a_1-2)\right\}\\
&\nu_4<1\\
\end{aligned}
\right.
\end{equation}
On the other hand, from \eqref{proj-vininner}, we obtain
\begin{equation}\label{est-rhsproj}
\begin{aligned}
&\left\|\lambda  Q^{-1}_{\omega,\alpha,\beta} [\Pi_{U^{\perp}}(v\cdot\nabla u)]_0\right\|_{\nu_1,a_1}\leq M\epsilon_0 \la_*^{\nu-\nu_1}(t)\\
&\left\|\lambda  Q^{-1}_{\omega,\alpha,\beta} \left([\Pi_{U^{\perp}}(v\cdot\nabla u)]_1+[\Pi_{U^{\perp}}(v\cdot\nabla u)]_{\perp}\right)\right\|_{\nu_2,a_2}\leq M\epsilon_0 \la_*^{\nu-\nu_2}(t)\\
&\left\|\lambda  Q^{-1}_{\omega,\alpha,\beta} [\Pi_{U^{\perp}}(v\cdot\nabla u)]_{-1}\right\|_{\nu_4,a}\leq M\epsilon_0 \la_*^{\nu-\nu_4}(t)\\
\end{aligned}
\end{equation}
Recall that the parameter $\epsilon_0>0$ in \eqref{LCF} is fixed and sufficiently small. Therefore, by letting
\begin{equation}\label{choice-inner1}
\left\{
\begin{aligned}
&\nu=\nu_1=\nu_2=\nu_4\\
&1<a<2\\
\end{aligned}
\right.
\end{equation}
the smallness in \eqref{est-rhsproj} comes from $\epsilon_0\ll 1$. Applying the linear theory developed in Section \ref{sec-ltinner} for the inner problems \eqref{eqphi1}--\eqref{eqphi4}, we then conclude the following proposition.

\begin{prop}\label{prop-innerop}
Assume \eqref{choice-inner} and \eqref{choice-inner1} hold. If $T>0$ and $\epsilon_0>0$ are sufficiently small,   then the system of equations \eqref{ptofijo} for $\Phi=(\phi_1,\phi_2,\phi_3,\phi_4)$ has a solution $\Phi\in  E_{\phi}.$
\end{prop}

We denote by $\mathcal T_{p}$, $\mathcal T_{\xi}$, $\mathcal T_{\alpha}$ and $\mathcal T_{\beta}$ the operators which return the parameter functions $p(t)$, $\xi(t)$, $\alpha(t)$, $\beta(t)$, respectively. The argument for adjusting the parameter functions such that \eqref{redueqn-m0}--\eqref{redueqn-m-1} hold is essentially similar to that of \cite{17HMF}. Note that the influence of the coupling $v\cdot\nabla u$ is negligible as shown in Section \ref{sec-d}. Therefore, the leading orders for the parameter functions $p(t)$, $\xi(t)$, $\alpha(t)$, $\beta(t)$ are the same as in Section \ref{sec-redu}. The reduced problem \eqref{redueqn-m0} yields an integro-differential equation for $p(t)$ which can be solved by the same argument as in \cite{17HMF}, while the reduced problems \eqref{redueqn-m1}--\eqref{redueqn-m-1} give relatively simpler equations for $\xi(t)$, $\alpha(t)$, $\beta(t)$, which can be solved by the fixed point argument. We omit the details.

\medskip

\subsection{Estimates of the velocity field \texorpdfstring{$v$}{v}}\label{sec-v}

\medskip

To solve the incompressible Navier--Stokes equation \eqref{eqn-NS}, we need to analyze the coupled forcing term
$$\epsilon_0\na\cdot(\na u\odot\na u-1/2|\na u|^2\mathbb{I}_2).$$
Observe that the main contribution in the forcing comes from $U+\eta_R Q_{\omega,\alpha,\beta}(\phi_0+\phi_1+\phi_{-1}+\phi_{\perp}),$
where $\phi_0$, $\phi_1$, $\phi_{-1}$, $\phi_{\perp}$ are in mode $0$, $1$, $-1$ and higher modes, respectively. From the linear theory in Section \ref{sec-ltinner}, the dominant terms are $U$ and $\phi_0$. So we next need to evaluate
$$\na\cdot(\na U\odot\na U-1/2\,|\na U|^2\,\mathbb{I}_2)~\mbox{ and }~\na\cdot(\na U\odot\na \phi_0-1/2\,(\na U:\na\phi_0)\,\mathbb{I}_2),$$
where $\na U:\na\phi_0=\sum_{ij}\pd_iU_j\pd_i(\phi_0)_j$. Recall
\[U(y)=\begin{bmatrix}e^{i\theta}\sin w(\rho)\\\cos w(\rho)\end{bmatrix},~~E_1(y)=\begin{bmatrix}e^{i\theta}\cos w(\rho)\\ -\sin w(\rho)\\ \end{bmatrix},~~E_2(y)=\begin{bmatrix}i e^{i\theta}\\ 0\\ \end{bmatrix}\]
so that
\EQN{
\pd_\rho U=w_\rho E_1,&~~\pd_\theta U=\sin w E_2,\\
\pd_\rho E_1=-w_\rho U,&~~\pd_\theta E_1=\cos w E_2.
}
Note that
\[\na\cdot(\na U\odot\na U-1/2\,|\na U|^2\,\mathbb{I}_2)=\De U\cdot\na U=-|\na U|U\cdot\na U=0.\]
For $\na\cdot(\na U\odot\na \phi_0-1/2\,(\na U:\na\phi_0)\,\mathbb{I}_2)$, we express the forcing in the polar coordinates. Since $\phi_0=\varphi E_1$ where $\varphi_0=\varphi_0(\rho)$, the first component
\EQN{
&~\quad(\la^{-3}\na\cdot(\na U\odot\na\phi_0))_1\\
&=\na_y\cdot(\na_yU\odot\na_y\phi_0)_1\\
&=\pd_{y_1}\left(\cos^2\theta\,\pd_\rho\varphi_0\,w_\rho+\frac{\sin^2\theta}{\rho^2}\,\varphi_0\sin w\cos w\right)\\
&\quad+\pd_{y_2}\left(\sin\theta\cos\theta\,\pd_\rho\varphi_0\,w_\rho-\frac{\sin\theta\,\cos\theta}{\rho^2}\,\varphi_0\sin w\cos w\right).
}
Changing $\pd_{y_1}$ and $\pd_{y_2}$ into $\pd_\rho$ and $\pd_\theta$, we obtain
\EQN{
&~\quad(\la^{-3}\na\cdot(\na U\odot\na\phi_0))_1\\
&=\cos\theta\,\pd_\rho\left(\cos^2\theta\,\pd_\rho\varphi_0\,w_\rho+\frac{\sin^2\theta}{\rho^2}\,\varphi_0\sin w\cos w\right)\\
&\quad-\frac{\sin\theta}{\rho}\,\pd_\theta\left(\cos^2\theta\,\pd_\rho\varphi_0\,w_\rho+\frac{\sin^2\theta}{\rho^2}\,\varphi_0\sin w\cos w\right)\\
&\quad+\sin\theta\,\pd_\rho\left(\sin\theta\cos\theta\,\pd_\rho\varphi_0\,w_\rho-\frac{\sin\theta\cos\theta}{\rho^2}\,\varphi_0\sin w\cos w\right)\\
&\quad+\frac{\cos\theta}{\rho}\,\pd_\theta\left(\sin\theta\cos\theta\,\pd_\rho\varphi_0\,w_\rho-\frac{\sin\theta\cos\theta}{\rho^2}\,\varphi_0\sin w\cos w\right)\\
&=\cos\theta\left(\pd_\rho^2\varphi_0\,w_\rho+\pd_\rho\varphi_0\,w_{\rho\rho}+\frac1\rho\,\pd_\rho\varphi_0\,w_\rho-\frac1{\rho^3}\,\varphi_0\sin w\cos w\right)\\
&=\cos\theta\,\left[\pd_\rho\left(\pd_\rho\varphi_0\,w_\rho+\frac{\varphi_0\,w_\rho}\rho+\int\varphi_0w_\rho^2\right)\right].
}
A similar calculation implies that the second component
\EQN{
(\la&^{-3}\na\cdot(\na U\odot\na\phi_0))_2=\sin\theta\,\left[\pd_\rho\left(\pd_\rho\varphi_0\,w_\rho+\frac{\varphi_0\,w_\rho}\rho+\int\varphi_0w_\rho^2\right)\right].
}
So $\na\cdot(\na U\odot\na\phi_0)=\na\left[\la^3(\pd_\rho\varphi_0\,w_\rho+\frac{\varphi_0\,w_\rho}\rho+\int\varphi_0w_\rho^2)\right]$ is a potential. Moreover, it is obvious that $\na\cdot(|\na U|^2\,\mathbb{I}_2)$ is a potential. Therefore, $\na\cdot(\na U\odot\na \phi_0-1/2\,(\na U:\na\phi_0)\,\mathbb{I}_2)$ is a potential, which can be absorbed in the pressure $P$ in problem \eqref{eqn-NS}.

Therefore, the leading term in $\mathcal F$ defined by \eqref{def-mF} is
\begin{equation}\label{mainforcing}
|\na U\odot\na \phi_{-1}|\leq \frac{\la^{\nu_4-2}(t)}{1+|y|^3}\|\phi_{-1}\|_{***,\nu_4},
\end{equation}
from which we conclude that
\begin{equation*}
\|\epsilon_0\na U\odot\na \phi_{-1}\|_{S,\nu-2,a+1}\leq \epsilon_0,
\end{equation*}
where we have used \eqref{choice-inner1}.

On the other hand, as mentioned in Remark \ref{rmk-SS}, the nonlinear term $v\cdot\nabla v$ in \eqref{eqn-NS} is of smaller order compared to the forcing $\epsilon_0\nabla \cdot \mathcal F$ if we look for a solution $v$ in the function space $E_v$ defined in \eqref{class-v}. Indeed, since $v\in E_v$, we have
\begin{equation*}
|v\cdot\nabla v|\lesssim \frac{\la_*^{2\nu-2}(t)}{1+|y|^3}
\end{equation*}
so that
\begin{equation*}
\left\|v\cdot\nabla v\right\|_{S,\nu-2,a+1}\lesssim \la_*^{\nu}(t)\ll 1~\mbox{ as }~t\to T.
\end{equation*}
Thus, the incompressible Navier--Stokes equation \eqref{eqn-NS} can be regarded as a perturbed Stokes system
\begin{equation*}
\partial_t v+\nabla P=\Delta v -\epsilon_0\nabla\cdot \mathcal F_1[p,\xi,\alpha,\beta,\Psi^*,\phi,v]
\end{equation*}
with
\begin{equation*}
\mathcal F_1[p,\xi,\alpha,\beta,\Psi^*,\phi,v]=\mathcal F[p,\xi,\alpha,\beta,\Psi^*,\phi,v]+v\otimes v,
\end{equation*}
where we have used the fact that $v$ is divergence-free so that we can write $v\cdot \nabla v=\nabla\cdot(v\otimes v)$.
We denote $\mathcal T_v$ by the operator which returns the solution $v$, namely
$$\mathcal T_v: E_v \rightarrow E_v$$
$$v\mapsto \mathcal T_v(v).$$
By \eqref{mainforcing} and the linear theory for the Stokes system developed in Section \ref{sec-SS}, we obtain
\begin{equation}\label{est-vvv}
\|\mathcal T_v(v)\|_{S,\nu-1,1}\leq C\epsilon_0\big( \|v\|_{S,\nu-1,1}+\|\Phi\|_{E_{\phi}}+\|\psi\|_{\sharp,\Theta,\gamma}+\|p\|_{X_p}+\|\xi\|_{X_{\xi}}+\|\alpha\|_{X_{\alpha}}+\|\beta\|_{X_{\beta}}+1\big).
\end{equation}

\medskip

\subsection{Proof of Theorem 1.1}

\medskip

Consider the operator
\begin{equation}\label{def-finaloperator}
\mathcal T=(\mathcal A,\mathcal T_{\psi},\mathcal T_{v},\mathcal T_{p},\mathcal T_{\xi},\mathcal T_{\alpha},\mathcal T_{\beta})
\end{equation}
defined in Section \ref{sec-d} and Section \ref{sec-v}. To prove Theorem \ref{thm}, our strategy is to show that the operator $\mathcal T$ has a fixed point in $\mathcal X$ by the Schauder fixed point theorem. Here the function space $\mathcal X$ is defined in \eqref{def-mX}. The existence of a fixed point in the desired space $\mathcal X$ follows from a similar manner as in \cite{17HMF}.

By collecting Proposition \ref{keyprop}, Proposition \ref{prop-outerop}, Proposition \ref{prop-innerop} and \eqref{est-vvv}, we conclude that the operator maps $\mathcal X$ to itself. On the other hand, the compactness of the operator $\mathcal T$ can be proved by suitable variants of the estimates. Indeed, if we vary the parameters $\gamma_*$, $\Theta$, $\nu$, $a$, $\nu_1$, $a_1$, $\nu_2$, $a_2$, $\nu_3$, $\nu_4$, $\delta$ slightly such that all the restrictions in \eqref{choice-outer}, \eqref{choice-outer1}, \eqref{choice-inner} and \eqref{choice-inner1} are satisfied, then one can show that the operator $\mathcal T$ has a compact embedding in the sense that if a sequence is bounded in the new variant norms, then there exists a subsequence which converges in the original norms used in $\mathcal X$. Thus, the compactness follows directly from a standard diagonal argument by Arzel\`a--Ascoli's theorem. Therefore, the existence of the desired solution for the single bubble case $k=1$ follows from the Schauder fixed point theorem.

The general case of multiple-bubble blow-up is essentially identical. The ansatz is modified as follows: we look for solution $u$ of the form
\begin{equation*}
u(x,t)=\sum_{j=1}^k U_j+ \Pi_{U_j^{\perp}}\varphi_j+a(\Pi_{U_j^{\perp}} \varphi_j) U_j,
\end{equation*}
where
\begin{align*}
U_j=&~U_{\la_j(t),\xi_j(t),\omega_j(t),\alpha_j(t),\beta_j(t)},\quad  \varphi_j=\varphi_{in}^j+\varphi_{out}^j,
\\
\varphi_{in}^j=&~\eta_{R(t)}(y_j)Q_{\omega_j(t),\alpha_j(t),\beta_j(t)}\phi(y_j,t),\quad y_j=\frac{x-\xi_j(t)}{\la_j(t)},
\\
\varphi_{out}^j=&~\psi(x,t)+Z^*(x,t)+\Phi^0_j+\Phi^{\alpha}_j+\Phi^{\beta}_j.
\end{align*}
Here $\Phi^{0}_j$, $\Phi^{\alpha}_j$ and $\Phi^{\beta}_j$ are corrections defined in a similar way as in \eqref{def-corrections} with $\la$, $\xi$, $\omega$, $\alpha$, $\beta$ replaced by $\la_j$, $\xi_j$, $\omega_j$, $\alpha_j$, $\beta_j$.  We are then led to one outer problem and $k$ inner problems for $u$ together with one Navier--Stokes equation for $v$ with exactly analogous estimates. A string of fixed point problems can be solved in the same manner. We omit the details.\qed

\medskip


\section*{Acknowledgements}

F. Lin is partially supported by NSF grant 1501000.
J. Wei is partially supported by NSERC of Canada.
C. Wang is partially supported by NSF grant 1764417.
We are grateful to Professor Tai-Peng Tsai for useful discussions. Y. Zhou would like to thank Dr. Mingfeng Qiu for useful discussions.

\medskip






\end{document}